\documentclass{article}

\usepackage[english]{babel}
\usepackage[cp1250]{inputenc}

\usepackage{amsthm}
\usepackage{amsmath}
\usepackage{amssymb}
\usepackage{tikz}
\usetikzlibrary{matrix,positioning,shapes}

\newtheorem{theorem}{Theorem}[section]
\newtheorem{proposition}[theorem]{Proposition}
\newtheorem{lemma}[theorem]{Lemma}
\newtheorem{corollary}[theorem]{Corollary}
\newtheorem{definition}[theorem]{Definition}
\newtheorem{remark}[theorem]{Remark}

\author{Stanis{\l}aw Szawiel \\ Marek Zawadowski \\ \\ Institute of Mathematics, University of Warsaw \\
ul. Banacha 2, 00-913 Warsaw, Poland}
\title{The Web Monoid and Opetopic Sets}
\date{\today}

\begin{document}
\maketitle

\begin{abstract}
We develop a new definition of opetopic sets. There are two main technical ingredients. The first is the systematic use of fibrations, which are implicit in most of the approaches in the literature. Their explicit use leads to certain clarifications in the construction of opetopic sets and other constructions. The second is the ``web monoid'', which plays a role analogous to the ``operad for operads'' of Baez and Dolan, the ``multicategory of function replacement'' of Hermida, Makkai and Power. We demonstrate that the web monoid is closely related to the ``Baez-Dolan slice construction'' as defined by Kock, Joyal, Batanin and Mascari.
\end{abstract}

\tableofcontents

\section*{Introduction}
\addcontentsline{toc}{section}{Introduction}

Opetopic sets are notorious for being difficult to define and work with. In this work we will separate the difficult parts from the easy parts, and encapsulate each difficult part in some formal structure. The end result is a rather straightforward definition of the category of opetopic sets. The category of opetopes, unfortunately, remains beyond reasonable reach (generators and relations are no good).

To facilitate this separation and isolation we will use lax monoidal fibrations, introduced in \cite{Zawadowski} for this purpose. Typed operads originally used by Baez and Dolan in \cite{HDA3}, and most other structures used for the definition of opetopic sets (or just opetopes), for example the multicategories of \cite{HMP} and polynomial functors of \cite{KJBM}, naturally assemble into fibrations, which turn out to be fibrations of monoids of some monoidal fibration. Lax monoidal fibrations seem to provide a natural language for working with opetopic sets.

All three approaches cited above use the fact that a free monoid in a certain monoidal category can be made into a monoid in a different monoidal category. Thus Baez and Dolan consider the ``operad for operads'' as an operad whose universe is the universe of the free operad on the universe of the terminal symmetric operad. Hermida, Makkai and Power construct the ``multicategory of function replacement'' from a free multicategory. Perhaps the slickest construction is used in \cite{KJBM}. We review it in section \ref{comparisonsection}, where we compare our work to that of \cite{KJBM}.

Our main tool for constructing opetopic sets has the same flavor. It is called the ``web monoid''. It can be constructed in any fibration with appropriate extra structure. This structure is where all the difficult parts of opetopic sets are enclosed.

The first piece of structure are two strong monoidal structures, say $\otimes$ and $\odot$, subject to certain cocontinuity conditions. This reflects the fact the web monoid is a $\otimes$-monoid on the underlying object of a free $\odot$-monoid. Furthermore we require that ``$\otimes$ distributes over $\odot$'', a notion made precise by the concept of a distributivity structure. In a strict world it would be a collection of equalities

\begin{eqnarray*}
(A \otimes X) \odot (B \otimes X) & = & (A \odot B) \otimes X \\
I_{\odot} & = & I_{\odot} \otimes X,
\end{eqnarray*}
where $I_{\odot}$ is the unit of $\odot$. As usual, equalities must be replaced by natural isomorphisms to obtain a useful concept.

To obtain the web monoid we consider $\mathcal{F}_{\odot}(I_{\otimes})$, the free $\odot$-monoid on the $\otimes$-unit, decide that the unit of the new multiplication is $\eta : I_{\otimes} \rightarrow \mathcal{F}_{\odot}(I_{\otimes})$, and demand that ``the new multiplication commutes with the free multiplication''. This is made precise by using the distributivity structure. The resulting $\otimes$-monoid structure is then unique. This result is called the ``three tensors theorem'' (due to the role of coproducts -- a third monoidal structure), and is completely abstract. Its proof, apart from an educated guess for the new multiplication, is a quite messy inductive calculation. The web monoid encapsulates the difficulty of seeing a new structure on the free monoid.

To construct our opetopic sets we must make this abstract machinery work for the fibration of monoidal signatures, $\mathbf{Sig}_{ma}$. This is where the combinatorial problems of opetopic sets are identified and partitioned into small pieces. These pieces are the two monoidal operations $\otimes$ and $\odot$ and a distributivity structure between them. How small are we talking about? We will illustrate this with an example.

The amalgamation permutations for function replacement (the central operation in \cite{HMP}) are the same, in a precise sense, as the amalgamation permutations for the web monoid. Their definition in \cite{HMP} is rather abstract and indirect. We could give explicit recursive formulas for them. But we can say a lot more. We know their origin -- they arise because the free multiplication in $\mathcal{F}_{\odot}(I_{\otimes})$ has nonstandard amalgamation. This in turn arises from the associativity of $\odot$ -- it has nonstandard amalgamation, and it must -- to preserve the ``geometry'' of opetopic sets. This can proved using the separation principle \ref{separationprinciple}, a trick inspired by the constructions in \cite[part II]{HMP}. Finally, the construction of $\odot$ reflects our geometrical intuition. On a more mundane level $\odot$ is uniquely determined if we want corollary \ref{uisstrictmonoidal} to be true (it is needed to compare our work to \cite{KJBM}). Thus the complexity of function replacement is reduced to the associativity of a monoidal structure. That is two levels of recursion less. The construction of $\odot$ is still quite involved, but completely explicit and ultimately manageable.

The rest of the structure is comparatively simple, and merely allows $\odot$ to exist (and distribute over $\otimes$) -- we have not distributed the difficulties evenly among our structures. Once the basic combinatorial widget $\odot$ is in place, a completely formal theory takes over. We have seen that $\odot$ generates the nonstandard amalgamation in the web monoid, and this is the only nontrivial ingredient in our definition of opetopic sets. This can also be seen clearly in section \ref{comparisonsection}. In that section not once must we consider any combinatorial formulas -- and there are plenty such formulas to go around. We only need to know they are there, and are a consistent part of our formalism.

There is one more thing which can be seen in section \ref{comparisonsection}. It is the central role of monoidal signatures. All the fibrations used by us, and in the slice construction are closely related to them, but only $\mathbf{Sig}_{ma}$ ``sees everything'' properly. It sees the two kinds of inputs which give rise to the two structures $\otimes$ and $\odot$, it sees free monoids in the fibrations used in the slice construction, and even its action on these fibrations is very nearly the same as one of its monoidal structures $\otimes$. This probably sounds vague, and the reader will have to read section \ref{comparisonsection} to see what to make of these claims.

In appendix \ref{amalgamationisnecessaryexample} we answer a very natural question, which has somehow managed to escape consideration in the other approaches: In the end, can we get rid of the amalgamation permutations? This is a sensible question, since many monoids with nonstandard amalgamation are isomorphic to ones with standard amalgamation. The answer is no: the web monoid is not isomorphic to a monoid with standard amalgamation, even if we begin with standard amalgamation. Our example consists of pictures, and should easily adapt to all other approaches.

\section{Lax Monoidal Fibrations}
We will need the theory of lax monoidal fibrations from \cite{Zawadowski}, for which the reader is referred there. The theory of fibrations can be found in \cite{Streicher}. Morphisms which are called ``cartesian'' in \cite{Streicher} will be called ``prone'' in this paper.

A lax monoidal fibration can be briefly defined as a lax monoid in the $2$-category of fibrations over some category $\mathcal{B}$, \emph{fibered functors} (not morphisms of fibrations!), and fibered natural transformations. Monoids are to be understood as using the ordinary product in the underlying $1$-category. We will be concerning ourselves with strong monoidal fibrations, that is those in which $\alpha, \lambda$ and $\rho$ are isomorphisms. The conventions on the direction of these arrows are therefore a matter of convenience. We will specify them here, but unfortunately no choice is optimal for the entire paper. The conventions in \cite{Zawadowski}, however, are uniquely determined by the examples therein.

\paragraph{Conventions.} If $\mathcal{E} \rightarrow \mathcal{B}$ is a functor, then the preimage of $O \in \mathcal{B}$ will be denoted $\mathcal{E}/O$, to be read ``$\mathcal{E}$ over $O$''. If this functor is the codomain fibration $\mathcal{C}^{\cdot \rightarrow \cdot} \rightarrow \mathcal{C}$, then this notation agrees with taking slices of $\mathcal{C}$. Thus if $O \in \mathcal{C}$ then $\mathcal{C}/O$ means the ordinary slice, and we will never write $\mathcal{C}^{\cdot \rightarrow \cdot}/O$.

In a lax monoidal fibration the coherence isomorphisms will have the following directions:

\begin{displaymath}
\begin{array}{l}
\alpha_{A, B, C} : A \otimes (B \otimes C) \rightarrow (A \otimes B) \otimes C \\
\lambda_{A} : I \otimes A \rightarrow A \\
\rho_{A} : A \rightarrow A \otimes I
\end{array}
\end{displaymath}

A pullback (in $\mathbf{Cat}$) of a fibration is a fibration. The same is true for lax monoidal fibrations. The following lemma states that any algebraic structure is preserved by pullback. It is exceedingly useful.

\begin{lemma}\label{pullbackspreservealgebra}
Let $F: \mathcal{B}' \rightarrow \mathcal{B}$ be a functor. Then the operation $\mathcal{E} \mapsto F^{\ast} \mathcal{E}$, of pulling back along $F$ extends to a $2$-functor $F^{\ast}: \mathbf{Fib}(\mathcal{B}) \rightarrow \mathbf{Fib}(\mathcal{B}')$, which preserves finite products in the underlying $1$-category.
\end{lemma}
\begin{proof}
$F^{\ast}$ extends to fibered functors by the universal property of pullbacks. Preservation of composition and identities also follows directly from this universal property.

We extend $F^{\ast}$ to natural transformations by hand. If $f, g: \mathcal{E} \rightarrow \mathcal{F}$ are two functors over $\mathcal{B}$, and $\tau: f \rightarrow g$ is a fibered natural transformation, then:
\begin{itemize}
\item[-] $F^{\ast} \mathcal{E}$ is given by pairs $(b', e)$, $b' \in \mathcal{B}'$, $e \in \mathcal{E}$, which project to the same object in $\mathcal{B}$. Likewise for morphisms.
\item[-] $F^{\ast} f$ is given by $(b', e) \mapsto (b', f(e))$. Likewise for morphisms.
\item[-] We define $F^{\ast} \tau_{(b', e)} = (1_{b'}, \tau_{e})$ (which is a morphism in $F^{\ast} \mathcal{F}$).
\end{itemize}

It is then obvious that $F^{\ast} \tau$ is fibered, natural, and that $F^{\ast}$ preserves composition of natural transformations.

Products in the $1$-category underlying $\mathbf{Fib}(\mathcal{B})$ are computed as pullbacks in $\mathbf{Cat}$ over $\mathcal{B}$, since any pullback of a fibration is a fibration, and we are using fibered functors as morphisms, so our hom-sets coincide with the ones in $\mathbf{Cat}/ \mathcal{B}$. Preservation of binary products follows since both $F^{\ast} \mathcal{E} \times_{\mathcal{B}} \mathcal{F}$ and $F^{\ast} \mathcal{E} \times_{\mathcal{B}'} F^{\ast} \mathcal{F}$ compute the limit (in $\mathbf{Cat}$) of the diagram

\begin{center}
\begin{tikzpicture}
\matrix (m) [matrix of math nodes, column sep = .8cm, row sep = .8cm, text height = 1.5ex, text depth = 0.25ex]{
& \mathcal{E} && \mathcal{F} \\
\mathcal{B}' && \mathcal{B} \\ 
};

\path[->] (m-1-2) edge (m-2-3)
			 (m-1-4) edge (m-2-3)
			 (m-2-1) edge node[auto] {$F$} (m-2-3);
\end{tikzpicture}
\end{center}
Preservation of the terminal object is obvious.
\end{proof}

\begin{corollary}\label{pullbackspreservemonoidalfibrations}
The pullback of a lax monoidal fibration is naturally a lax monoidal fibration.
\end{corollary}
\begin{proof}
The $2$-functor from the above lemma provides the necessary structure. The coherence conditions hold, because composites of natural transformations are preserved (and equality also -- trivially so).
\end{proof}

\begin{remark}
According to the definitions in \cite{Zawadowski} the canonical projection $F^{\ast} \mathcal{E} \rightarrow \mathcal{E}$ is a strict monoidal functor over $F: \mathcal{B}' \rightarrow \mathcal{B}$.
\end{remark}

We can now see that the fibers of a lax monoidal fibration are lax monoidal categories, by considering the pullbacks along functors $1 \rightarrow B$, which give the fibers of a fibration. Similarly, morphisms over a given one can be multiplied, by considering pullbacks along functors from the category $\mathbf{2} = (\cdot \longrightarrow \cdot)$. It was also noted in \cite{Zawadowski} that the reindexing functors are lax monoidal. Their strongness is equivalent to $\otimes$ being a morphism of fibrations. This will never happen in our examples.

We will need a few facts about universal properties in fibrations.

\begin{lemma}[cf. lemma 5.5 in \cite{Zawadowski}]\label{freedominfibersisfreedom}
Consider $\mathcal{U}: \mathcal{M} \rightarrow \mathcal{E}$, a morphism of fibrations over $\mathcal{B}$. Let $X \in \mathcal{E}/O$. Then a vertical arrow $X \rightarrow \mathcal{U}(M)$ is universal from $X$ to $\mathcal{U}$ if and only if it is universal from $X$ to the restriction of $\mathcal{U}$ to $\mathcal{M}/O$.
\end{lemma}
\begin{proof}
Using prone morphisms we can reduce morphisms between fibers to morphisms in the fiber over $O$, where we assumed universality. The other implication is trivial.
\end{proof}

The fibration of diagrams of type $\mathcal{D} \in \mathbf{Cat}$ is the pullback of $\mathcal{E}^{\mathcal{D}} \rightarrow \mathcal{B}^{\mathcal{D}}$ along the constant diagram functor $\Delta : \mathcal{B} \rightarrow \mathcal{B}^{\mathcal{D}}$, see \cite{Streicher}. Then $\Delta_{\mathcal{E}}: \mathcal{E} \rightarrow \mathcal{E}^{\mathcal{D}}$ factors into a morphism over $\mathcal{B}$, which we will still call the constant diagram functor, followed by the canonical projection. A fibered colimit of an object $F$ of such a fibration is a vertical universal arrow from $F$ to the constant diagram functor $\Delta_{\mathcal{E}}$ (considered over $\mathcal{B}$), as usual.

\begin{corollary}\label{fiberedcolimits}
If the fibration $\mathcal{E} \rightarrow \mathcal{B}$ has a type of colimit (eg. coproducts, pushouts, filtered colimits) fiberwise, then it has the fibered version of this type of colimit.

If a fibration has a type of colimit fiberwise, then taking the colimit extends to a functor on the fibration of diagrams of the given type. 
\end{corollary}
\begin{proof}
The needed universal property follows immediately from lemma \ref{freedominfibersisfreedom}, since the constant diagram functor $\mathcal{E} \rightarrow \Delta^{\ast} \mathcal{E}^{\mathcal{D}}$ (again, considered over $\mathcal{B}$) preserves prone morphisms. The second statement is a formal consequence of the first.
\end{proof}

Note that the condition in this corollary refers only to fibers. It is therefore stable under pullback. Thus existence of fibered colimits is stable under pullback.

\paragraph{The fibered slice.} We will need one more general construction for lax monoidal fibrations. It is needed exclusively for section \ref{comparisonsection}.

Recall that if $\mathcal{C}$ is a monoidal category and $M \in Mon(\mathcal{C})$ is a monoid, then the slice category $\mathcal{C}/M$ is also naturally a monoidal category, with monoidal product of $A \rightarrow M$ and $B \rightarrow M$ defined by the composite $A \otimes B \rightarrow M \otimes M \xrightarrow{\mu} M$, where $\otimes$ is the product in $\mathcal{C}$ and $\mu$ is multiplication in $M$. The unit is the unit of the monoid $e : I \rightarrow M$, and there are obvious correct choices for the coherence isomorphisms. If $\mathcal{C}$ has pullbacks, these categories are fibers of the monoidal fibration $\mathcal{C} \downarrow \mathcal{U}$ over $Mon(\mathcal{C})$, where $\mathcal{U}: Mon(\mathcal{C}) \rightarrow \mathcal{C}$ is the forgetful functor.

This construction has a fibered analogue. Let $\mathcal{E}$ be a monoidal fibration over $\mathcal{B}$, and let $\mathcal{E}^{(\cdot \rightarrow \cdot)}$ be the fibration of diagrams of type $\cdot \rightarrow \cdot$ in $\mathcal{E}$. There is an obvious functor $(cod): \mathcal{E}^{(\cdot \rightarrow \cdot)} \rightarrow \mathcal{E}$ sending each arrow to its codomain. If $\mathcal{E}$ has pullbacks then this functor is a fibration -- the fibered analogue of a fundamental fibration. As before let $\mathcal{U}: Mon(\mathcal{E}) \rightarrow \mathcal{E}$ be the forgetful functor. The structure we are looking for is the pullback of $(cod)$ along $\mathcal{U}$,

\begin{center}
\begin{tikzpicture}
\matrix (m) [matrix of math nodes, row sep = 1cm, column sep = 1cm, text height = 1.5ex, text depth = .25ex] {
\mathcal{E} \downdownarrows \mathcal{U} & \mathcal{E}^{(\cdot \rightarrow \cdot)} \\
Mon(\mathcal{E}) & \mathcal{E} \\
};

\path[->] (m-1-1) edge node[auto] {} (m-1-2)
			 (m-1-1) edge node[auto, swap] {$\mathcal{U}^{\ast}(cod)$} (m-2-1)
			 (m-2-1) edge node[auto] {$\mathcal{U}$} (m-2-2)
			 (m-1-2) edge node[auto] {$(cod)$} (m-2-2);

\end{tikzpicture}
\end{center}
\noindent
which is a fibration over $Mon(\mathcal{E})$. Its fibers are precisely all the categories of the form $\mathcal{E}_{O}/M$, where $O \in \mathcal{B}$, $\mathcal{E}_{O}$ is the fiber of $\mathcal{E}$ over $O$, and $M$ is a monoid in $\mathcal{E}_{O}$. The above discussion gives us a monoidal structure on $\mathcal{E} \downdownarrows \mathcal{U}$. Concretely the  monoidal product of $A \rightarrow M$ and $B \rightarrow M$ is given by $A \otimes B \rightarrow M \otimes M \xrightarrow{\mu} M$ (as before) and the unit functor is $I(M) = I(O) \xrightarrow{e} M$, the unit of multiplication in $M$.

\section{The Three Tensors Theorem}

In this section we will construct the monoid which will do the heavy lifting in our definition of opetopic sets. It was first discovered by the authors in the context of monoidal signatures, but the abstract construction given here has several advantages. The most obvious one is generality and conceptual clarity. But the most important one is simplicity -- the original construction consisted almost entirely of checking whether one page-long term is equal to another. It was quite unreadable.

\subsection{Free Monoids in Monoidal Fibrations}
The main theorem asserts the existence of a certain extra structure on a free monoid. Its construction will use an explicit construction of this free monoid, which will be given here. The basic ideas behind this construction seem have been first stated explicitly in \cite{Adamek}. A very general account of such constructions has been given in \cite{Kelly}. We will follow the very brief and readable \cite[Appendix B]{Baues-Jibladze-Tonk}, and refer the reader there for all the calculations omitted here. The context there is a single monoidal category, but the calculations adapt to monoidal fibrations verbatim.

Let $\mathcal{E}$ be a strong monoidal fibration over $\mathcal{B}$, that is we assume $\alpha, \lambda$ and $\rho$ to be isomorphisms. We wish to construct a fibered left adjoint to the forgetful functor $\mathcal{U} : Mon(\mathcal{E}) \rightarrow \mathcal{E}$. We assume the following:
\begin{itemize}
\item [a)] $\mathcal{E}$ has fiberwise finite coproducts\footnote{Binary coproducts would suffice, but this would ruin the name of our main theorem.} and filtered colimits.
\item [b)] The monoidal product $\otimes$ preserves fibered filtered colimits in both variables, and fibered binary coproducts in the left variable.
\end{itemize}

The condition $a)$ is stable under pullback, and gives us fibered filtered colimits, by corollary \ref{fiberedcolimits}. 

Let $X \in \mathcal{E}/O$. We define 

\begin{displaymath}
\begin{array}{l}
X_{0} = I_{O} \\
X_{n+1} = I_{O} \sqcup (X \otimes X_{n}),
\end{array}
\end{displaymath}
where $I_{O}$ is the unit of $\otimes$ in the fiber over $O$ and $\sqcup$ is the coproduct. We have arrows 

\begin{displaymath}
\begin{array}{l}
i_{n}: X_{n} \rightarrow X_{n+1} \\
i_{0}: I_{O} \rightarrow I_{O} \sqcup X \textnormal{ is the coprojection} \\
i_{n+1} = 1 \sqcup (1 \otimes i_{n}).
\end{array}
\end{displaymath}
We define $X_{\infty}$, the universe of the free monoid on $X$, as the colimit of the $X_{i}$:

\begin{displaymath}
X_{\infty} = \varinjlim(X_{0} \rightarrow X_{1} \rightarrow X_{2} \rightarrow X_{3} \rightarrow \cdots)
\end{displaymath}
To define multiplication we define the morphisms $\mu_{n,m}: X_{n} \otimes X_{m} \rightarrow X_{n+m}$:

\begin{displaymath}
\mu_{0,m} = \lambda_{X_{m}} : I_{O} \otimes X_{m} \rightarrow X_{m}
\end{displaymath}
and for $n \geq 1$ we have 

\begin{displaymath}
X_{n} \otimes X_{m} \simeq (I_{O} \sqcup (X \otimes X_{n-1})) \otimes X_{m} \simeq X_{m} \sqcup (X \otimes X_{n-1}) \otimes X_{m},
\end{displaymath}
and define

\begin{displaymath}
\mu_{n, m} = (i_{m, n+m}, j_{n+m}(1 \otimes \mu_{n-1, m})) : X_{m} \sqcup X \otimes X_{n-1} \otimes X_{m} \rightarrow X_{n+m},
\end{displaymath}
where $i_{m, n+m}: X_{m} \rightarrow X_{n+m}$ is the inclusion (the composite of the appropriate $i_{k}$), and $j_{k}: X \otimes X_{k-1} \rightarrow X_{k} \simeq I \sqcup X \otimes X_{k-1}$ is the coprojection.

By the fact that $\otimes$ preserves filtered colimits, and the (easily checked) compatibility of the $\mu_{n, m}$ we may pass to the colimit $\mu: X_{\infty} \otimes X_{\infty} \rightarrow X_{\infty}$ of the maps $i_{n+m, \infty} \circ \mu_{n, m}: X_{n} \otimes X_{m} \rightarrow X_{\infty}$, where $i_{n+m, \infty}: X_{n+m} \rightarrow X_{\infty}$ is the canonical map to the colimit. We also have the unit of our monoid $\eta: I = X_{0} \rightarrow X_{\infty}$, given again by the canonical map to the colimit.

This construction is functorial in $X$. Consider a morphism $f: X \rightarrow Y$ over $u: O \rightarrow Q$ in $\mathcal{B}$. We set
\begin{displaymath}
\begin{array}{l}
f_{0} = I_{u}: X_{0} = I_{O} \rightarrow I_{Q} = Y_{0} \\
f_{n+1} = I_{u} \sqcup f \otimes_{u} f_{n-1}.
\end{array}
\end{displaymath}
Again, the (obvious) compatibility implies the existence of a morphism $f_{\infty} : X_{\infty} \rightarrow Y_{\infty}$ (we define $\sqcup$ and $f_{\infty}$ using remark \ref{fiberedcolimits}), and it can be checked that it is a monoid homomorphism over $u$, with respect to $\mu$ and $\eta$.

\begin{theorem}\label{freemonoidbjt}
If $\mathcal{E}$ has fiberwise finite coproducts and $\otimes$ preserves fibered filtered colimits in both variables and binary coproducts in the left variable, then the free monoid functor is $X \mapsto \mathcal{F}(X) = (X_{\infty}, \mu, \eta)$ on objects, and $f \mapsto f_{\infty}$ on morphisms.
\end{theorem}
\begin{proof}
All the calculations in \cite[Appendix B]{Baues-Jibladze-Tonk} clearly apply in each fiber, and $\eta$ is natural in the entire fibration. The universality of $\eta$ in the entire fibration follows from lemma \ref{freedominfibersisfreedom}, since the forgetful functor from monoids is always a morphism of fibrations.
\end{proof}

We will require some additional facts about the above construction. They were discovered in the course of the proof of the main theorem, but a very similar phenomenon was used in \cite[part 2]{HMP} under the name ``unique readability''. If the structures under consideration are multicategories or operads, then the free monoids consist of trees\footnote{With additional structure of course. Note also, that vertices of these trees represent operations and leaves represent inputs, and these are different parts of the structure -- we are \emph{not} dealing with ordinary graphs!} or terms. We will now see that we can ``identify'' the first vertex in these trees or function symbol in these terms.

\begin{proposition}\label{canonicalsection}
Under the assumptions of theorem \ref{freemonoidbjt} the multiplication in the free monoid $\mu : X_{\infty} \otimes X_{\infty} \rightarrow X_{\infty}$ has a vertical section $\hat{s} : X_{\infty} \rightarrow X_{\infty} \otimes X_{\infty}$, which factors as $X_{\infty} \xrightarrow{s} X_{1} \otimes X_{\infty} \xrightarrow{i \otimes 1} X_{\infty} \otimes X_{\infty}$, where $i: X_{1} \rightarrow X_{\infty}$ is the canonical map. 
\end{proposition}

In fact the components of the map $s$ (see the proof) will be more important than either $s$ or $\hat{s}$, which are only necessary for the application of the bootstrap lemma \ref{bootstraplemma}.

\begin{proof}
We write $\mu_{1, \infty} : X_{1} \otimes X_{\infty} \rightarrow X_{\infty}$ for the colimit of $\mu_{1, m} : X_{1} \otimes X_{m} \rightarrow X_{m+1}$, from the construction above, with respect to $m$. Hence $\mu_{1, \infty} \circ 1 \otimes i_{m, \infty} = i_{m+1, \infty} \circ \mu_{1,m}$, where $i_{k, \infty}$ is the canonical map $X_{k} \rightarrow X_{\infty}$. Also $\mu_{1, \infty} = \mu \circ i_{1, \infty} \otimes 1$, as is easily seen by composing both sides on the right with $1 \otimes i_{m, \infty}$. We will construct $s: X_{\infty} \rightarrow X_{1} \otimes X_{\infty}$ such that $\mu_{1, \infty} \circ s = 1_{X_{\infty}}$, and define $\hat{s}$ via the commutative diagram

\begin{center}
\begin{tikzpicture}
\matrix (m) [matrix of math nodes, column sep = 1.5cm, row sep = 1.5cm, text height =1.5ex, text depth =0.25ex]{
X_{\infty} & X_{1} \otimes X_{\infty} & X_{\infty} \\
& X_{\infty} \otimes X_{\infty} \\
};

\path[->] (m-1-1) edge node[auto] {$s$} (m-1-2)
			 (m-1-2) edge node[auto] {$\mu_{1, \infty}$} (m-1-3)
			 (m-1-1) edge node[auto, swap] {$\hat{s}$} (m-2-2)
			 (m-2-2) edge node[auto, swap] {$\mu$} (m-1-3)
			 (m-1-2) edge node (center) {} node[above right = .1 of center] {$i_{1, \infty} \otimes 1$} (m-2-2);

\end{tikzpicture}
\end{center}
Since $\otimes$ preserves filtered colimits, it suffices to construct a compatible family of maps $s_{m}: X_{m} \rightarrow X_{1} \otimes X_{m-1}$, for $m > 0$, such that $\mu_{1,m-1} \circ s_{m} = 1_{X_{m}}$. We have

\begin{displaymath}
X_{n} = I \sqcup X \otimes X_{n-1}
\end{displaymath}
\begin{displaymath}
X_{1} \otimes X_{n-1} = (I \sqcup X \otimes I) \otimes X_{n-1} \simeq X_{n-1} \sqcup X \otimes X_{n-1},
\end{displaymath}
and define

\begin{displaymath}
s_{n} =  I \sqcup X \otimes X_{n-1} \xrightarrow{i_{0, n-1} \sqcup 1_{X \otimes X_{n-1}}} X_{n-1} \sqcup X \otimes X_{n-1}.
\end{displaymath}
In these terms $\mu_{1, n-1}$ is easily found to be

\begin{displaymath}
\mu_{1,n-1} = (i_{n-1, n}, j_{n}(1_{X} \otimes 1_{X_{n-1}})) = (i_{n-1, n}, j_{n}).
\end{displaymath}
We can now calculate $\mu_{1, n-1} \circ s_{n}$:

\begin{displaymath}
(i_{n-1, n}, j_{n}) \circ (i_{0, n-1} \sqcup 1_{X \otimes X_{n-1}}) = (i_{0, n}, j_{n}),
\end{displaymath}
which is the identity $I \sqcup X \otimes X_{n-1} \rightarrow X_{n}$.
\noindent
The compatibility condition for $s_{n}$ is implied by the stronger condition

\begin{displaymath}
1 \otimes i_{n-1} \circ s_{n} = s_{n+1} \circ i_{n}.
\end{displaymath}
Expanding the definitions, it asserts the commutativity of the square

\begin{center}
\begin{tikzpicture}
\matrix (m) [matrix of math nodes, column sep = 2cm, row sep = 1cm, text height = 1.5ex, text depth = 0.25ex] {
I \sqcup X \otimes X_{n} & X_{n} \sqcup X \otimes X_{n} \\
I \sqcup X \otimes X_{n-1} &  X_{n-1} \sqcup X \otimes X_{n-1}\\
};

\path[->] (m-1-1) edge node[auto] {$i_{0, n} \sqcup 1$} (m-1-2)
			 (m-2-1) edge node[auto] {$1 \sqcup 1 \otimes i_{n-1}$} (m-1-1)
			 (m-2-1) edge node[auto] {$i_{0, n-1} \sqcup 1$} (m-2-2)
			 (m-2-2) edge node[auto, swap] {$i_{n-1} \sqcup 1 \otimes i_{n-1}$} (m-1-2);
\end{tikzpicture}
\end{center}
which is obvious. We may therefore pass to the colimit, and conclude the proof.
\end{proof}

We will call the maps $s, \hat{s}$, constructed above, the \emph{canonical sections} of $\mu$, or \emph{unique readability morphisms}. The following technical lemma is needed in the proof of the main theorem. It asserts a kind of coherence of $s$ with respect to multiplication.

\begin{lemma}[Coherence lemma]\label{coherencelemma}
The following diagram commutes (for $n > 0$)

\begin{center}
\begin{tikzpicture}
\matrix (m) [matrix of math nodes, column sep = 1cm, row sep = 1cm, text height = 1.5ex, text depth = .25ex]{
X_{n} \otimes X_{m} & X_{n+m} & X_{1} \otimes X_{n+m-1} \\
(X_{1} \otimes X_{n-1}) \otimes X_{m} && X_{1} \otimes (X_{n-1} \otimes X_{m}) \\
};

\path[->] (m-1-1) edge node [auto] {$\mu_{n, m}$} (m-1-2)
			 (m-1-2) edge node [auto] {$s_{n+m}$} (m-1-3)
			 (m-1-1) edge node [auto, swap] {$s_{n} \otimes 1$} (m-2-1)
			 (m-2-3) edge node [auto, swap] {$1 \otimes \mu_{n-1, m}$} (m-1-3)
			 (m-2-1) edge node[auto] {$\alpha^{-1}$} (m-2-3);
\end{tikzpicture}
\end{center}

\end{lemma}
\begin{proof}
We use the coherence theorem to ignore the coherence isomorphisms. Since $X_{n} \otimes X_{m}$ is the coproduct $I \otimes X_{m} \sqcup X \otimes X_{n-1} \otimes X_{m}$ it suffices to check the commutativity on each factor. Since any $s_{k}$ is the identity on the second factor it is easy to see that both second factors are $j \circ (1 \otimes \mu_{n-1,m})$, where $j$ is the coprojection as in the construction of $\mu$.

The second factor can be calculated as follows. The lower way is straightforward. It is $\mu_{n-1, m} \circ i_{0, n-1} \otimes 1$, which is (by the unit laws for $\mu$) $i_{m, n+m-1} : W_{m} \rightarrow W_{n+m-1}$. The upper way unfortunately mixes the components, so we must unwind one more level of definition. The relevant component of $\mu_{n,m}$ is $i_{m, n+m}$, which is $1 \sqcup 1 \otimes i_{m-1, n+m-1}$. Composing it with $s_{n+m} = (i_{0,n+m-1}, 1_{X \otimes X_{n+m-1}})$ yields $i_{0, n+m-1} \sqcup 1 \otimes i_{m-1, n+m-1} = i_{m, n+m-1}$, as required.

\end{proof}

\begin{remark}
From now on we will occasionally abuse notation and write $i_{k}$ for any of the maps $i_{k,l}$. The codomain will always be clear form context.
\end{remark}

\subsection{Distributivity of Monoidal Structures}
For any category $\mathcal{C}$ the category of endofunctors $End(\mathcal{C})$ is strict monoidal under composition of functors. The monoidal structure is composition in diagrammatic order (that is $(x) f \circ g$ means ,,first apply $f$ to $x$, then $g$ to (x) f'', but only if $f$ and $g$ are objects of $End(\mathcal{C})$). If in addition $\mathcal{C}$ was itself monoidal, we obtain functors $\mathcal{C} \rightarrow End(\mathcal{C})$ which send each $X \in \mathcal{C}$ to either $X \otimes (-)$ or $(-) \otimes X$. We will always be interested in the latter functor, which we will denote by $R$. Interestingly these functors are always monoidal in a natural way. Namely we have

\begin{eqnarray*}
(A) (-) \otimes (X \otimes Y) & = & A \otimes (X \otimes Y) \\
(A) (-) \otimes (X) \circ (-) \otimes Y & = & (A \otimes X) \otimes Y,
\end{eqnarray*}
and a natural isomorphism between these two is given by $\alpha_{A, X, Y}^{-1}$. The unit isomorphism is given by the appropriate components of $\rho$. The coherence diagrams for this monoidal functor are the defining coherence diagrams of a monoidal category (as in \cite{MacLane}), with some morphisms replaced by their inverses.

If $\odot$ is yet another monoidal structure on $\mathcal{C}$, then we can also define a monoidal category $End_{\odot}(\mathcal{C})$ of strong $\odot$-monoidal endomorphisms of $\mathcal{C}$, as follows. The identity functor $1_{\mathcal{C}}$ has an obvious monoidal structure, and will serve as a unit. The monoidal structure is composition of monoidal functors. It is easy to see that the horizontal composite of monoidal transformations is again monoidal, and so we can take as arrows the monoidal natural transformations. This category is of course still strict monoidal.

There is an obvious strict monoidal functor $U : End_{\odot}(\mathcal{C}) \rightarrow End(\mathcal{C})$, which forgets the additional data.

\paragraph{Definition.} Let $\mathcal{C}$ be a category, and suppose we are given  two strong monoidal structures on $\mathcal{C}$, denoted $(\odot, I_{\odot}, \alpha^{\odot}, \lambda^{\odot}, \rho^{\odot})$ and $(\otimes, I_{\otimes}, \alpha^{\otimes}, \lambda^{\otimes}, \rho^{\otimes})$. As above, let $R$ denote the functor $X \mapsto (-) \otimes X$. A distributivity structure of $\otimes$ over $\odot$ is given by a lift of $R$ to $End_{\odot}(\mathcal{C})$ along $U$, as a monoidal functor:

\begin{center}
\begin{tikzpicture}
\matrix (m) [matrix of math nodes, column sep = 2cm, row sep = 1cm, text height = 1.5ex, text depth = 0.25ex] {
& End_{\odot}(\mathcal{C}) \\
\mathcal{C} & End(\mathcal{C}) \\
};

\path[->] (m-2-1) edge node[auto, swap] {$R$} (m-2-2)
			 (m-1-2) edge node[auto] {$U$} (m-2-2);
\path[->, dashed] (m-2-1) edge node[auto] {$\tilde{R}$} (m-1-2); 

\end{tikzpicture}
\end{center}
which means that we require $R = U \circ \tilde{R}$ as monoidal functors.

We can unravel this definition and state it explicitly as extra data and properties for $\mathcal{C}$. First, every functor $R(X) = (-) \otimes X$ becomes $\odot$-monoidal. This gives us isomorphisms

\begin{displaymath}
\varphi_{A, B, X} : (A \otimes X) \odot (B \otimes X) \rightarrow (A \odot B) \otimes X 
\end{displaymath}
\begin{displaymath}
\psi_{X} : I_{\odot} \rightarrow I_{\odot} \otimes X
\end{displaymath}
which make $R(X)$ into a $\odot$-monoidal functor (which is $\tilde{R}(X)$). Of course for every morphism $f$ in $\mathcal{C}$, the natural transformation $R(f) = (-) \otimes f$ is required to be $\odot$-monoidal (giving $\tilde{R}(f)$). Second, since we require equality of $R$ and $U \circ \tilde{R}$ as monoidal functors, we see that the isomorphisms $(\alpha^{\otimes})^{-1} : \tilde{R}(X) \circ \tilde{R}(Y) \rightarrow \tilde{R}(X \otimes Y)$ and $\rho^{\otimes}: 1_{\mathcal{C}} \rightarrow \tilde{R}(I_{\otimes})$ giving the monoidal structure of $\tilde{R}$ become $\odot$-monoidal natural transformations. These properties are written as diagrams in appendix \ref{appendixa}. Conversely, natural transformations\footnote{Note that at this point it is not clear that $\varphi$ and $\psi$ are natural in $X$. This is demonstrated in the appendix.} $\varphi_{A, B, X}$ and $\psi_{X}$ subject to the coherence diagrams in the appendix determine a unique distributivity structure.

\begin{remark}
We will never consider more than one distributivity structure at a time, so we will abuse language and say that ``$\otimes$ distributes over $\odot$'', as if this were a property, and keep all the structure implicit. We will never change the notation for the elements of a distributivity structure introduced above.
\end{remark}

\begin{remark}
There are as many variations of this definition as there are versions of $End_{\odot}(\mathcal{C})$. We could use lax functors, opmonoidal ones, or lift the functor $L(X) = X \otimes (-)$. The choice is dictated by the application. For example, the main theorem is still true if lift $R$ to \emph{left-unital} monoidal functors, that is those which preserve only $\lambda^{\odot}$ but not necessarily $\rho^{\odot}$.
\end{remark}

\paragraph{Fibered distributivity.} There is no difficulty in stating the fibered equivalent of this definition -- simply replace endofunctor categories with exponential fibrations. We will be working with bifibrations, so the theory of \cite[section 4]{Zawadowski} can be applied. For this reason we limit the definition of a distributivity structure to the case of bifibrations. Since none of our monoidal structures will be morphisms of fibrations, we must use exponential fibrations computed in the category of fibrations and fibered functors, not morphisms of fibrations. This does not change the diagrammatic form of the definition, but only the fact that $\varphi$ and $\psi$ become fibered natural.

Observe that a distributivity structure on a fibration restricts to a distributivity structure on each fiber.

\paragraph{Example.} Let $(\mathcal{C}, \otimes)$ be a monoidal category with finite coproducts, and suppose that $\otimes$ preserves them in the left variable. Then the natural maps
\begin{eqnarray*}
A\otimes X \sqcup B \otimes X & \rightarrow & (A \sqcup B) \otimes X \\
0 & \rightarrow & 0 \otimes X
\end{eqnarray*}
are isomorphisms which define a distributivity structure of $\otimes$ over $\sqcup$. All the conditions are satisfied because of universality.

The preservation of products also defines a distributivity structure, but the directions of the natural arrows are opmonoidal rather than monoidal.

\subsection{The Main Theorem}
We say that a category (or fibration) $\mathcal{C}$ \emph{admits the free monoid construction for $\otimes$} if the assumptions of theorem \ref{freemonoidbjt} are true for $(\mathcal{C}, \otimes)$. All the categories and fibrations under consideration will have three monoidal structures -- two arbitrary ones, and the coproduct.

\begin{theorem}[The Three Tensors Theorem]\label{threetensorstheorem}
If $\mathcal{C}$ admits the free monoid constrution for $\odot$ and $\otimes$, and $\otimes$ distributes over $\odot$, then there is a unique $\otimes$-monoid structure on $\mathcal{F}_{\odot}(I_{\otimes})$, the free $\odot$-monoid on the $\otimes$-unit, such that the unit of the adjunction $\mathcal{F}_{\odot} \dashv \mathcal{U}_{\odot}$, $\eta_{I_{\otimes}} : I_{\otimes} \rightarrow \mathcal{F}_{\odot}(I_{\otimes})$ is the unit of the multiplication $\nu : \mathcal{F}_{\odot}(I_{\otimes}) \otimes \mathcal{F}_{\odot}(I_{\otimes}) \rightarrow \mathcal{F}_{\odot}(I_{\otimes})$, which in turn makes the following \emph{main diagram} commute (we abbreviate $\mathcal{F}_{\odot}(I_{\otimes})$ to $\mathcal{W}$):

\begin{center}
\begin{tikzpicture}
\matrix (m) [matrix of math nodes, column sep = 2cm, row sep = 1cm, text height = 1.5ex, text depth = .25ex]{
(\mathcal{W} \otimes \mathcal{W}) \odot (\mathcal{W} \otimes \mathcal{W}) & \mathcal{W} \odot \mathcal{W} \\
(\mathcal{W} \odot \mathcal{W}) \otimes \mathcal{W} & & \\
\mathcal{W} \otimes \mathcal{W} & \mathcal{W}\\
};

\path[->] (m-1-1) edge node[auto] {$\nu \odot \nu$} (m-1-2)
			(m-1-1) edge node[auto, swap] {$\varphi_{\mathcal{W},\mathcal{W},\mathcal{W}}$} (m-2-1)
			(m-2-1) edge node[auto, swap] {$\mu \otimes 1_{\mathcal{W}}$} (m-3-1)
			(m-1-2) edge node[auto] {$\mu$} (m-3-2)
			(m-3-1) edge node[auto] {$\nu$} (m-3-2);
\end{tikzpicture}
\end{center}
In the above diagram $\mu$ is the free multiplication in $\mathcal{F}_{\odot}(I_{\otimes})$.
\end{theorem}

The resulting monoid $(\mathcal{W}, \nu, \eta)$ is called the \emph{web monoid}, because in our applications its elements look like webs. The proof will demonstrate slightly more -- the multiplication is determined by the unit conditions and the main diagram, and its associativity follows from the construction. Obviously $\mathcal{F}_{\odot}(I_{\otimes})$ exists by the assumptions of the theorem, and is given by the construction of theorem \ref{freemonoidbjt}.

The main diagram states that $\nu$ and $\mu$ commute. We need the distributivity structure to explain how we can apply $\mu$ and $\nu$ to a single object in both orders (for more on this point see the next to last paragraph in section \ref{comparisonsection}). This property is analogous to \cite[part II, lemma 4]{HMP}, which states, in a restricted case, that multiplication in the free monoid commutes with function replacement.

We will actually need the fibered version of the above theorem, which asserts that in a fibered context the formation of $\mathcal{W}$ can be turned into a functor.

\begin{theorem}[Fibered Three Tensors Theorem]\label{fiberedttt}
If in the assumptions of theorem \ref{threetensorstheorem} each concept is replaced with its fibered analogue, then the conclusion is the existence of a unique $\otimes$-monoid structure on the functor $\mathcal{F}_{\odot}(I_{\otimes}(-))$ (of the base category) whose unit is the unit of the fibered free $\odot$-monoid adjunction and whose multiplication makes the main diagram commute.
\end{theorem}

The free monoid construction is, by corollary \ref{fiberedcolimits}, stable under pullback. Likewise, a fibered distributivity structure restricts to an ordinary distributivity structure in each fiber. Thus the functor $\mathcal{W}$ takes each $O \in \mathcal{B}$ to the web monoid in the fiber over $O$. The only new assertion of the fibered version is that the morphisms $\mathcal{W}(u) = \mathcal{F}_{\odot}(I_{\otimes}(u))$, for $u: O \rightarrow Q \in \mathcal{B}$ are homomorphisms with respect to the new multiplication.

\subsection{Essential Steps in The Proof of The Main Theorem}
We must check that $\nu$ is unique, construct it, and verify the conditions of the theorem. We will do the first two of these steps here, and carry out the remaining calculations in the appendix. These steps contain the key, and only idea of the proof.

As $\otimes$ preserves filtered colimits we need only determine compatible components $\nu_{n} : \mathcal{W}_{n} \otimes \mathcal{W} \rightarrow \mathcal{W}$, where $\mathcal{W}_{n}$ is the n-th stage of the construction of $\mathcal{W} = \mathcal{F}_{\odot}(I_{\otimes})$ from theorem \ref{freemonoidbjt}. We will prove that these components are uniquely defined by the conditions of the theorem and define $\nu$ using these components. We begin with a lemma.

\begin{lemma}[Bootstrap lemma]\label{bootstraplemma}
If the following diagram commutes,
\begin{center}
\begin{tikzpicture}
\matrix (m) [matrix of math nodes,column sep = 2cm, row sep = .7cm, text height = 1.5ex, text depth = .25ex] {
A & B \\
C \\
D & E \\
};

\path[->] (m-1-1) edge node[auto] {$f$} (m-1-2)
			 (m-1-2) edge node[auto] {$g$} (m-3-2)
			 (m-1-1) edge node[auto, swap] {$h$} (m-2-1)
			 (m-2-1) edge node[auto, swap] {$k$} (m-3-1)
			 (m-3-1) edge node[auto] {$l$} (m-3-2);

\end{tikzpicture}
\end{center}
where $h$ is an isomorphism and $k$ has a section $s$, then
\begin{displaymath}
l = g \circ f \circ h^{-1} \circ s
\end{displaymath}
In addition, if $k$ is an isomorphism, then the diagram commutes if and only if the above equation holds.
\end{lemma}
The proof is trivial. Applying this lemma to the main diagram and the canonical section $\hat{s}$ of the free multiplication $\mu : \mathcal{W} \odot \mathcal{W} \rightarrow \mathcal{W}$ we obtain the equation:

\begin{displaymath}
\nu = \mu \circ (\nu \odot \nu) \circ \varphi^{-1} \circ \hat{s}
\end{displaymath}
which must be satisfied by $\nu$, but is unhelpful until we precompose it with $i_{n, \infty} \otimes 1: \mathcal{W}_{n} \otimes \mathcal{W} \rightarrow \mathcal{W} \otimes \mathcal{W}$, and obtain the following diagram:

\begin{center}
\begin{tikzpicture}
\matrix (m) [matrix of math nodes,column sep = 1.2cm, row sep = 1.2cm, text height = 1.5ex, text depth = .25ex]{
(\mathcal{W}_{1} \otimes \mathcal{W}) \odot (\mathcal{W}_{n-1} \otimes \mathcal{W}) & (\mathcal{W} \otimes \mathcal{W}) \odot (\mathcal{W} \otimes \mathcal{W}) & \mathcal{W} \odot \mathcal{W} \\
(\mathcal{W}_{1} \odot \mathcal{W}_{n-1}) \otimes \mathcal{W} & (\mathcal{W} \odot \mathcal{W}) \otimes \mathcal{W} & & \\
\mathcal{W}_{n} \otimes \mathcal{W} & \mathcal{W} \otimes \mathcal{W} & \mathcal{W}\\
};

\path[->] (m-1-2) edge node[auto,swap] {$\nu \odot \nu$} (m-1-3)
			 (m-2-2) edge node[auto, swap] {$\varphi^{-1}$} (m-1-2)
			 (m-3-2) edge node[auto, swap] {$\hat{s} \otimes 1$} (m-2-2)
			 (m-1-3) edge node[auto] {$\mu$} (m-3-3)
			 (m-3-2) edge node[auto] {$\nu$} (m-3-3)
			 (m-1-1) edge node[auto] {} (m-1-2)
			 (m-2-1) edge node[auto] {$\varphi^{-1}$} (m-1-1)
			 (m-3-1) edge node[auto] {$s_{n} \otimes 1$} (m-2-1)
			 (m-3-1) edge node[auto] {$i_{n} \otimes 1$} (m-3-2)
			 (m-2-1) edge node[auto] {$(i_{1} \odot i_{n-1}) \otimes 1$} (m-2-2)
			 (m-1-1) edge[in=160, out=20] node[auto] {$\nu_{1} \odot \nu_{n-1}$} (m-1-3)
			 (m-3-1) edge[in=200, out=340] node[auto] {$\nu_{n}$} (m-3-3);
\end{tikzpicture}
\end{center}
The boundary of this diagram provides an inductive definition of $\nu_{n}$, starting from $\nu_{1}$. The unnamed arrow is $(i_{1} \otimes 1) \odot (i_{n-1} \otimes 1)$, and all the $i$ maps should have an additional $\infty$ subscript (omitted for readability). This diagram is commutative \emph{if} $\nu$ exists. The top and bottom ``bigons'' or ``biangles'' are commutative, since by definition $\nu_{n} = \nu \circ i_{n,\infty} \otimes 1$, and $\odot$ is a functor. The upper small rectangle is commutative by naturality of $\varphi$. To prove commutativity of the lower small rectangle note the diagram

\begin{center}
\begin{tikzpicture}
\matrix (m) [matrix of math nodes,column sep = 1.3cm, row sep = 1cm, text height = 1.5ex, text depth = .25ex]{
\mathcal{W}_{1} \odot \mathcal{W}_{n-1} & \mathcal{W}_{1} \odot \mathcal{W} & \mathcal{W} \odot \mathcal{W} \\
\mathcal{W}_{n} & \mathcal{W} & \mathcal{W} \\
};

\path[->] (m-1-1) edge node[auto] {$1\odot i_{n-1}$} (m-1-2)
			 (m-1-2) edge node[auto] {$i_{1} \odot 1$} (m-1-3)
			 (m-2-1) edge node[auto] {$s_{n}$} (m-1-1)
			 (m-2-2) edge node[auto] {$s$} (m-1-2)
			 (m-2-1) edge node[auto] {$i_{n}$} (m-2-2)
			 (m-2-2) edge node[auto] {$1$} (m-2-3)
			 (m-2-3) edge node[auto] {$\hat{s}$} (m-1-3);
\end{tikzpicture}
\end{center}
which is the lower small rectangle with $s$ added in the middle. The left square commutes since $s$ is by definition the limit of $1 \odot i_{n-1, \infty} \circ s_{n}$. The right square commutes by proposition \ref{canonicalsection} (this is how we defined $\hat{s}$).

So far we have obtained that if $\nu$ exists, then the $\nu_{n}$ must satisfy

\begin{equation}\label{definitionofnun}
\nu_{n} = \mu \circ (\nu_{1} \odot \nu_{n-1}) \varphi^{-1} (s_{n} \otimes 1),
\end{equation}
which means that any candidate for $\nu$ is uniquely determined by $\nu_{1}$. The equation immediately gives the compatibility condition $\nu_{n} \circ i_{n} \otimes 1 = \nu_{n-1}$ -- just add $\mathcal{W}_{n-1} \otimes \mathcal{W}$ in the lower left corner of the diagram above, two analogous small rectangles above it, and use induction. We define $\nu_{0} = \nu_{1} \circ i_{0}$.

We will show in the appendix that $\nu_{1}$ is uniquely determined by the unit conditions. We note that
\begin{displaymath}
\mathcal{W}_{1} \otimes \mathcal{W} \simeq (I_{\odot} \sqcup I_{\otimes}) \otimes \mathcal{W} \simeq I_{\odot} \otimes \mathcal{W} \sqcup I_{\otimes} \otimes \mathcal{W}
\end{displaymath}
Thus the map $\nu_{1}$ is determined by what happens on both of these components. The calculations in the appendix give these components as
\begin{displaymath}
\nu_{1} = (i_{0} \psi_{\mathcal{W}}^{-1}, \lambda_{\mathcal{W}}^{\otimes})
\end{displaymath}
Where $i_{0} : I_{\odot} \rightarrow \mathcal{W}$.

\begin{proposition}[Uniqueness of $\nu$]
If $\nu$ exists, then it is the colimit of the arrows $\nu_{n} : \mathcal{W}_{n} \otimes \mathcal{W} \rightarrow \mathcal{W}$, with $\nu_{0}, \nu_{1}$ defined above, and $\nu_{k}$ defined by induction using equation \ref{definitionofnun}, for $k > 1$.
\end{proposition}
\begin{proof}
$\nu$ is determined by the family $\nu_{n} = \nu \circ i_{n, \infty} \otimes 1$. The calculations above (or in the appendix, in the case $n = 1$) determine these components uniquely.
\end{proof}

\begin{definition}[The definition of $\nu$]\label{definitionofnu}
We define $\nu : \mathcal{W} \otimes \mathcal{W} \rightarrow \mathcal{W}$ as the colimit of the arrows $\nu_{n} : \mathcal{W}_{n} \otimes \mathcal{W} \rightarrow \mathcal{W}$.
\end{definition}

We are now left with checking that this definition works. As already indicated, this is done in the appendix, since it consists of tedious inductive calculations.

\section{Fibrations of Signatures}

We will now introduce the structures among which the three tensors theorem was first discovered and applied. They are based on signatures with nonstandard amalgamation. Familiarity with ordinary multisorted signatures is not required, but will help greatly.

A few remarks about the approach of \cite{HMP} are in order. The structures used there are very similar to ours -- if not for the two levels, they would be exactly monoids in the fibration of signatures with amalgamation. The end results (the web monoid and the multicategory of function replacement) differ in annoying details, but the same problem must be overcome in their construction. It is the fact that replacing function symbols in a formal composite with formal composites results in very complicated shuffling among the surviving function symbols, regardless of the conventions chosen to extract the symbols from a formal composite. Indeed an example is given in appendix \ref{amalgamationisnecessaryexample} which shows that shuffling \emph{must} occur, that is, nonstandard amalgamation is necessary, and follows from the ``geometry'' we wish to preserve. In \cite{HMP} this is dealt with by requiring the construction to be functorial, and defining the permutations only in cases in which they are trivially determined (all the function symbols are distinct). We can give explicit recursive formulas for these bijections in all cases, and functoriality follows from this.

We will not write down these formulas however, since the results are quite un\TeX-able, and do not seem to be of interest. The reader may extract them from the abstract construction of $\nu$ in the three tensors theorem and the definitions given here.

\paragraph{The operad of symmetries.} This ordinary operad will provide us with the combinatorics necessary for our description of opetopic sets. Using it systematically will allow us to minimize our computational effort. Let $S_{n}$ be the n-th symmetric group. The operad of symmetries $\mathbf{S}$ has the $S_{n}$ as the sets of $n$-ary operations, and composition is defined by (\cite{Leinster}):

\begin{displaymath}
\sigma \ast (\rho_{1}, \ldots, \rho_{n})(k_{1} + \ldots + k_{i-1} + j) = k_{\sigma^{-1}(1)} + \ldots + k_{\sigma^{-1}(\sigma(i)-1)} + \rho_{i}(j),
\end{displaymath}
where $\sigma \in S_{n}$, $\rho_{i} \in S_{k_{i}}$, and $1 \leq i \leq n$, $1 \leq j \leq k_{i}$

This just means that we permute $n$ disjoint blocks according to $\sigma$, and apply $\rho_{i}$ in the block that was i-th in the beginning. Contrary to what has been claimed in the literature, the operation $\ast$ is \emph{not} a homomorphism of groups, much less so an obvious one.

It will be convenient to adopt a notation which allows us to compute compositions of permutations of the form $\sigma \ast (\rho_{1}, \ldots, \rho_{n})$ in certain cases. First we will assume the number of blocks is fixed and equal to $n$. We will then write $(i, j)$ for the $j$-th entry in the $i$-th block. For each block the number of entries is arbitrary. Thus, by definition, we have

\begin{displaymath}
\sigma \ast (\rho_{1}, \ldots, \rho_{n})(i,j) = (\sigma(i), \rho_{i}(j))
\end{displaymath}
This will allow us to compose such permutations easily, but only if the block lengths match. This will never be a problem, since we will always know that this happens a priori.

With this notation in hand we can prove

\begin{lemma}\label{lemmaforfunctorialityofotimes} If $\sigma, \sigma' \in S_{n}$ and $\rho_{i}, \rho_{i}'$ are permutations, such that the domains of $\rho_{i}$ and $\rho_{\sigma(i)}'$ are equal, then

\begin{displaymath}
\sigma ' \ast (\rho_{1}', \dots, \rho_{n}') \circ \sigma \ast (\rho_{1}, \dots, \rho_{n})= (\sigma' \circ \sigma) \ast (\rho_{\sigma(1)}' \circ \rho_{1}, \dots, \rho_{\sigma(n)}' \circ \rho_{n})
\end{displaymath}

\end{lemma}
\begin{proof}
The assumptions allow us to use the double index notation. We have

\begin{displaymath}
\sigma \ast (\rho_{1}, \ldots, \rho_{n})(i,j) = (\sigma(i), \rho_{i}(j))
\end{displaymath}
applying $\sigma ' \ast (\rho_{1}', \dots, \rho_{n}')$ to this (this is where we use the assumptions on the domains) yields

\begin{displaymath}
(\sigma ' (\sigma(i)), \rho_{\sigma(i)}'(\rho_{i}(j)))
\end{displaymath}
concluding the proof
\end{proof}

Note that we can also compute inverses with this notation. We have

\begin{displaymath}
\sigma \ast (\rho_{1}, \ldots, \rho_{n})^{-1}(i,j) = (\sigma^{-1}(i), \rho^{-1}_{\sigma^{-1}(i)}(j)).
\end{displaymath}
Again, we note that this formula is valid only if $j$ has a proper range for every $i$ (in this case, $j$ ranges over the domain of $\rho^{-1}_{\sigma^{-1}(i)}$). This will never be a problem in our calculations.

We will occasionaly use longer indices, for example $(i,j,k)$, but we will not calculate with them. However such calculations would be justified, since $\mathbf{S}$ is associative.

\subsection{Signatures with Amalgamation}

\paragraph{Notation.} Let  $[n] = \{ 0, \ldots, n \}$, and $(n] = \{ 1, \ldots, n \}$ for $n \in \mathbb{N}$. In particular $[n] = [0] \cup (n]$ and $(0] = \emptyset $. For a set $O$ we define $O^{\dag}_{n} = O^{[n]}$, $O^{\ast}_{n} = O^{(n]}$ and $O^{\dag} = \bigcup_{n \in \mathbb{N}} O^{[n]}$, $O^{\ast} = \bigcup_{n \in \mathbb{N}} O^{(n]}$. Thus $O^{\ast}$ is the set of ordered lists of elements of $O$, and $O^{\dag}$ is the set of lists with an additional element at the beginning.

$S_n$, the n-th symmetric group, for $n > 0$, acts on $O^{\dag}_{n}$ i $O^{\ast}_{n}$ on the right by precomposition (the lists are functions on $[n]$), leaving $0$ fixed.

If $d : [n] \rightarrow O$ is a list, then we denote its restrictions of positive numbers by $d^{+}:(n]\rightarrow O$, and its restriction to $[0]$ by $d^{-}: [0] \rightarrow O$. This establishes a bijection $\langle (-)^{-}; (-)^{+} \rangle : O^{\dag} \rightarrow O \times O^{\ast}$. We have an obvious functor $(-)^{\dag} : \mathbf{Set} \rightarrow \mathbf{Set}$.

\paragraph{The category of signatures with amalgamation.} It will be denoted $\mathbf{Sig}_{a}$. Its objects are set maps $\partial: A \rightarrow O^{\dag}$. They will usually be referred to by their domain, leaving the map, which is called the typing function, implicit. For $a \in A$ we denote $\partial_{a} := \partial(a): [n] \rightarrow O$ and $|a| := n$, which is called the arity of $a$. For morphisms $(A, \partial, O) \rightarrow (B, \partial, Q)$ we take triples $(f, \sigma, u)$ (denoted just by $f$ henceforth), where $f: A \rightarrow B$ and $u: O \rightarrow Q$ are functions, and for any $a \in A$, $\sigma_{a}$ is a permutation $(n] \rightarrow (n]$, where $n = |a|$, which makes the following diagram commute.

\begin{center}
\begin{tikzpicture}
\matrix (m) [matrix of nodes, row sep = 1cm, column sep = 1cm, text height =1.5ex, text depth =0.25ex] {
 $[n]$ & $[n]$ \\
$O$ & $Q$ \\
};

\path[->] (m-1-1) edge node[auto] {$\sigma_{a}$} (m-1-2)
			 (m-1-1) edge node[auto, swap] {$\partial_{a}$} (m-2-1)
			 (m-2-1) edge node[auto] {$u$} (m-2-2)
			 (m-1-2) edge node[auto] {$\partial_{f(a)}$} (m-2-2);

\end{tikzpicture}
\end{center}

The fibration of signatures with amalgamation is defined by the functor $p_{a}: \mathbf{Sig}_{a} \rightarrow \mathbf{Set}$, which sends $\partial: A \rightarrow O^{\ast}$ to $O$, and each morphism $(f, \sigma, u)$ to $u$. Prone morphisms arise from pullbacks along $u^{\dag}$ (take all $\sigma$ to be the identity), like in a fundamental fibration.

Here is some intuition. The elements of $A$ are function symbols, like in any ordinary signature considered in logic or universal algebra. The arities determine how many inputs every function symbol has. The elements of $O$ are called sorts or types. For example the signature of ``rings and modules'' would have two types -- one for the elements of a ring, and another for the elements of a module, as well as function symbols that define ring and module operations, one of which allows the ring to act on the module.

The morphisms are defined in such a way as to allow us to specify, for each function symbol, how its inputs are related to the ones in its image. This information is specified in the permutations $\sigma_{a}$, and is required to be compatible with type changes (i.e. the $u$ map). These permutations are called amalgamation permutations, and if they are nontrivial we say that we are dealing with nonstandard amalgamation. Otherwise we will say that the morphism is strict, or has trivial or standard amalgamation.

\paragraph{Cocompleteness properties of $\mathbf{Sig}_{a}$.} Signatures without amalgamation have very nice (co)completeness properties -- they are complete and cocomplete. Unfortunately the addition of amalgamation permutations spoils some of these properties, as the following example shows.

Let $A$ be a signature with one binary function symbol, over a singleton set $O = \{ \ast \}$. Then we have two obvious morphisms $A \rightarrow A$ -- the identity, and a morphism which permutes the inputs of our function symbol. Since permutations are invertible, these two morphisms are not equalized or coequalized by any other morphism. Therefore $\mathbf{Sig}_{a}$ does not have all equalizers or coequalizers (fibered or not).

Note also that signatures with amalgamation do not have a terminal object (fiberwise, or globally).

Fortunately signatures with amalgamation are nice enough for all the constructions we will deal with.

\begin{proposition}\label{exacnessforsiga}
The fibration $\mathbf{Sig}_{a}$ has the following (co)completeness properties:
\begin{enumerate}

\item All pullbacks (fibered or not).

\item Small coproducts (fibered or not).

\item All (small) fibered filtered colimits.

\end{enumerate}
\end{proposition}
\begin{proof}
We leave the first two items as a warmup exercise. The construction of fibered filtered colimits unfortunately requires some work.

By corollary \ref{fiberedcolimits} we may restrict our attention to individual fibers. We will consider the fiber $\mathbf{Sig}_{a}/O$. There is a forgetful functor $U: \mathbf{Sig}_{a}/O \rightarrow \mathbf{Set}/O$, which forgets the input types, but not the output (it corresponds to the projection $O^{\dag} \simeq O \times O^{\ast} \rightarrow O$). The category $\mathbf{Set}/O$ is obviously cocomplete. We will use it to build our colimits.

Consider a filtered diagram $F: \mathcal{D} \rightarrow \mathbf{Sig}_{a}/O$. Each signature is the coproduct of countably many signatures consisting of all the $n$-ary function symbols of the original signature, for $n \in \mathbb{N}$. Morphisms preserve arity, so this is also true for the entire diagram $F$. We may therefore assume that all the values of $F$ consist of signatures with function symbols of a fixed arity $n \in \mathbb{N}$.

By cocompleteness $U \circ F$ has a colimiting cone $\tau: U \circ F \xrightarrow{\cdot} X$ in $\mathbf{Set}/O$. We will show that it can be lifted to a cone in $\mathbf{Sig}_{a}/O$. The fact that any such lift is colimiting is trivial, since permutations are invertible.

If all the values of $F$ are empty, the colimit is empty, and we are done. We may assume that $F$ has nonempty values. Since all the function symbols in the values of $F$ are $n$-ary, we declare that each element of $X$ is also an $n$-ary symbol. We must define the typing of each symbol, and the permutation amalgamations of the components of the colimiting cone. For each $x \in X$ (which exist, since $F$ has nonempty values) consider its inverse image in the diagram $U \circ F$ -- those function symbols, which map to $x$ under the components of the colimiting cone. These inverse images are disjoint, and therefore we can consider them separately.

Choose an $f \in F(d)$ which maps to $x$ under $\tau_{d}$. We declare that the amalgamation permutations of $\tau_{d}$ are the identity for $f$. This gives us a typing of $x$. This also determines the amalgamation permutations of all other symbols which map to $x$ -- the diagram is filtered, and permutations are invertible, so considering only the permutations we can get anywhere in the inverse image of $x$ starting from $f$. Such a procedure may result in a contradiction -- and it does in the example we gave above for nonexistence of coequalizers. But in our case the diagram is filtered, so any two parallel morphisms are equalized by a third one, and no contradiction can arise. Any two potentially different ways for getting from $f$ to another symbol have equal amalgamation permutations.
\end{proof}

\paragraph{The monoidal structure.} We will now define the monoidal fibration structure on our signatures. We begin by defining the two fibered functors.

\begin{center}
\begin{tikzpicture}
\matrix (m) [matrix of math nodes, row sep = 1cm, column sep = 1cm, text height =1.5ex, text depth =0.25ex] {
\mathbf{Sig}_{a} \times_{\mathbf{Set}} \mathbf{Sig}_{a} & \mathbf{Sig}_{a} & \mathbf{Set} \\
& \mathbf{Set} & \\
};

\path[->] (m-1-1) edge node[auto] {$\otimes$} (m-1-2)
			 (m-1-3) edge node[auto, swap] {$I$} (m-1-2)
			 (m-1-1) edge node[auto, swap] {$p_{a} \times_{\scriptscriptstyle{\mathbf{Set}}} p_{a}$} (m-2-2)
			 (m-1-3) edge node[auto] {$1_{\mathbf{Set}}$} (m-2-2)
			 (m-1-2) edge node[auto] {$p_{a}$} (m-2-2);

\end{tikzpicture}
\end{center}
\noindent
If $A$ and $B$ are signatures over $O$, then we set

\begin{displaymath}
A \otimes_{O} B = \lbrace \langle a, b_{i} \rangle_{i \in (|a|]} : a \in A, b_{i} \in B, \partial_{a}(i) = \partial_{b_{i}}(0), \mbox{for } i \in (|a|] \rbrace,
\end{displaymath}
\noindent
which is to be thought of as the signature of formal composites of symbols from $A$ and $B$. Note that we allow $|a| = 0$, which means a nullary formal composite (no $b_{i}$)\footnote{Without such composites associativity fails, among many other important things.}. For the typing we set

\begin{displaymath}
\partial^{A \otimes_{O} B}_{\langle a, b_{i} \rangle_{i \in (|a|]}} = [ \partial_{a}^{-}, \partial_{b_{i}}^{+}]_{i \in (|a|]} : [\sum_{i = 1}^{|a|} |b_{i}|] \longrightarrow O.
\end{displaymath}
\noindent
This means that the output type of $\langle a, b_{i} \rangle$ is the output type of $a$, and the input types are those of the $b_{i}$ placed side by side, in order of increasing $i$.

For morphisms $f,g$ over $u: O \rightarrow Q$ we set
 
\begin{displaymath}
f \otimes_{u} g (\langle a, b_{i} \rangle_{i \in (|a|]}) = \langle f(a), g(b_{\sigma_{a}^{-1}(j)}) \rangle_{j \in (|f(a)|]}
\end{displaymath}

\begin{displaymath}
(\sigma \otimes_{u} \tau)_{\langle a, b_{i} \rangle_{i \in (|a|]}} = \sigma_{a} \ast (\tau_{b_{1}}, \ldots, \tau_{b_{|a|}}),
\end{displaymath}
\noindent
where $\ast$ is composition in the operad of symmetries. This defines a functor by lemma \ref{lemmaforfunctorialityofotimes}.

For the unit we set $I(O) = \partial^{I_{O}}: O \rightarrow O^{\dag}$, which assigns to every $o \in O$ the unary typing with constant value $o$. This defines a fibered functor in an obvious way.

We can now give a little more intuition. The ``transformation rule'' we have chosen for our tensor product, that is the definition of $f \otimes g$, follows from the fact that the amalgamation permutations $\sigma$ specify how the inputs of a function symbol are mapped to its image under the morphism. Then the definition of $\sigma \otimes \tau$ represents a natural way to combine the actions of these permutations on a formal composite. This is the defining feature distinguishing signatures with amalgamation from ordinary ones, and easily results in some rather baroque formulas\footnote{Footnotes 7 and 8 should give a sense of exactly how baroque they can get!}, which we have done our best to avoid.

Monoids with respect to this structure give us a way to factually compose formal composites. Returning to our example of rings and modules we see that this would give a multiplication table between ring and module operations, which could encode the usual axioms for rings and modules. No nonstandard amalgamation would be necessary. The category of actions of this monoid along the tautologous action of $\mathbf{Sig}_{a}$ on $cod: \mathbf{Set}^{\rightarrow} \rightarrow \mathbf{Set}$ (see section \ref{comparisonsection}, or \cite{Zawadowski}) would then be equivalent to the category of (all) modules.

We still need to define the coherence isomorphisms $\alpha, \lambda$, and $\rho$. They are given by

\begin{eqnarray*}
\alpha_{A, B, C}(\langle a, \langle b_{i}, c_{i,j} \rangle \rangle) & = & \langle \langle a, b_{i} \rangle, c_{i,j} \rangle \\
\lambda_{A}(\langle 1_{\partial_{a}^{A}(0)}, a \rangle) & = & a \\
\rho_{A}(a) & = & \langle a, 1_{\partial_{a}^{A}(1)}, \dots, 1_{\partial_{a}^{A}(n)} \rangle,
\end{eqnarray*}
\noindent
and all the permutations taken to be the identity. The double index $(i,j)$ is ordered lexicographically.

\begin{theorem}
The structure given above defines a strong monoidal structure on the fibration $\mathbf{Sig}_{a}$.
\end{theorem}
\begin{proof}
The only nontrivial thing to prove is the naturality of $\alpha$. This means we must check the commutativity of the following diagram.

\begin{center}
\begin{tikzpicture}
\matrix (m) [matrix of math nodes, row sep = 1cm, column sep = 2cm, text height = 1.5ex, text depth = .25ex] {
A \otimes_{O} (B \otimes_{O} C) & (A \otimes_{O} B) \otimes_{O} C & O \\
A' \otimes_{Q} (B' \otimes_{Q} C') & (A' \otimes_{Q} B') \otimes_{Q} C' & Q \\
};

\path[->] (m-1-1) edge node[auto] {$\alpha_{A, B, C}$} (m-1-2)
			 (m-2-1) edge node[auto] {$\alpha_{A', B', C'}$} (m-2-2)
			 (m-1-1) edge node[auto, swap] {$f \otimes_{u} (g \otimes_{u} h)$} (m-2-1)
			 (m-1-2) edge node[auto] {$(f \otimes_{u} g) \otimes_{u} h$} (m-2-2)
			 (m-1-3) edge node[auto] {$u$} (m-2-3);

\end{tikzpicture}
\end{center}
where all the morphisms $f, g$ and $h$ are over $u: O \rightarrow Q$. The amalgamation permutations of these morphisms will be denoted $\sigma, \tau$ and $\delta$, respectively.

The two ways of going around the diagram give us 

\begin{displaymath}
\langle \langle f(a), g(b_{\sigma_{a}^{-1}(i)}) \rangle, h(c_{\sigma \otimes \tau_{\langle a, b_{i} \rangle}^{-1}(i, j)}) \rangle
\end{displaymath}
and amalgamation permutations $(\sigma \otimes \tau) \otimes \delta_{\langle \langle a, b_{i} \rangle, c_{i, j} \rangle}$, and

\begin{displaymath}
\langle \langle f(a), g(b_{\sigma_{a}^{-1}(i)}) \rangle, h(c_{\sigma_{a}^{-1}(i), \tau_{b_{\sigma_{a}^{-1}(i)}}^{-1}(j)}) \rangle
\end{displaymath}
with amalgamation permutations $\sigma \otimes (\tau \otimes \delta)_{\langle a, \langle b_{i}, c_{i, j} \rangle \rangle}$.

The equality of the amalgamation permutations follows from the associativity of $\ast$ -- the multiplication in the operad of symmetries. The terms are equal, since

\begin{displaymath}
\sigma \otimes \tau_{\langle a, b_{i} \rangle}^{-1}(i, j) = (\sigma_{a}^{-1}(i), \tau_{b_{\sigma_{a}^{-1}(i)}}^{-1}(j))
\end{displaymath}
by the formula for inverses we gave when discussing the operad of symmetries.
\end{proof}

\begin{remark}
This structure is neither left nor right closed (considering it fiber-by-fiber). It is not left closed because $A \otimes (-)$ does not preserve coproducts. For right closedness it is easy to see that if $A$ consists of a single symbol, then for any $X$ the maps $A \rightarrow X$ form a free $Aut(A)$-set. Taking $X = \underline{Hom}(B, C)$ for suitable $A, B$ and $C$ yields a contradiction since the set of maps $A \otimes B \rightarrow C$ need not admit a free $Aut(A)$ action.
\end{remark}

\subsection{Monoidal Signatures with Amalgamation}

Consider the functor $\mathcal{U}: Mon(\mathbf{Sig}_{a}) \rightarrow \mathbf{Set}$, which maps every monoid $M$ to its underlying set of function symbols and every homomorphism of monoids $u$ to the underlying function. The fibration of monoidal signatures with amalgamation is defined as the pullback of $\mathbf{Sig}_{a}$ along $\mathcal{U}$,

\begin{center}
\begin{tikzpicture}
\matrix (m) [matrix of math nodes, row sep = 1cm, column sep = 1cm, text height = 1.5ex, text depth = .25ex] {
\mathbf{Sig}_{ma} & \mathbf{Sig}_{a} \\
Mon(\mathbf{Sig}_{a}) & \mathbf{Set} \\
};

\path[->] (m-1-1) edge node[auto] {} (m-1-2)
			 (m-1-1) edge node[auto, swap] {$p_{ma}$} (m-2-1)
			 (m-2-1) edge node[auto] {$\mathcal{U}$} (m-2-2)
			 (m-1-2) edge node[auto] {$p_{a}$} (m-2-2);

\end{tikzpicture}
\end{center}
\noindent
By corollary \ref{pullbackspreservemonoidalfibrations} this is a monoidal fibration. The monoidal structure is defined by exactly the same formulas as the one for $\mathbf{Sig}_{a}$, but the set of types (formerly $O$) remembers that it is a monoid (denoted $M$).

The entire fibration is also almost exactly the same as $\mathbf{Sig}_{a}$. The only difference is that types form a monoid in $\mathbf{Sig}_{a}$ (over some set of types in $\mathbf{Set}$). We will continue to speak about function symbols and amalgamation permutations in this context. Note that $\mathbf{Sig}_{ma}$ inherits all the fibered (co)completeness properties of proposition \ref{exacnessforsiga}, since they are pullback-stable.

The main point of this construction is that when the types form a monoid, a new monoidal structure appears, and we can define a distributivity structure. Before we do so, however, we need some preparation.

\paragraph{The separation principle.} We will often need to verify equality of certain natural transformations. The problem can be split into two parts -- check equality on function symbols, and on permutation amalgamations. The first part is usually easy, but the second part is often intractable -- the formulas are just too complicated (see for example the calculations in appendix \ref{coherencefordistributivity} and try to compare the amalgamation permutations for condition I).

The separation principle is inspired by the construction of the multicategory of function replacement in \cite{HMP}, where the permutations are completely avoided (at a cost of definiteness of the construction). We will settle for a little less. We will exploit naturality to get rid of the second part. Establishing naturality will be difficult enough. We note the following trivial lemma.

\begin{lemma}\label{strictnesslemma}
If in the following diagram in $\mathbf{Sig}_{ma}$ (or $\mathbf{Sig}_{a}$) the morphisms $h$ and $k$ are strict, then for $a \in A$ we have $\sigma_{a} = \theta_{h(a)}$, where $\sigma$ are the amalgamation permutations of $f$ and $\theta$ are the permutations of $g$.
\end{lemma}

\begin{center}
\begin{tikzpicture}
\matrix (m) [matrix of math nodes, row sep = 1cm, column sep = 1cm, text height = 1.5ex, text depth = .25ex] {
A & B \\
A' & B' \\
};

\path[->] (m-1-1) edge node[auto] {$f$} (m-1-2)
			 (m-1-1) edge node[auto, swap] {$h$} (m-2-1)
			 (m-2-1) edge node[auto] {	$g$} (m-2-2)
			 (m-1-2) edge node[auto] {$k$} (m-2-2);

\end{tikzpicture}
\end{center}

We will consider the following construction. Let $M \in Mon(\mathbf{Sig}_{a})$ be a monoid. We can construct a new monoid $M_{\mathbb{N}}$ over the same set of types as follows. The universe of $M_{\mathbb{N}}$ is $M \times \mathbb{N}$, the typing is defined by projecting onto $M$: $M \times \mathbb{N} \rightarrow M \rightarrow O^{\dag}$. The unit is the composite $I \rightarrow M \simeq M \times \{ 0 \} \hookrightarrow M_{\mathbb{N}}$, and multiplication is defined by

\begin{displaymath}
\mu^{M_{\mathbb{N}}} \langle (f,n_{0}), (g_{1}, n_{1}), \dots, (g_{k}, n_{k}) \rangle = (\mu^{M}(f, g_{1}, \dots, g_{k}), \sum_{i = 0}^{k} n_{i})
\end{displaymath}
The fact that these formulas define a monoid follows from the fact that $\mathbb{N}$ is a monoid in $\mathbf{Set}$. Obviously the projection map $\pi: M_{\mathbb{N}} \rightarrow M$ is a homomorphism of monoids. It will be essential in applications that this homomorphism is strict. In fact we have defined a functor, and $\pi$ is a natural transformation, but this will be irrelevant in our arguments.

Let $\mathcal{E}$ be a fibration over $Mon(\mathbf{Sig}_{a})$ and let $F, G: \mathcal{E} \rightarrow \mathbf{Sig}_{ma}$ be two fibered functors. Consider the following two properties of $F$ and $G$ (which are to hold for any $M$):

\begin{enumerate}

\item For any $X \in \mathcal{E}$ over $M$ there is a prone morphism $\theta: Y \rightarrow X$ over $\pi$ such that both $F(\theta)$ and $G(\theta)$ are strict morphisms.

\item For any (and hence every) prone morphism $\theta: Y \rightarrow X$ over $\pi$ the following holds: for any $f \in F(X)$ there is an $\tilde{f} \in F(Y)$ in the fiber over $f$ (i.e. mapping to $f$ under $F(\theta)$) whose typing is injective.

\end{enumerate}
Pairs of functors possessing property 1 will be called \emph{agreeable}. Functors with property 2 will be called \emph{separated}. All pairs of functors we will deal with will be agreeable and separated, and this fact will always be very easy to check.

The notion of agreeability can be extended to any set of functors (we require all of them to be strict on a single prone morphism). We will then say that functors in this set are jointly agreeable.

\begin{theorem}[The Separation Principle]\label{separationprinciple}
Let $\mathcal{E}$ be a fibration over $Mon(\mathbf{Sig}_{a})$, let $F, G: \mathcal{E} \rightarrow \mathbf{Sig}_{ma}$ be two fibered functors, and let $\phi^{1}, \phi^{2}: F \rightarrow G$ be two fibered natural transformations whose components are equal on function symbols. If $F$ and $G$ are agreeable and $F$ is separated, then $\phi^{1} = \phi^{2}$.
\end{theorem}

\begin{proof}
We must prove equality of all components. Since they are equal on function symbols, we must check the equality of amalgamation permutations. Consider an $f \in F(X)$ and choose a prone $\theta$ over $\pi$ for which both $F(\theta)$ and $G(\theta)$ are strict. We have the following situation:

\begin{center}
\begin{tikzpicture}
\matrix (m) [matrix of math nodes, row sep = 1.2cm, column sep = 1cm, text height = 1.5ex, text depth = .25ex] {
F(Y) & G(Y) \\
F(X) & G(X) \\
};

\node[left = of m-1-1] (tildef) {$\tilde{f}$};
\node[below = 1.1cm of tildef] (f) {$f$};

\path[|->] (tildef) edge (f);

\path[->] (m-1-1.15) edge node[auto] {$\phi_{Y}^{1}$} (m-1-2.165)
			 (m-1-1.345) edge node[auto, swap] {$\phi_{Y}^{2}$} (m-1-2.195)
			 (m-1-1) edge node[auto, swap] {$F(\theta)$} (m-2-1)
			 (m-2-1.15) edge node[auto] {$\phi_{X}^{1}$} (m-2-2.165)
			 (m-2-1.345) edge node[auto, swap] {$\phi_{X}^{2}$} (m-2-2.195)
			 (m-1-2) edge node[auto] {$G(\theta)$} (m-2-2);

\end{tikzpicture}
\end{center}
By separability there is an $\tilde{f} \in F(Y)$ in the fiber over $f$ whose typing is injective. We know that $\phi_{Y}^{1}$ and $\phi_{Y}^{2}$ are equal on $\tilde{f}$. Their amalgamation permutations on $\tilde{f}$ are uniquely determined, since its typing is injective. Thus they are equal. Now lemma \ref{strictnesslemma} implies that the amalgamation permutations are equal for $f$ also.
\end{proof}

\begin{remark}
The separation principle is also true for $\mathbf{Sig}_{a}$. Instead of considering $M_{\mathbb{N}}$ we consider $O \times \mathbb{N}$ for $O \in \mathbf{Set}$.
\end{remark}

\begin{remark}
There is considerable room in the above argument -- one need not consider $M_{\mathbb{N}}$, but some other monoid with infinite fibers over $M$. In our applications is also important that the projection $\pi$ has standard amalgamation. The choices we have made work in general and make the statement of the separation principle short enough to be applicable.
\end{remark}

\begin{remark}
A similar argument can be used to \emph{define} natural transformations between agreeable functors, when we know what to do on function symbols. This is what is done in \cite[part II]{HMP} to construct multiplication in the multicategory of function replacement.
\end{remark}

\paragraph{The second monoidal strucutre $\odot$.} Consider an object $A \in \mathbf{Sig}_{ma}$ over a monoid $M$. It has a typing $A \rightarrow M^{\dag}$, and we can consider the output type $A \rightarrow M^{\dag} \rightarrow M$. Since $M$ is a monoid over some set $O$, it has its own typing $M \rightarrow O^{\dag}$. The composite $A \rightarrow M^{\dag} \rightarrow M \rightarrow O^{\dag}$ gives us a typing of $A$ over $O$. Thus every $a \in A$ has \emph{two} kinds of inputs and outputs. The ones just defined will be called horizontal, the old ones will be called vertical. As a set we define

\begin{displaymath}
A \odot_{M} B = A \otimes_{O} B,
\end{displaymath}
which means

\begin{displaymath}
\{ \dot{\langle} a, b_{i} \dot{\rangle}: \partial^{M}_{\partial_{a}^{A}(0)} (i) = \partial^{M}_{\partial_{b_{i}}^{B}(0)}(0) \},
\end{displaymath}
where $i$ ranges over $(0, k]$ with $k = |\partial^{A}_{a}(0)|$ (arity computed in $M$). We will use the notation $\dot{\langle} \ldots \dot{\rangle}$ and $\langle \ldots \rangle$ to distinguish between elements of $A \odot_{M} B$ and $A \otimes_{M} B$. To ease notation we will write $\check{a}$ for $\partial_{a}^{A}(0)$. We must define the typing of this set, that is a function $A \odot_{M} B \rightarrow M^{\dag}$. The output type is the composite

\begin{displaymath}
\partial_{\dot{\langle} a, b_{i} \dot{\rangle}}^{\odot, -} = A \otimes_{O} B \rightarrow M \otimes_{O} M \xrightarrow{\mu} M,
\end{displaymath}
where the arrows from $A$ and $B$ to $M$ are the output types (considered as morphisms in $\mathbf{Sig}_{a}$ with trivial amalgamation), used to define the horizontal typing, and $\mu$ is multiplication in $M$. Using our notation we can write
\begin{displaymath}
\partial_{\dot{\langle} a, b_{i} \dot{\rangle}}^{\odot, -} = \mu(\check{a}, \check{b}_{i}).
\end{displaymath}
The inputs are just the inputs of $a$ and $b_{i}$ concatenated in order

\begin{displaymath}
\partial_{\dot{\langle} a, b_{i} \dot{\rangle}}^{\odot, +} = [\partial_{a}^{A, +}, \partial_{b_{1}}^{B, +}, \dots, \partial_{b_{k}}^{B, +}] : (n_{0} + n_{1} + \dots n_{k}] \rightarrow M
\end{displaymath}

The reader should imagine that $a$ and $b_{i}$ are arranged on a level surface (``horizontally'') and all the vertical inputs (including those of $a$) are visible ``from above'', and are available in forming $(A \odot_{M} B)\otimes_{M} C$.

We need to describe the values of $\odot$ on morphisms. It is easy to see that a morphism in $\mathbf{Sig}_{ma}$ is completely described by a quintuple $(f, \tau, u, \sigma, v)$, where $(u, \sigma, v)$ is a homomorphism of monoids $M \rightarrow N$ in $\mathbf{Sig}_{a}$ over some function $v$ in $\mathbf{Set}$, and $(f, \tau)$ is a morphism in $\mathbf{Sig}_{a}$ from $A \rightarrow M^{\dag}$ to $B \rightarrow N^{\dag}$ over $u$. Let $f$ and $f'$ be two such homomorphisms. We define

\begin{displaymath}
f \odot f'(\dot{\langle} a, b_{1}, \dots, b_{k} \dot{\rangle}) = \dot{\langle} f(a), f'(b_{\sigma_{\check{a}}^{-1}(1)}), \dots, f'(b_{\sigma_{\check{a}}^{-1}(k)}) \dot{\rangle}
\end{displaymath}
as a function. The permutation $\tau \odot \tau'$ is permutes the \emph{vertical} inputs of the formal composites according to $\tau$ for the inputs from $a$ and $\tau'$ for inputs from the $b_{i}$, and places the blocks in which these inputs are arranged on the block belonging to its image. Formally we define $\tau \odot \tau'$ as follows

\begin{displaymath}
\tau \odot \tau'_{\dot{\langle} a, b_{i} \dot{\rangle}} = (1, \sigma_{\check{a}}) \ast (\tau_{a}, \tau'_{b_{1}}, \ldots, \tau'_{b_{k}}).
\end{displaymath}
where $(1, \sigma)$ means the coproduct of the identity on the singleton and $\sigma$ (conjugated by a translation to act on $[2, k+1]$). Using the same notation for more general permutations we could have written

\begin{displaymath}
\tau \odot \tau'_{\dot{\langle} a, b_{i} \dot{\rangle}} = (\tau_{a}, \sigma_{\check{a}} \ast (\tau'_{b_{1}}, \ldots, \tau'_{b_{k}}))
\end{displaymath}
The entire morphism $f \odot f'$ is now the quintuple $(f \odot f', \tau \odot \tau', u, \sigma, v)$.

The unit $I_{\odot}$ takes values $I_{\odot}(M)$, which are defined to be $\partial_{I_{\odot}(M)}: O \rightarrow M^{\dag} \simeq M \times M^{\ast}$. The first factor is the output type, and $\partial_{I_{\odot}(M)}(o)$ takes the value $e(o)$ on this factor, where $e: O \rightarrow M$ is the unit of multiplication in $M$. The inputs are defined to be empty for all $o \in O$.

We must define the coherence isomorphisms. Both $\rho^{\odot}$ and $\lambda^{\odot}$ are defined analogously to the previous case, replacing $\langle \dots \rangle$ with $\dot{\langle} \dots \dot{\rangle}$. They are given by

\begin{eqnarray*}
\lambda_{A}(\dot{\langle} 1_{\partial_{\check{a}}^{M}(0)}, a \dot{\rangle}) & = & a\\
\rho_{A}(a) & = & \dot{\langle} a, 1_{\partial_{\check{a}}^{M}(1)}, \dots, 1_{\partial_{\check{a}}^{M}(|\check{a}|)} \dot{\rangle} \\ 
\end{eqnarray*}

We take the amalgamation permutations to be the identity, since the unit $I_{\odot}$ has no vertical inputs. If it did there would be no bijection between the inputs of both sides. 

The only problem is the definition of $\alpha^{\odot}$. The problem is that $M$ may have nonstandard amalgamation. That is the multiplication $M \otimes_{O} M \rightarrow M$ need not be strict -- it can mix the inputs according to some nontrivial permutation. We have used it to define the horizontal typing. Because of this the naive associativity isomorphism $A \odot (B \odot C) \rightarrow (A \odot B) \odot C$ is not even well defined, since the values it ``should'' have are not necessarily among the elements of $(A \odot B) \odot C$.

The correct solution, for geometrical and other reasons\footnote{For example, the constructions in section \ref{comparisonsection} depend critically on this definition}, is the following. Denote by $\gamma$ (more precisely $\gamma^{M}$) the amalgamation permutations of the multiplication map of $M$ (that is of $\mu: M \otimes_{O} M \rightarrow M$, which lives in $\mathbf{Sig}_{a}/O$). Define

\begin{displaymath}
\alpha^{\odot}_{A, B, C}(\dot{\langle} a, \dot{\langle} b_{i}, c_{i, j} \dot{\rangle} \dot{\rangle}) = \dot{\langle} \dot{\langle} a, b_{i} \dot{\rangle}, c_{\gamma_{\langle \check{a}, \check{b}_{i} \rangle}^{-1}(i, j)} \dot{\rangle}
\end{displaymath}
as a function. We leave it to the reader to see that the term on the left is well defined. Indeed, this is the only formula which works when the horizontal inputs of $\dot{\langle} a, b_{i} \dot{\rangle}$ are all distinct.

We must define the vertical amalgamation permutations. For this consider $a, b_{i}$ and $c_{i,j}$ as formal variables. Let $\kappa_{\dot{\langle} a, \dot{\langle} b_{i}, c_{i, j} \dot{\rangle} \dot{\rangle}}$ be the permutation which sends each formal variable on the list $\dot{\langle} a, \dot{\langle} b_{i}, c_{i, j} \dot{\rangle} \dot{\rangle}$ to itself on the list $\dot{\langle} \dot{\langle} a, b_{i} \dot{\rangle}, c_{\gamma_{\langle \check{a}, \check{b}_{i} \rangle}^{-1}(i, j)} \dot{\rangle}$. More precisely these lists are

\begin{eqnarray*}
\dot{\langle} a, \dot{\langle} b_{1}, c_{1, 1}, \dots, c_{1, l_{1}} \dot{\rangle} \dots \dot{\langle} b_{k}, c_{k, 1}, \dots, c_{k, l_{k}} \dot{\rangle} \dot{\rangle}\\
\dot{\langle} \dot{\langle} a, b_{1}, \dots, b_{k} \dot{\rangle}, c_{\gamma_{\langle \check{a}, \check{b}_{i} \rangle}^{-1}(1, 1)}, \dots, c_{\gamma_{\langle \check{a}, \check{b}_{i} \rangle}^{-1}(k, l_{k})} \dot{\rangle} \\
\end{eqnarray*}

To define the amalgamation permutations of $\alpha^{\odot}$, which we will denote by $\pi$, we make $\kappa$ act on blocks of the appropriate length (the length of the inputs of each function symbol)

\begin{displaymath}
\pi_{\dot{\langle} a, \dot{\langle} b_{i}, c_{i, j} \dot{\rangle} \dot{\rangle}} = \kappa_{\dot{\langle} a, \dot{\langle} b_{i}, c_{i, j} \dot{\rangle} \dot{\rangle}} \ast (1_{(|a|]}, \dots 1_{(|c_{k, l_{k}}|]})
\end{displaymath}
thus, each block of inputs ``tracks'' the position of its corresponding function symbol. Again, this permutation is the only well defined one when all the inputs of $a, b_{i}$ and $c_{i, j}$ are distinct. Therefore by the (proof of the) separation principle it is the only formula that can be natural, given what we want to do with the function symbols.

We will now prove that $\alpha^{\odot}$ is natural and satisfies the pentagon identity. The rest of the proof that $(\odot, \alpha, \lambda, \rho)$ defines a monoidal fibration structure is very easy and formally identical to the corresponding part of the proof for $\otimes$ in $\mathbf{Sig}_{a}$.

\paragraph{Naturality of $\alpha^{\odot}$.} We consider the diagram

\begin{center}
\begin{tikzpicture}
\matrix (m) [matrix of math nodes, row sep = 1cm, column sep = 2cm, text height = 1.5ex, text depth = .25ex] {
A \odot (B \odot C) & (A \odot B) \odot C & M \\
A' \odot (B' \odot C') & (A' \odot B') \odot C' & N \\
};

\path[->] (m-1-1) edge node[auto] {$\alpha_{A, B, C}^{\odot}$} (m-1-2)
			 (m-2-1) edge node[auto] {$\alpha_{A', B', C'}^{\odot}$} (m-2-2)
			 (m-1-1) edge node[auto, swap] {$f \odot (g \odot h)$} (m-2-1)
			 (m-1-2) edge node[auto] {$(f \odot g) \odot h$} (m-2-2)
			 (m-1-3) edge node[auto] {$u$} (m-2-3);

\end{tikzpicture}
\end{center}
All three morphisms are over the homomorphism of monoids $u$. We will check that they are equal as functions first, and then consider the amalgamation permutations. Consider the term 

\begin{displaymath}
\dot{\langle} a,  \dot{\langle} b_{i}, c_{i, j} \dot{\rangle} \dot{\rangle}
\end{displaymath}
Applying both ways to go around the diagram we obtain

\begin{displaymath}
\dot{\langle} \dot{\langle} f(a), g(b_{\sigma_{\check{a}}^{-1}(i)}) \dot{\rangle}, h(c_{\xi_{1}^{-1}(i, j)}) \dot{\rangle}
\end{displaymath}
and

\begin{displaymath}
\dot{\langle} \dot{\langle} f(a), g(b_{\sigma_{\check{a}}^{-1}(i)}) \dot{\rangle}, h(c_{\xi_{2}^{-1}(i, j)}) \dot{\rangle}
\end{displaymath}
Where $\sigma$ are the amalgamation permutations of $u$, and $\xi_{1}$ and $\xi_{2}$ are given as follows\footnote{The ``check'' symbol over the lowermost index $a$ was replaced by $\partial$ due to \TeX-nical issues.}:

\begin{eqnarray*}
\xi_{1} & = & \gamma_{\langle \check{f(a)}, \check{g(b_{\sigma_{\partial{a}}^{-1}(i)})} \rangle}^{N} \circ \sigma \otimes \sigma_{\langle \check{a}, \check{b_{i}} \rangle}\\
\xi_{2} & = & \sigma_{\mu(\check{a}, \check{b}_{i})} \circ \gamma_{\langle \check{a}, \check{b}_{i} \rangle}^{M} \\
\end{eqnarray*}
Their equality follows from the fact that $u$ is a homomorphism of monoids -- this is the equality required from the amalgamation permutations of a homomorphism.

We are left with proving that the amalgamation permutations are equal, thus we must prove that

\begin{eqnarray*}
\pi_{\dot{\langle} f(a), \dot{\langle} g(b_{\sigma_{\check{a}}^{-1}(i)}), h(c_{\sigma \otimes \sigma_{\langle \check{a}, \check{b}_{i} \rangle}^{-1}(i, j)}) \dot{\rangle} \dot{\rangle}} \circ \tau \odot (\delta \odot \zeta)_{\dot{\langle} a, \dot{\langle} b_{i}, c_{i, j} \dot{\rangle} \dot{\rangle}} = \\
(\tau \odot \delta) \odot \zeta_{ \dot{\langle} \dot{\langle} a, b_{i} \dot{\rangle}, c_{\gamma_{\langle \check{a}, \check{b}_{i} \rangle}^{-1}(i, j)} \dot{\rangle}} \circ \pi_{\dot{\langle} a, \dot{\langle} b_{i}, c_{i, j} \dot{\rangle} \dot{\rangle}}
\end{eqnarray*}

Both these permutations permute the input blocks of the function symbols $a, b_{i}$ and $c_{i,j}$, and apply some permutation inside each block. This follows from our definitions of $\pi$ and $\beta \odot \chi$. We will prove their equality in two (concurrent) steps: we will show that the block permutations are equal, and then that the same permutation is applied inside each block.

To see the equality of block permutations we argue for each function symbol. The argument for $a$ is trivial. The arguments for $b_{i}$ are similar to (and simpler than) the arguments for $c_{i,j}$. We will therefore only consider those last symbols. Each $c_{i_{0},j_{0}}$ is part of a larger symbol $\dot{\langle} b_{i_{0}}, c_{i_{0},j} \dot{\rangle}$. Let us analyze what both sides do to this block and its elements.

The left permutation applies $\delta \odot \zeta_{\dot{\langle} b_{i_{0}}, c_{i_{0},j} \dot{\rangle}}$ to the input block of this larger symbol, and moves it to its position in $\dot{\langle} f(a), \dot{\langle} g(b_{\sigma_{\check{a}}^{-1}(i)}), h(c_{\sigma \otimes \sigma_{\langle \check{a}, \check{b}_{i}, \rangle}^{-1}(i,j)}) \dot{\rangle}  \dot{\rangle}$ (i.e. applies $(1, \sigma_{\check{a}})$ to the input blocks). In particular $\zeta_{c_{i_{0}, j_{0}}}$ is applied to our input block, and it is placed on the block of $h(c_{i_{0}, j_{0}})$ in $\dot{\langle} f(a), \dot{\langle} g(b_{\sigma_{\check{a}}^{-1}(i)}), h(c_{\sigma \otimes \sigma_{\langle \check{a}, \check{b}_{i}, \rangle}^{-1}(i,j)}) \dot{\rangle}  \dot{\rangle}$. The final $\pi$ moves the block to its position in $\dot{\langle} \dot{\langle} f(a), g(b_{\sigma_{\check{a}}^{-1}(i)}) \dot{\rangle}, h(c_{\xi_{2}^{-1}(i, j)}) \dot{\rangle}$, with $\xi_{2}$ given above.

The right permutation breaks up this bigger block, since $\pi$ is applied first. By definition of $\pi$, the input block of $c_{i_{0},j_{0}}$ is moved to its position in $\dot{\langle} \dot{\langle} a, b_{i} \dot{\rangle}, c_{\gamma_{\langle \check{a}, \check{b}_{i} \rangle}^{-1}(i,j)} \dot{\rangle}$. Then we must apply $(\tau \odot \delta) \odot \zeta_{ \dot{\langle} \dot{\langle} a, b_{i} \dot{\rangle}, c_{\gamma_{\langle \check{a}, \check{b}_{i} \rangle}^{-1}(i, j)} \dot{\rangle}}$. Looking at the definition, we see that $\zeta_{c_{i_{0}, j_{0}}}$ is applied to our block, and then all the blocks (including those of $a$ and $b_{i}$, which are permuted by $(1, \sigma_{\check{a}})$ before this) are permuted by $(1, \sigma_{\mu(\check{a}, \check{b}_{i})})$. This means that the input blocks of $c_{i,j}$ are permuted by $\sigma_{\mu(\check{a}, \check{b}_{i})}$. This means that our block lands on the block of $h(c_{i_{0}, j_{0}})$ in $\dot{\langle} \dot{\langle} f(a), g(b_{\sigma_{\check{a}}^{-1}(i)}) \dot{\rangle}, h(c_{\xi_{1}^{-1}(i, j)}) \dot{\rangle}$.

But we know that $\xi_{1} = \xi_{2}$. Therefore the block permutations are equal. We have also seen that $\zeta_{c_{i_{0}, j_{0}}}$ is applied to our block in both cases. Thus both permutations are equal.

\paragraph{The pentagon identity.} Since we have established naturality, we can apply the separation principle \ref{separationprinciple}. All the functors in the pentagon diagram are jointly agreeable and separated -- their values on prone morphisms (which are constructed by pullback, as in $\mathbf{Sig}_{a}$) are strict. To see separability consider a term 

\begin{displaymath}
\dot{\langle} a, \dot{\langle} b_{i}, \dot{\langle} c_{i, j}, d_{i, j, k} \dot{\rangle} \dot{\rangle} \dot{\rangle}
\end{displaymath}
label the inputs of $a$ by consecutive natural numbers, starting with $0$, then label the inputs of $b_{1}$ with consecutive numbers after the ones used for $a$. Repeat this process until the last $d_{i, j, k}$ is reached. Label the outputs so as to maintain composability. This defines the needed lift. By the separation principle we are reduced to checking the pentagon identity on function symbols.

After some amount of calculation we find that we must compare

\begin{displaymath}
\dot{\langle} \dot{\langle} \dot{\langle} a, b_{i} \dot{\rangle}, c_{\gamma_{\langle \check a, \check b_{i} \rangle}^{-1}(i, j)} \dot{\rangle},  d_{(1 \otimes \gamma)_{\langle \check a, \langle \check b_{i}, \check c_{i, j} \rangle \rangle}^{-1} \circ \gamma_{\langle \check a, \mu(\check{b}_{i}, \check c_{i, j}) \rangle}^{-1} (i, j, k)} \dot{\rangle}
\end{displaymath}
and\footnote{Is five levels of indexing baroque enough?}

\begin{displaymath}
\dot{\langle} \dot{\langle} \dot{\langle} a, b_{i} \dot{\rangle}, c_{\gamma_{\langle \check a, \check b_{i} \rangle}^{-1}(i, j)} \dot{\rangle}, d_{(\gamma \otimes 1)_{\langle \dot{\langle} \check a, \check b_{i} \dot{\rangle}, \check c_{i, j} \rangle}^{-1} \circ \gamma_{\langle \mu(\check{a}, \check{b}_{i}), \check c_{\gamma_{\langle \check a, \check b_{i} \rangle}^{-1}(i, j)} \rangle}^{-1}(i, j, k)} \dot{\rangle}
\end{displaymath}
which comes down to the equality

\begin{displaymath}
(1 \otimes \gamma)_{\langle \check a, \langle \check b_{i}, \check c_{i, j} \rangle \rangle}^{-1} \circ \gamma_{\langle \check a, \mu(\check b_{i}, \check c_{i, j}) \rangle}^{-1} = (\gamma \otimes 1)_{\langle \langle \check a, \check b_{i} \rangle, \check c_{i, j} \rangle}^{-1} \circ \gamma_{\langle \mu(\check{a}, \check b_{i}), \check c_{\gamma_{\langle \check a, \check b_{i} \rangle}^{-1}(i, j)} \rangle}^{-1}
\end{displaymath}
This equality is satisfied by the amalgamation permutations of $\mu$ by virtue of associativity.

\subsection{Distributivity for Monoidal Signatures}
We can now define the distributivity structure on $\mathbf{Sig}_{ma}$ which will give us the web monoids used in the definition of opetopic sets. To define a distributivity structure we need only specify $\varphi_{A, B, X}$ and $\psi_{X}$, which satisfy certain coherence conditions.

By definition, $\varphi_{A, B, X} : (A \otimes X) \odot (B \otimes X) \rightarrow (A \odot B) \otimes X$ maps the term

\begin{displaymath}
\dot{\langle} \langle a, x_{0,1}, \dots, x_{0, l_{0}} \rangle, \langle b_{1}, x_{1,1}, \dots, x_{1, l_{1}} \rangle, \dots, \langle b_{k}, x_{k, 1}, \dots, x_{k, l_{k}} \rangle \dot{\rangle}
\end{displaymath}
to

\begin{displaymath}
\langle \dot{\langle} a, b_{1}, \dots, b_{k} \dot{\rangle}, x_{0, 1}, \dots, x_{0, l_{0}}, x_{1, 1}, \dots, x_{1, l_{1}}, \dots, x_{k, 1}, \dots, x_{k, l_{k}} \rangle
\end{displaymath}
with trivial amalgamation permutations.

$\psi_{X}: I_{\odot} \rightarrow I_{\odot} \otimes X$ is defined by

\begin{displaymath}
1_{o} \mapsto \langle 1_{o}, - \rangle
\end{displaymath}
for $o \in O$ (the set of types of $M$). The $(-)$ represents an empty list, since the vertical inputs of elements of $I_{\odot}$ are empty.

\begin{theorem}\label{distributivityforsigma}
The above definitions give $\mathbf{Sig}_{ma}$ a distributivity structure of $\otimes$ over $\odot$.
\end{theorem}
The proof of this theorem consists of checking the coherence conditions listed after the definition of a distributivity structure (and drawn as diagrams in appendix \ref{appendixa}). This is done in appendix \ref{coherencefordistributivity}.

Using this theorem and the main theorem \ref{fiberedttt} we obtain a section $\mathcal{W}: Mon(\mathbf{Sig}_{a}) \rightarrow Mon(\mathbf{Sig}_{ma}, \otimes)$ of the fibration of $\otimes$-monoids in monoidal signatures. It associates to each ordinary monoid $M \in Mon(\mathbf{Sig}_{a})$ the corresponding web monoid $\mathcal{W}(M) \in Mon(\mathbf{Sig}_{ma}, \otimes)$. Since $(\mathbf{Sig}_{ma}, \otimes)$ is a pullback of $(\mathbf{Sig}_{a}, \otimes)$, the web monoid is naturally a monoid in $\mathbf{Sig}_{a}$ (we can forget that $M$ is a monoid after forming $\mathcal{W}(M)$). If we wish to emphasize the difference between these two structures we will write $|\mathcal{W}(M)|$ for the image in $\mathbf{Sig}_{a}$. We will not be consistent about this, however.

The name ``web monoid'' comes from the fact that its elements look like webs, at least in some circumstances. In any case the pasting diagram monoids $\mathbf{S}_{n}$ (defined below) can catch flies of dimension $n + 1$.

\section{The Category of Opetopic Sets}

An opetopic set $X$ is given by

\begin{enumerate}
\item A sequence of objects $X_{n} \in \mathbf{Sig}_{a}/X_{n-1}$, each in the fiber over the previous one (considered as a set), for $n \in \mathbb{N}_{> 0}$. By definition $X_{0}$ is a set.
\item A sequence of monoids $\mathbf{S}_{n} \in Mon(\mathbf{Sig}_{a})/X_{n}$, for $n \in \mathbb{N}$.
\item A sequence of strict morphisms $X_{n+1} \xrightarrow{\vartheta_{n}^{X}} \mathbf{S}_{n}$ in $\mathbf{Sig}_{a}/X_{n}$, equipped with prone arrows (in $Mon(\mathbf{Sig}_{a})$)
\end{enumerate}

\begin{center}
\begin{tikzpicture}
\matrix (m) [matrix of math nodes, row sep = 1cm, column sep = 2cm, text height = 1.5ex, text depth = .25ex] {
\mathbf{1} & \mathbf{S}_{0} && \mathcal{W}(\mathbf{S}_{n}) & \mathbf{S}_{n+1} \\
\lbrace \ast \rbrace^{\dag} & X_{0}^{\dag} && S_{n}^{\dag} &  X_{n+1}^{\dag}\\
};

\path[->] (m-1-2) edge node[auto, swap] {$\xi_{-1}^{X}$} (m-1-1)
			 (m-1-2) edge node[auto] {$\partial^{\mathbf{S}_{0}}$} (m-2-2)
			 (m-2-2) edge node[auto, swap] {$(\exists !)^{\dag}$} (m-2-1)
			 (m-1-1) edge  (m-2-1)
			 (m-1-5) edge node[auto, swap] {$\xi_{n}^{X}$} (m-1-4)
			 (m-1-5) edge node[auto] {$\partial^{\mathbf{S}_{n+1}}$} (m-2-5)
			 (m-2-5) edge node[auto, swap] {$(\vartheta_{n}^{X})^{\dag}$} (m-2-4)
			 (m-1-4) edge node[auto, swap] {$\partial^{\mathcal{W}(\mathbf{S}_{n})}$} (m-2-4);

\end{tikzpicture}
\end{center}

A morphism $f: X \rightarrow Y$ of opetopic sets is a sequence of functions $f_{n}: X_{n} \rightarrow Y_{n}$, for $n \in \mathbb{N}$, such that:

\begin{enumerate}
\item The morphism $f_{n+1}: X_{n+1} \rightarrow Y_{n+1}$ is well defined as a strict morphism in $\mathbf{Sig}_{a}$ over $f_{n}$, for $n \in \mathbb{N}_{> 0}$.
\item The induced homomorphisms $\bar{f}_{n}: \mathbf{S}_{n} \rightarrow \mathbf{T}_{n}$ make the following diagrams commute
\end{enumerate}

\begin{center}
\begin{tikzpicture}[overline/.style={preaction = {draw = white, -, line width = 6pt}}]
\matrix (m) [matrix of math nodes, row sep = 1.5cm, column sep = 1.5cm, text height = 1.5ex, text depth = .25ex] {
& \mathbf{1} && \mathbf{S}_{0} \\
\mathbf{1} && \mathbf{T}_{0} \\
& \lbrace \ast \rbrace ^{\dag} && X_{0}^{\dag} \\
\lbrace \ast \rbrace ^{\dag} && Y_{0}^{\dag} \\
};

\path[->] (m-1-4) edge node[auto, swap] {$\xi_{-1}^{X}$} (m-1-2)
			 (m-1-2) edge node[auto, swap] {$1$} (m-2-1)
			 (m-1-4) edge node[auto] {$\bar{f}_{0}$} (m-2-3)
			 (m-1-4) edge node[auto] {$\partial^{\mathbf{S}_{0}}$} (m-3-4)
			 (m-3-4) edge node[auto] {$f_{0}^{\dag}$} (m-4-3)
			 (m-4-3) edge node[auto, swap] {$(\exists !)^{\dag}$} (m-4-1)
			 (m-1-2) edge node[auto] {} (m-3-2)
			 (m-3-2) edge node[auto, swap] {$1$} (m-4-1)
			 (m-3-4) edge node[auto] (zuck) {} node[swap, left = 1cm of zuck, auto] {$(\exists !)^{\dag}$} (m-3-2)
			 (m-2-1) edge node[auto] {} (m-4-1)
			 (m-2-3) edge[overline] node[auto] (ruck) {} node[above = 1cm of ruck, auto] {$\partial^{\mathbf{T}_{0}}$} (m-4-3)
			 (m-2-3) edge[overline] node[auto, swap] (buck) {} node[swap, right = 1cm of buck, auto] {$\xi_{-1}^{Y}$} (m-2-1);

\end{tikzpicture}
\end{center}

\begin{center}
\begin{tikzpicture}[overline/.style={preaction = {draw = white, -, line width = 6pt}}]
\matrix (m) [matrix of math nodes, row sep = 1.5cm, column sep = 1.5cm, text height = 1.5ex, text depth = .25ex] {
& \mathcal{W}(\mathbf{S}_{n}) && \mathbf{S}_{n+1} \\
\mathcal{W}(\mathbf{T}_{n}) && \mathbf{T}_{n+1} \\
& S_{n}^{\dag} && X_{n+1}^{\dag} \\
T_{n}^{\dag} && Y_{n+1}^{\dag} \\
};

\path[->] (m-1-4) edge node[auto, swap] {$\xi_{n}^{X}$} (m-1-2)
			 (m-1-2) edge node[auto, swap] {$\mathcal{W}(\bar{f}_{n})$} (m-2-1)
			 (m-1-4) edge node[auto] {$\bar{f}_{n+1}$} (m-2-3)
			 (m-1-4) edge node[auto] {$\partial^{\mathbf{S}_{n+1}}$} (m-3-4)
			 (m-3-4) edge node[auto] {$f_{n+1}^{\dag}$} (m-4-3)
			 (m-4-3) edge node[auto, swap] {$(\vartheta_{n}^{Y})^{\dag}$} (m-4-1)
			 (m-1-2) edge node[auto] (buck) {} node[swap, below = 1cm of buck, auto] {$\partial^{\mathcal{W}(\mathbf{S}_{n})}$} (m-3-2)
			 (m-3-2) edge node[auto, swap] {$\bar{f}_{n}^{\dag}$} (m-4-1)
			 (m-3-4) edge node[auto] (zuck) {} node[swap, left = 1cm of zuck, auto] {$(\vartheta_{n}^{X})^{\dag}$} (m-3-2)
			 (m-2-1) edge node[auto, swap] {$\partial^{\mathcal{W}(\mathbf{T}_{n})}$} (m-4-1)
			 (m-2-3) edge[overline] node[auto] (ruck) {} node[above = 1cm of ruck, auto] {$\partial^{\mathbf{T}_{n+1}}$} (m-4-3)
			 (m-2-3) edge[overline] node[auto, swap] (hook) {} node[swap, right = 1cm of hook, auto] {$\xi_{n}^{Y}$} (m-2-1);

\end{tikzpicture}
\end{center}
and

\begin{center}
\begin{tikzpicture}
\matrix (m) [matrix of math nodes, row sep = 1cm, column sep = 1cm, text height = 1.5ex, text depth = .25ex] {
\mathbf{S}_{n} & X_{n+1} \\
\mathbf{T}_{n} & Y_{n+1} \\
};

\path[->] (m-1-2) edge node[auto, swap] {$\vartheta_{n}^{X}$} (m-1-1)
			 (m-1-2) edge node[auto] {$f_{n+1}$} (m-2-2)
			 (m-2-2) edge node[auto, swap] {$\vartheta_{n}^{Y}$} (m-2-1)
			 (m-1-1) edge node[auto, swap] {$\bar{f}_{n}$} (m-2-1);

\end{tikzpicture}
\end{center}
considered in $\mathbf{Set}$, for $n \in \mathbb{N}$

\section{Comparison with ``Polynomial functors and opetopes''}\label{comparisonsection}

We will now compare the web monoid for monoidal signatures with the ``Baez-Dolan slice construction'' as defined in \cite{KJBM}. The answer is that $|\mathcal{W}(M)|$ is the result of the slice construction on $M$. To formally state the answer we need to recall some facts and definitions from \cite{Zawadowski}.

\paragraph{Polynomial endofunctors.} Let $O \in \mathbf{Set}$. A polynomial functor over $O$ is a finitary functor preserving wide pullbacks $\mathbf{Set}/O \rightarrow \mathbf{Set}/O$. A morphism of polynomial functors is a cartesian natural transformation. This gives us a category $\mathbf{Poly}(O)$ of polynomial functors over $O$. These categories naturally assemble into a fibration of polynomial functors over $\mathbf{Set}$, denoted $\mathbf{Poly} \rightarrow \mathbf{Set}$.

The identity functor is polynomial and polynomial endofunctors can be composed. This give $\mathbf{Poly}$ the structure of a strict monoidal fibration. In fact it is a monoidal subfibration of the exponential fibration $Exp(\mathbf{Set})$ (the exponential object in $\mathbf{Cat}/\mathbf{Set}$ of the codomain fibration with itself).

Monoids in this fibration are exactly the polynomial monads -- monads whose underlying functor is polynomial and whose structure morphisms are cartesian.

\paragraph{The action of $\mathbf{Sig}_{a}$ on the basic fibration.} An action of a monoidal fibration on another fibration is defined exactly like an action of a monoidal category on another category, but everything is fibered (functors too -- we still do not require preservation of prone morphisms). The definition is spelled out in detail in \cite{Zawadowski}. Signatures with amalgamation act on the basic fibration $cod: \mathbf{Set}^{\cdot \rightarrow \cdot} \rightarrow \mathbf{Set}$ as follows. If $A \in \mathbf{Sig}_{a}/O$ and $(X \rightarrow O) \in \mathbf{Set}/O$ then we define

\begin{displaymath}
A \star X = \{ (a, x_{1}, \dots, x_{|a|}) : \partial_{a}^{A}(i) = d(x_{i}) \textnormal{ for } i = 1, \dots, |a| \}
\end{displaymath}
again, we allow $|a| = 0$ above. The structure morphism $d^{\star}: A \star X \rightarrow O$ is given by
\begin{displaymath}
d^{\star}(a, x_{1}, \dots, x_{|a|}) = \partial_{a}^{A}(0)
\end{displaymath}

The action of $\star$ on morphisms is defined in an obvious way. This defines an action

\begin{center}
\begin{tikzpicture}
\matrix (m) [matrix of math nodes, column sep = .3cm, row sep = 1cm, text height = 1.5ex, text depth = .25ex]{
\mathbf{Sig}_{a} \times_{\mathbf{Set}} \mathbf{Set}^{\cdot \rightarrow \cdot} && \mathbf{Set}^{\cdot \rightarrow \cdot} \\
& \mathbf{Set}\\
};

\path[->] (m-1-1) edge node[auto] {$\star$} (m-1-3)
			 (m-1-1) edge node[auto,swap] {$p_{a} \times cod$} (m-2-2)
			 (m-1-3) edge node[auto] {$cod$} (m-2-2);

\end{tikzpicture}
\end{center}
The exponential adjoint of this action $rep_{a}: \mathbf{Sig}_{a} \rightarrow Exp(\mathbf{Set})$ is given by $A \mapsto A \star (-)$. It is described by the following theorem 

\begin{theorem}\label{sigaispoly}
$rep_{a}: \mathbf{Sig}_{a} \rightarrow \mathbf{Poly}$ is an equivalence of monoidal fibrations.
\end{theorem}
The proof is given in \cite[section 6]{Zawadowski}. We will often identify a signature with its associated functor.

Note that if we forget that an object of $\mathbf{Sig}_{a}$ has inputs (but not the output!) we obtain an object of $ \mathbf{Set}^{\cdot \rightarrow \cdot}$, since $O^{\dag} \simeq O \times O^{\ast}$ gives a decomposition of the typing into (output, inputs) -- forgetting the second factor leaves us with a map $A \rightarrow O$, which is an object of $\mathbf{Set}^{\cdot \rightarrow \cdot}$. This defines a fibered forgetful functor $U: \mathbf{Sig}_{a} \rightarrow \mathbf{Set}^{\cdot \rightarrow \cdot}$

We can also construct a fibered functor $\overline{-}: \mathbf{Set}^{\cdot \rightarrow \cdot} \rightarrow \mathbf{Sig}_{a}$, which defines the vertical inputs to be empty. This is neither a left nor right adjoint to $U$. We will call $\overline{X}$ the \emph{sterile signature} associated to $X$. This gives us a useful formula for the action.

\begin{lemma}\label{formulaforaction}
$A \star X = U(A \otimes \overline{X})$.
\end{lemma}
\begin{proof}
We have equality of the defining formulas for both sides.
\end{proof}
In fact, \emph{defining} the action this way might be a good idea.

Note that there are no morphisms from between sterile and non-sterile signatures, and that $\overline{-}$ is fully faithful.
\begin{corollary}\label{actiononcolimits}
$\star$ preserves coproducts in the left variable and filtered colimits in both
\end{corollary}
\begin{proof}
Looking at lemma \ref{formulaforaction} we see that $\overline{-}$ is cocontinuous and $\otimes$ preserves the listed colimits. $U$ trivially preserves coproducts and preserves filtered colimits by their construction, which was given in \ref{exacnessforsiga}.
\end{proof}

\paragraph{The slice construction from \cite{KJBM}.} There are no fibrations in \cite{KJBM}, and thus we will be forced to restrict our discussion to individual fibers. This means we loose functoriality of the web monoid in the argument monoid $M$. We will fix this in a moment, but first we will give the original construction.

Recall that the slice of a monoidal category by a monoid is naturally a monoidal category. Let $M \in Mon(\mathbf{Poly}(O))$ be a polynomial monad over $O$. We obtain a natural monoidal structure on $\mathbf{Poly}(O)/M$. Then monoids over $M$ are the same as monoids in $\mathbf{Poly}(O)/M$.

Note that free polynomial monads (free monoids) on a polynomial functor exist. They can be constructed using theorem \ref{freemonoidbjt}. An explicit construction can be found in \cite{KJBM}. The same is true for monoids $\mathbf{Poly}(O)/M$.

\begin{remark}\label{monadicityremark}
In fact $Mon(\mathbf{Poly}(O)/M)$ is monadic over $\mathbf{Poly}(O)/M$.
\end{remark}

\begin{lemma}\label{somerandomlemma}
$\mathbf{Poly}(O)/M$ is equivalent to $\mathbf{Set}/M$
\end{lemma}
Here we treat $M$ as the corresponding set of function symbols given by theorem \ref{sigaispoly}.
\begin{proof}
By theorem \ref{sigaispoly} $\mathbf{Poly}(O)/M$ is equivalent to $(\mathbf{Sig}_{a}/O)/M$. We will see in theorem \ref{basicequivalencethm} (when we restrict to fibers) that this category is equivalent to $\mathbf{Set}/M$.
\end{proof}
Of course in the above proof the fact that $M$ is a monoid plays no role. It is only needed to construct the monoidal structure on $\mathbf{Poly}(O)/M$.

We can now give the slice construction. Let $M \in Mon(\mathbf{Poly}(O))$ be a polynomial monad. The category $\mathbf{Poly}(O)/M$ is monoidal and has a free monoid functor. This gives rise to the free monoid monad $T_{M}: \mathbf{Poly}(O)/M \rightarrow \mathbf{Poly}(O)/M$. By lemma \ref{somerandomlemma} this is equivalent to a monad $M^{+}: \mathbf{Set}/M \rightarrow \mathbf{Set}/M$. This monad is polynomial (one can check this directly, but it will also follow from our results).

\begin{definition}
The \cite{KJBM} Baez-Dolan slice construction is the assignment $M \mapsto M^{+}$.
\end{definition}
By remark \ref{monadicityremark} $M^{+}$ is the ``operad for operads over $M$'', as it should be.

We can now state the comparison theorem.

\begin{theorem}\label{comparisontheorem}
For any $M \in Mon(\mathbf{Sig}_{a})$ we have an isomorphism
\begin{displaymath}
rep_{a}(|\mathcal{W}(M)|) \simeq rep_{a}(M)^{+}
\end{displaymath}
\end{theorem}
Using theorem \ref{sigaispoly} to identify monoids in $\mathbf{Sig}_{a}$ with polynomial monads we can write more clearly

\begin{displaymath}
|\mathcal{W}(M)| \simeq M^{+}
\end{displaymath}

Formation of the web monoid is therefore very nearly identical to the Baez-Dolan construction. However the construction of the web monoid is completely different. Since $(-)^{+}$ is not a functor, we can not say that the isomorphism in the theorem is natural. We will see in a moment that when $(-)^{+}$ is extended to a functor by working with fibrations, then the isomorphism is natural.

We will need to reformulate theorem \ref{comparisontheorem} in order to prove it. Note that the equivalence between $\mathbf{Poly}$ and $\mathbf{Sig}_{a}$ is monoidal, and hence induces an equivalence between polynomial monads and monoids in signatures. This in turn gives an equivalence of fibered slices of these fibrations:

\begin{displaymath}
\mathbf{Sig}_{a} \downdownarrows \mathcal{U}_{sig} \simeq \mathbf{Poly} \downdownarrows \mathcal{U}_{poly}
\end{displaymath}
over the equivalence of monoid fibrations. The fibers of the fibered slice of $\mathbf{Poly}$ are (by construction) exactly all the categories of the form $\mathbf{Poly}(O)/M$, where $M$ is a monoid in $\mathbf{Poly}(O)$. Thus we have assembled into a fibration all the categories used in the Baez-Dolan construction. This fibration has a free monoid monad, which will be denoted $(-)^{+}$. The $(-)$ stands for an argument from the base: $M^{+}(X)$ is the free monoid in the fiber over $M$ on $X$. This is the natural way to extend the Baez-Dolan construction into a functor by working with fibrations. It is also the same (up to equivalence) as the extension given by hand in \cite{KJBM}. At this point we will forget about polynomial functors and work exclusively with signatures.

Theorem \ref{basicequivalencethm} gives us an equivalence of fibration which on fibers is exactly the equivalence asserted in lemma \ref{somerandomlemma}. This allows us to carry out the Baez-Dolan construction in all fibers at once. For this we need the pullback action

\paragraph{The pullback action.} We defined $\mathbf{Sig}_{ma}$ as the pullback of $\mathbf{Sig}_{a}$ by the functor $Mon(\mathbf{Sig}_{a}) \rightarrow \mathbf{Set}$ which sends a monoid $M$ to its underlying set of function symbols. We constructed a monoidal structure on it by a corollary of lemma \ref{pullbackspreservealgebra}. But the lemma states much more -- any algebraic structure can be pulled back. In particular we can pull back an action of a lax monoidal fibration, for example the action of $\mathbf{Sig}_{a}$ on the basic fibration. Thus we obtain the following situation:

\begin{center}
\begin{tikzpicture}[overline/.style={preaction = {draw = white, -, line width = 6pt}}]
\matrix (m) [matrix of math nodes, column sep = -0.5cm, row sep = .8cm, text height = 1.5ex, text depth = .25ex]{
&& \mathcal{U}^{\ast} \mathbf{Set}^{\cdot \rightarrow \cdot} &&& \mathbf{Set}^{\cdot \rightarrow \cdot} \\
\\
\mathbf{Sig}_{ma} \times_{Mon(\mathbf{Sig}_{a})} \mathcal{U}^{\ast} \mathbf{Set}^{\cdot \rightarrow \cdot} &&& \mathbf{Sig}_{a} \times_{\mathbf{Set}} \mathbf{Set}^{\cdot \rightarrow \cdot} \\
& Mon(\mathbf{Sig}_{a}) &&& \mathbf{Set} \\
};

\path[->] (m-1-3) edge node[auto] {} (m-1-6)
			 (m-3-1) edge node[auto] {$\mathcal{U}^{\ast} \star$} (m-1-3.south west)
			 (m-3-1) edge node[auto] {} (m-4-2)
			 (m-4-2) edge node[auto] {$\mathcal{U}^{\ast}$} (m-4-5)
			 (m-3-4) edge node[auto] {$\star$} (m-1-6)
			 (m-3-4) edge node[auto] {} (m-4-5)
			 (m-1-6) edge node[auto] {$cod$} (m-4-5)
			 (m-1-3) edge node[auto] (sth) {} node[above right = .6cm of sth] {$\mathcal{U}^{\ast} cod$} (m-4-2)
			 (m-3-1) edge[overline] node[auto] {} (m-3-4);

\end{tikzpicture}
\end{center}
Again, the formula for $\mathcal{U}^{\ast} \star$ is the same as the one for $\star$, but the set of types (or the codomain of $d: X \rightarrow M$) forms a monoid in $\mathbf{Sig}_{a}$. We will denote the pullback action by $\star$ in the sequel, and also denote $\mathcal{U}^{\ast}U$ as $U$. This should not cause any confusion. Note that the formula of lemma \ref{formulaforaction} is still true for the pullback action, as is its corollary.

At this point it should be clear that the fiber of $\mathcal{U}^{\ast} \mathbf{Set}^{\cdot \rightarrow \cdot}$ over $M \in Mon(\mathbf{Sig}_{a})$ is isomorphic to $\mathbf{Set}/M$, the slice of $\mathbf{Set}$ over the set of function symbols of $M$. Theorem \ref{basicequivalencethm} gives an adjoint equivalence $\mathbf{Sig}_{a} \downdownarrows \mathcal{U} \simeq \mathcal{U}^{\ast} \mathbf{Set}^{\cdot \rightarrow \cdot}$, and hence we can view the monad $(-)^{+}$ as acting on the latter fibration. This allows us to state the fibered version of the comparison theorem.

\begin{theorem}[Comparison Theorem -- Fibered Version]\label{fiberedcomparisontheorem}
There is an isomorphism of monads $\mathcal{W}(-) \star (=) \simeq (-)^{+}(=)$, where $\mathcal{W}$ is the web monoid functor.
\end{theorem}

The original comparison theorem \ref{comparisontheorem} follows immediately from this one when we apply the forgetful functor in each fiber to the pullback action and return to the action of $\mathbf{Sig}_{a}$ on the fundamental fibration. This theorem makes it clear that the proper base category for these constructions is not $\mathbf{Set}$ but rather $Mon(\mathbf{Sig}_{a})$, and this cannot be easily seen without using fibrations.

This theorem also neatly summarizes the differences between our approach and that of \cite{KJBM}. Their construction takes place in two different categories (or fibrations). In each only one type of inputs is visible (vertical or horizontal in our terminology). We have found a third fibration which sees both kinds of inputs, and all the relevant structure in the two original fibrations. It is a somewhat amusing fact that $\mathbf{Sig}_{ma}$ knows what free monoids look like in those other fibrations.

The proof of theorem \ref{fiberedcomparisontheorem} takes the rest of this section.

\subsection{Proof of The Comparison Theorem}

The proof of theorem \ref{fiberedcomparisontheorem} will consist of establishing certain properties of the pullback action, and an alternative description of $\mathcal{U}^{\ast} \mathbf{Set}^{\cdot \rightarrow \cdot}$.

\paragraph{An alternative description of $\mathcal{U}^{\ast} \mathbf{Set}^{\cdot \rightarrow \cdot}$.} Our first result is that  $\mathbf{Sig}_{a} \downdownarrows \mathcal{U}$ (the fibered slice) is equivalent to $\mathcal{U}^{\ast} \mathbf{Set}^{\cdot \rightarrow \cdot}$. The proof requires some preliminary constructions.

\begin{lemma}
For any $M \in \mathbf{Sig}_{a}/O$ there is a bijection $\{ \textnormal{set maps } X \rightarrow M\} \simeq \{ \textnormal{strict morphisms } X \rightarrow M \textnormal{ over } O \}$.
\end{lemma}
\begin{proof}
To a function $X \rightarrow M$ we assign a strict morphism, with $X$ typed by the composition $X \rightarrow M \xrightarrow{\partial} O^{\dag}$. Conversely, if the morphism is strict, then the typing is defined by that formula, so we can forget it.
\end{proof}

The full subfibration of $\mathbf{Sig}_{a} \downdownarrows \mathcal{U}$ of strict objects is defined as follows. Recall that the objects of $\mathbf{Sig}_{a} \downdownarrows \mathcal{U}$ are morphisms $A \rightarrow \mathcal{U}(M)$ in $\mathbf{Sig}_{a}$ over some $O \in \mathbf{Set}$, where $M$ is a monoid in $\mathbf{Sig}_{a}/O$. An object is called strict if the morphism $A \rightarrow \mathcal{U}(M)$ is strict. This fibration will be denoted by $\mathbf{Sig}_{a} \downdownarrows \mathcal{U}_{str}$.

\begin{corollary}\label{somerandomcorollary}
The subfibration of $\mathbf{Sig}_{a} \downdownarrows \mathcal{U}$ of strict objects is isomorphic to $\mathcal{U}^{\ast} \mathbf{Set}^{\cdot \rightarrow \cdot}$.
\end{corollary}
\begin{proof}
The above lemma defines a bijection on objects. A morphism between strict objects has, by lemma \ref{strictnesslemma}, the same amalgamation permutations as the morphism in the base, and can therefore be regarded as a function. Conversely any function between strict objects can be made into a morphism by setting the amalgamation permutations to what lemma \ref{strictnesslemma} says they should be. These constructions are clearly inverse to each other.
\end{proof}

We will now show that the subfibration of $\mathbf{Sig}_{a} \downdownarrows \mathcal{U}$ of strict objects is in fact equivalent to $\mathbf{Sig}_{a} \downdownarrows \mathcal{U}$. We will use a functorial factorization for this purpose. This construction was first used in \cite{HMP} for monoids.

Consider a morphism $f: A \rightarrow B$ in $\mathbf{Sig}_{a}$. We will factor it into two morphisms $A \xrightarrow{\zeta_{f}} A[f] \rightarrow B$, with the first morphism an isomorphism and the second morphism strict. The construction is simple: $A[f]$ is the same set as $A$, but with typing defined by $\partial^{A} \circ \sigma^{-1}$, which means $\partial^{A[f]}(a) = \partial^{A}(a) \circ \sigma_{a}^{-1}$, where $\sigma$ are the permutations of $f$, and $\partial^{A}$ is the original typing.

We now set the morphism $\zeta_{f} : A \rightarrow A[f]$ to be the identity on function symbols and have permutations given by $\sigma$. Obviously it is an isomorphism. The second morphism acts as $f$ on the function symbols, but is strict.

Recall that a functorial factorization is a section of the composition functor $\mathcal{C}^{\cdot \rightarrow \cdot \rightarrow \cdot} \longrightarrow \mathcal{C}^{\rightarrow}$.

\begin{lemma}\label{functorialfactorization}
The above construction uniquely defines a functorial factorization on $\mathbf{Sig}_{a}$ and on $Mon(\mathbf{Sig}_{a})$.
\end{lemma}

\begin{proof}
Since the morphisms are factorized into an isomorphism followed by some other morphism, the middle of the factorization is uniquely defined by the commutativity conditions. Functoriality is then trivial. It is also easy to check that if we require the morphism $\zeta_{f}: A \rightarrow A[f]$ to be an isomorphism of monoids, then the middle factorization will also be a homomorphism if $f$ was one.
\end{proof}

In the second part of the argument the requirement that $f$ is a homomorphism is important. It is quite surprising that one cannot transport monoid structures in $\mathbf{Sig}_{a}$ along isomorphisms -- the permutations of the isomorphism may ruin associativity. The condition that transport is possible is easy to write down, and satisfied by the permutations of monoid homomorphisms.

This functorial factorization is not fibered in any good sense -- the first morphism is always over the identity, and the second is over whatever the original morphism was over.

Now let $f : A \rightarrow \mathcal{U}(M)$ be an object of $\mathbf{Sig}_{a} \downdownarrows \mathcal{U}$. Then $A[f] \rightarrow \mathcal{U}(M)$ is a strict object. Since the factorization was functorial, this defines a fibered functor $fct: \mathbf{Sig}_{a} \downdownarrows \mathcal{U} \rightarrow \mathbf{Sig}_{a} \downdownarrows \mathcal{U}_{str}$. There is also the obvious inclusion $i: \mathbf{Sig}_{a} \downdownarrows \mathcal{U}_{str} \hookrightarrow \mathbf{Sig}_{a} \downdownarrows \mathcal{U}$

\begin{theorem}
The above functors form an adjoint equivalence over $Mon(\mathbf{Sig}_{a})$.
\end{theorem}
\begin{proof}
Save the adjoint part, this is a purely formal consequence of having a functorial factorization which factors a morphism into an isomorphism followed by another morphism. Inclusion followed by factorization is the identity on $\mathbf{Sig}_{a} \downdownarrows \mathcal{U}_{str}$. A factorization followed by inclusion is isomorphic to the identity functor on $\mathbf{Sig}_{a} \downdownarrows \mathcal{U}$ by the following diagram:

\begin{center}
\begin{tikzpicture}
\matrix (m) [matrix of math nodes, column sep = 1cm, row sep = .8cm, text height = 1.5ex, text depth = .25ex] {
A & B \\
A[h] & B[k] \\
M & N\\
};

\path[->] (m-1-1) edge node[auto] {$f$} (m-1-2)
			 (m-1-1) edge node[auto, swap] {$\zeta_{h}$} (m-2-1)
			 (m-2-1) edge node[auto] {$fct(f)$} (m-2-2)
			 (m-1-2) edge node[auto] {$\zeta_{k}$} (m-2-2)
			 (m-2-1) edge node[auto, swap] {} (m-3-1)
			 (m-2-2) edge node[auto] {} (m-3-2)
			 (m-3-1) edge node[auto] {$u$} (m-3-2)
			 (m-1-1) edge[bend right = 70] node[auto, swap] {$h$} (m-3-1)
			 (m-1-2) edge[bend left = 70] node[auto] {$k$} (m-3-2);

\end{tikzpicture}
\end{center}
The components $\zeta_{h}$ of the functorial factorization form an isomorphism from the identity functor to the composite of factorization and inclusion.

For the adjunction we take the components $\zeta_{h}^{-1}$ to be the counit -- it is the identity on function symbols, so we only need to worry about its amalgamation permutations. The unit is the identity. The triangular identities are then state that the following two composites are the identity

\begin{eqnarray*}
i(X)  \xrightarrow{1}  i \circ fct \circ i(X) \xrightarrow{\zeta_{i(X)}^{-1}} i(X) \\
fct(X)  \xrightarrow{1} fct \circ i \circ fct(X) \xrightarrow{fct(\zeta_{X}^{-1})} fct(X)
\end{eqnarray*}
They are true, since $fct(\zeta_{h}^{-1})$ is the identity by lemma \ref{strictnesslemma} (or direct calculation), and $\zeta_{i(X)}^{-1}$ is the identity for strict objects $X$.
\end{proof}

Combining this theorem with corollary \ref{somerandomcorollary} we have

\begin{theorem}\label{basicequivalencethm}
There is an adjoint equivalence $\mathbf{Sig}_{a} \downdownarrows \mathcal{U} \rightarrow \mathcal{U}^{\ast} \mathbf{Set}^{\cdot \rightarrow \cdot}$.
\end{theorem}

\paragraph{A monoidal structure on $\mathcal{U}^{\ast} \mathbf{Set}^{\cdot \rightarrow \cdot}$.} The fibration $\mathbf{Sig}_{a} \downdownarrows \mathcal{U}$ is monoidal, and we have shown that it is equivalent to $\mathbf{Sig}_{a} \downdownarrows \mathcal{U}_{str} \simeq \mathcal{U}^{\ast} \mathbf{Set}^{\cdot \rightarrow \cdot}$. We can get a monoidal structure on the latter fibration by the following general construction

Let $\mathcal{C}, \mathcal{D}$ be categories (fibrations), and let $F: \mathcal{C} \rightarrow \mathcal{D}$, $G: \mathcal{D} \rightarrow \mathcal{C}$ be an adjoint equivalence of categories (fibrations) with counit and unit isomorphisms $\varepsilon: GF \rightarrow 1_{\mathcal{C}}$ and $\eta: 1_{\mathcal{D}} \rightarrow FG$. If $\mathcal{C}$ is monoidal, then we can make $F$ and $G$ into a monoidal equivalence using the following natural formulas:

\begin{eqnarray*}
I_{\mathcal{D}} & = & F(I_{\mathcal{C}}) \\
A \otimes_{\mathcal{D}} B & = & F(G(A) \otimes_{\mathcal{C}} G(B)) \\ \\
\alpha_{A, B, C}^{\mathcal{D}} & = & F(\varepsilon_{G(A) \otimes G(B)}^{-1} \otimes 1 \circ \alpha_{G(A), G(B), G(C)} \circ 1 \otimes \varepsilon_{G(A) \otimes G(B)}) \\
\lambda_{A}^{\mathcal{D}} & = & \eta_{A}^{-1} \circ F(\lambda_{G(A)}^{\mathcal{C}} \circ \varepsilon_{I} \otimes 1) \\
\rho_{A}^{\mathcal{D}} & = & \eta_{A}^{-1} \circ F(\rho_{G(A)}^{\mathcal{C}} \circ 1 \otimes \varepsilon_{I}) \\ \\
\phi_{A, B}^{2, F} & = & F(\varepsilon_{A} \otimes \varepsilon_{B}): F(A) \otimes F(B) = F(GF(A) \otimes GF(B)) \rightarrow F(A \otimes B) \\
\phi^{0, F} & = & 1_{F(I)} \\
\phi_{A, B}^{2, G} & = & \varepsilon_{G(A) \otimes G(B)}^{-1}: G(A) \otimes G(B) \rightarrow G(A \otimes B) = GF(G(A) \otimes G(B)) \\
\phi^{0, G} & = & \varepsilon_{I}^{-1} : I \rightarrow GF(I),
\end{eqnarray*}
where the last four items define monoidal structures on $F$ and $G$, respectively.

\begin{theorem}
The above construction defines a monoidal structure on $\mathcal{D}$, $F$ and $G$, for which $\epsilon$ and $\eta$ are monoidal transformations.
\end{theorem}
\begin{proof}
Exercise. Everything follows from naturality of various transformations except (all) diagrams involving $\eta$, where the triangular identities are also needed.
\end{proof}

We can calculate what this structure looks like in our case. For example the units are unchanged, since they have only unary function symbols. The associativity isomorphism is the following

\begin{displaymath}
\alpha_{A, B, C}(\langle a, \langle b_{i}, c_{i, j} \rangle \rangle) = \langle \langle a, b_{i} \rangle, c_{\gamma_{\langle \partial{a}, \partial{b}_{i} \rangle}^{-1}(i, j)} \rangle
\end{displaymath}
For $A, B, C$ in the fiber over $M$. $\partial$ denotes the structure morphisms to $M$, and $\gamma$ the amalgamation permutations of the multiplication map in $M$. This structure will be denoted by $\otimes$.

\begin{corollary}\label{uisstrictmonoidal}
The functor $U: (\mathbf{Sig}_{ma}, \odot) \rightarrow (\mathcal{U}^{\ast} \mathbf{Set}^{\cdot \rightarrow \cdot}, \otimes)$ is strict monoidal.
\end{corollary}
\begin{proof}
All the formulas for parts of both monoidal structures coincide.
\end{proof}

This gives an alternative construction of $\mathbf{Sig}_{ma}$ -- take $\mathbf{Sig}_{a} \downdownarrows \mathcal{U}$, strictify and add vertical inputs. The construction by pullback is significantly more efficient.

Now a (well-engineered) miracle happens. Consider the exponential adjoint $\hat{\star}: \mathcal{U}^{\ast} \mathbf{Set}^{\cdot \rightarrow \cdot} \rightarrow \underline{Hom}_{Mon(\mathbf{Sig}_{a})}(\mathbf{Sig}_{ma}, \mathcal{U}^{\ast} \mathbf{Set}^{\cdot \rightarrow \cdot})$ of the pullback action. We have the following

\begin{theorem}
This adjoint lifts to $(\odot, \otimes)$-monoidal functors:

\begin{center}
\begin{tikzpicture}
\matrix (m) [matrix of math nodes, column sep = 2cm, row sep = 1cm, text height = 1.5ex, text depth = 1ex] {
& \underline{Hom}_{Mon(\mathbf{Sig}_{a})}^{\odot, \otimes}(\mathbf{Sig}_{ma}, \mathcal{U}^{\ast} \mathbf{Set}^{\cdot \rightarrow \cdot}) \\
\mathcal{U}^{\ast} \mathbf{Set}^{\cdot \rightarrow \cdot} & \underline{Hom}_{Mon(\mathbf{Sig}_{a})}(\mathbf{Sig}_{ma}, \mathcal{U}^{\ast} \mathbf{Set}^{\cdot \rightarrow \cdot}) \\
};

\path[->] (m-2-1) edge node[auto, swap] {$\hat{\star}$} (m-2-2)
			 (m-1-2) edge node[auto] {} (m-2-2);
\path[->, dashed] (m-2-1) edge node[auto] {$\tilde{\star}$} (m-1-2); 

\end{tikzpicture}
\end{center}

\end{theorem}
\begin{proof}
By lemma \ref{formulaforaction} (which is still true for $\mathbf{Sig}_{ma}$ by pullback) we have 

\begin{displaymath}
\hat{\star} = U_{\ast} \circ R \circ \overline{-},
\end{displaymath}
where $R$ is the functor $X \mapsto (-) \otimes X$ for $\mathbf{Sig}_{ma}$, $U_{\ast} = \underline{Hom}_{Mon(\mathbf{Sig}_{a})}(1, U)$ is the action of $U$ by postcomposition, and $\overline{-}$ is the sterile signature functor.

Since we have a lift $\tilde{R}$ of $R$ to $\underline{End}_{Mon(\mathbf{Sig}_{a})}^{\odot}(\mathbf{Sig}_{ma})$, and $U$ is strict $(\odot, \otimes)$-monoidal we can define $\tilde{\star}$ by

\begin{displaymath}
\tilde{\star} = U_{\ast} \circ \tilde{R} \circ \overline{-}
\end{displaymath}
\end{proof}

Concretely this gives us the following natural isomorphisms:

\begin{eqnarray*}
(A \star X) \otimes (B \star X) & \xrightarrow{\phi_{A, B, X}} & (A \odot B) \star X \\
\langle ( a, x_{i,j} ), ( b_{1}, x_{i',j'} ), \dots, ( b_{k}, x_{i'',j''} ) \rangle & \mapsto &
( \dot{\langle} a, b_{1}, \dots, b_{k} \dot{\rangle}, x_{m, n} ) \\\\
I_{\odot} \star X & \rightarrow & I_{\otimes} \\
(1_{o}, - ) & \mapsto & 1_{o}, \\
\end{eqnarray*}
which are given by formally the same formulas as distributivity for $\mathbf{Sig}_{ma}$. They give each functor $(-) \star X$ the structure of a monoidal functor $(\mathbf{Sig}_{ma}/M, \odot) \rightarrow (\mathcal{U}^{\ast} \mathbf{Set}^{\cdot \rightarrow \cdot}/M, \otimes)$ where $X$ is over $M$.

\begin{corollary}\label{propertiesofthepullbackaction}
The pullback action has the following properties:
\begin{enumerate}
\item Every functor $(-) \star X$ maps $\odot$-monoids in $\mathbf{Sig}_{ma}/M$ to $\otimes$-monoids in $\mathcal{U}^{\ast} \mathbf{Set}^{\cdot \rightarrow \cdot}/M$, where $X$ is over $M$.
\item $\mathcal{F}_{\odot}(I_{\otimes}) \star X \simeq \mathcal{F}_{\otimes}(X)$. In particular this isomorphism maps multiplication to multiplication $\mu_{I_{\otimes}}^{\mathcal{F}_{\odot}} \star X \simeq \mu_{X}^{\mathcal{F}_{\otimes}}$, and the units and  counits: $\eta^{\mathcal{F}_{\odot}}_{I_{\otimes}} \star X \simeq \eta^{\mathcal{F}_{\otimes}}_{X}$ and $\varepsilon_{I_{\otimes}} \star X \simeq \varepsilon_{X}$.
\end{enumerate}
\end{corollary}
\begin{proof}
The first point is trivial. The second one follows from the formula for free monoids in theorem \ref{freemonoidbjt}, the fact that $\star$ preserves filtered colimits and coproducts in the left variable (pullback of corollary \ref{actiononcolimits}), and the fact that $\star$ is an action, which gives $I_{\otimes} \star X \simeq X$. Thus $(-) \star X$ maps the free $\odot$-monoid construction in $\mathbf{Sig}_{ma}$ to the free $\otimes$-monoid construction in $\mathcal{U}^{\ast} \mathbf{Set}^{\cdot \rightarrow \cdot}$. Combining these facts gives $\mathcal{F}_{\odot}(I_{\otimes}) \star X \simeq \mathcal{F}_{\otimes}(I_{\otimes} \star X) \simeq \mathcal{F}_{\otimes}(X)$ along with all the associated structure.
\end{proof}

\begin{proposition}\label{distributivitycompatibility}
Let $\varphi$ be the distributivity isomorphism in $\mathbf{Sig}_{ma}$, $\phi$ the isomorphism we defined above, and let $a$ be the associativity isomorphism for the pullback action. Then the following diagram commutes:
\begin{center}
\begin{tikzpicture}
\matrix (m) [matrix of math nodes, column sep = 1cm, row sep = 1.5cm, text height = 1.5ex, text depth = .25ex]{
\left[A \star (Y \star X) \right] \otimes \left[B \star (Y \star X) \right] & \left[(A \otimes Y) \star X \right] \otimes \left[ (B \otimes Y) \star X \right] \\
& \left[(A \otimes Y) \odot (B \otimes Y) \right] \star X\\
(A \odot B) \star (Y \star X) & \left[(A \odot B) \otimes Y \right] \star X \\
};

\path[->] (m-1-1) edge node[auto] {$a \otimes a$} (m-1-2)
			 (m-1-1) edge node[auto] {$\phi_{A, B, Y \star X}$} (m-3-1)
			 (m-1-2) edge node[auto] {$\phi_{A \otimes Y, B \otimes Y, X}$} (m-2-2)
			 (m-2-2) edge node[auto] {$\varphi_{A, B, Y} \star X$} (m-3-2)
			 (m-3-1) edge node[auto] {$a$} (m-3-2);

\end{tikzpicture}
\end{center}
\end{proposition}
\begin{proof}
Direct calculation. An entirely analogous calculation is done for diagram II in appendix \ref{coherencefordistributivity}.
\end{proof}

The above diagram is analogous to diagram II in appendix \ref{appendixa}. Indeed $\phi$ is in some sense $\varphi$ and $a$ is in some sense $\alpha^{\otimes}$ for $\mathbf{Sig}_{ma}$, just like $\star$ is in some sense $\otimes$ by lemma \ref{formulaforaction}. At this time we do not know how to make this analogy more precise, because in forming $a$ we need to know that the middle variable is in $\mathbf{Sig}_{ma}$, and the formula of \ref{formulaforaction} cannot possibly remember this.

\begin{theorem}
$M^{+}(-) \simeq \mathcal{W}(M) \star (-)$ as monoids, naturally in $M$.
\end{theorem}
\begin{proof}
$(-)^{+}$ acts as the free monoid monad on $\mathbf{Sig}_{a} \downdownarrows \mathcal{U}$, which is monoidally equivalent to $\mathcal{U}^{\ast} \mathbf{Set}^{\cdot \rightarrow \cdot}$ with the $\otimes$-structure, by theorem \ref{basicequivalencethm} and construction of $\otimes$. Thus $(-)^{+}$ is isomorphic to the free monoid monad on $\mathcal{U}^{\ast} \mathbf{Set}^{\cdot \rightarrow \cdot}$. By the second point of the above corollary the universe of the web monoid acts as the free $\otimes$-monoid monad on $\mathcal{U}^{\ast} \mathbf{Set}^{\cdot \rightarrow \cdot}$, and a natural isomorphism drops out.

We must check whether it is an isomorphism of monads. The units are mapped to each other by definition -- in both cases they are the unit of the same adjunction (for $\mathcal{W}$ this follows again from corollary \ref{propertiesofthepullbackaction}). This leaves multiplication. We must check if $\nu \star X \circ a$ is $\varepsilon_{\mathcal{F}_{\otimes}(X)} \simeq \varepsilon_{\mathcal{F}_{\odot}(I_{\otimes})} \star X$. Consider the following diagram (we abbreviate $\mathcal{W} = \mathcal{W}(M)$):

\hspace*{-1.7in}
\begin{tikzpicture}
\matrix (m) [matrix of math nodes, column sep = .4cm, row sep = 1.2cm, text height = 1.5ex, text depth = .25ex]{
\mathcal{F}_{\otimes}^{2} (X) \otimes \mathcal{F}_{\otimes}^{2} (X) & \left[(\mathcal{W} \otimes \mathcal{W}) \star X \right] \otimes \left[ (\mathcal{W} \otimes \mathcal{W}) \star X \right] & (\mathcal{W} \star X) \otimes (\mathcal{W} \star X) & \mathcal{F}_{\otimes}(X) \otimes \mathcal{F}_{\otimes}(X) \\
\left[\mathcal{W} \star (\mathcal{W} \star X) \right] \otimes \left[ \mathcal{W} \star (\mathcal{W} \star X) \right] & \left[(\mathcal{W} \otimes \mathcal{W}) \odot (\mathcal{W} \otimes \mathcal{W}) \right] \star X & (\mathcal{W} \odot \mathcal{W}) \star X \\
(\mathcal{W} \odot \mathcal{W}) \star (\mathcal{W} \star X) & \left[(\mathcal{W} \odot \mathcal{W}) \otimes \mathcal{W}\right] \star X \\
\mathcal{W} \star (\mathcal{W} \star X) & (\mathcal{W} \otimes \mathcal{W}) \star X & \mathcal{W} \star X \\
\mathcal{F}_{\otimes}^{2} (X) \\
};

\path[->] (m-1-2) edge node (place) {} node[above = .2 of place] {$\underbrace{(\nu \star X) \otimes (\nu \star X)}$} (m-1-3)
			 (m-4-1) edge node[auto] {$\simeq$} (m-5-1)
			 (m-1-1) edge (m-1-2)
			 (m-5-1) edge (m-4-2)
			 (m-3-1) edge node[auto] {$a$} (m-3-2)
			 (m-3-1) edge node[auto] {$\mu \star 1$} (m-4-1)
			 (m-4-1) edge node[auto] {$a$} (m-4-2)
			 (m-2-1) edge node[auto] {$\phi$} (m-3-1)
			 (m-1-1) edge node[auto, swap] {$\simeq \otimes \simeq$} (m-2-1)
			 (m-2-1) edge node[auto] {$a \otimes a$} (m-1-2)
			 (m-1-3) edge node[auto] {$\simeq \otimes \simeq$} (m-1-4)
			 (m-1-4) edge node[auto] {$\mu_{X}^{\mathcal{F}_{\otimes}}$} (m-4-3)
			 (m-1-3) edge node[auto] {$\phi$} (m-2-3)
			 (m-1-2) edge node[auto, swap] {$\phi$} (m-2-2)
			 (m-2-2) edge node (pl) {} node[above = .2 of pl] {$(\nu \odot \nu) \star X$} (m-2-3)
			 (m-2-2) edge node[auto, swap] {$\varphi \star X$} (m-3-2)
			 (m-3-2) edge node[auto, swap] {$(\mu \otimes 1) \star X$} (m-4-2)
			 (m-2-3) edge node[auto] {$\mu \star X$} (m-4-3)
			 (m-4-2) edge node[auto] {$\nu \star X$} (m-4-3);

\draw[-, rounded corners] (m-1-1.south west) -- (m-2-1.north west) -- (m-2-1.south west);
\path[->] (m-2-1.south west) edge[bend right = 15] node[fill = white] {$\mu^{\mathcal{F}_{\otimes}}_{\mathcal{F}_{\otimes}(X)}$} (m-5-1.north west);

\end{tikzpicture}
The central square is just the main diagram starred with $X$. It commutes by the definition of $\nu$. The square above it commutes by naturality of $\phi$. The rightmost triangle commutes by corollary \ref{propertiesofthepullbackaction}, as does the leftmost ``bigon'' or ``biangle''. The trapezoid commutes by proposition \ref{distributivitycompatibility}. The square below it commutes by naturality. Since $\mathcal{F}_{\otimes}$ is determined only up to natural isomorphism, the isomorphisms marked $\simeq$ are irrelevant, and can be taken, for example, to be the identity.

Thus we see that $\nu \star X \circ a$ gives a natural homomorphism of monoids $\mathcal{F}_{\otimes}^{2}(X) \rightarrow \mathcal{F}_{\otimes}(X)$. Since $\mathcal{F}_{\otimes}$ is the free monoid monad, all such homomorphisms are determined by what they do to the unit $\eta_{\mathcal{F}_{\otimes}(X)}$. But by the unit conditions for $\nu$ we see that $\nu \star X \circ a \circ \eta_{\mathcal{F}_{\otimes}(X)}$ is the identity, as the following diagram shows

\begin{center}
\begin{tikzpicture}
\matrix (m) [matrix of math nodes, column sep = 1cm, row sep = 1.5cm, text height = 1.5ex, text depth = .25ex]{
\mathcal{F}_{\otimes}(X) & \mathcal{W} \star X & I_{\otimes} \star (\mathcal{W} \star X) & (I_{\otimes} \otimes \mathcal{W}) \star X \\
\mathcal{F}^{2}_{\otimes}(X) & & \mathcal{W} \star (\mathcal{W} \star X) & (\mathcal{W} \otimes \mathcal{W}) \star X \\
& & \mathcal{F}_{\otimes}(X) & \mathcal{W} \star X \\
};

\path[->] (m-1-1) edge node[auto] {$\simeq$} (m-1-2)
			 (m-1-2) edge node[auto] {} (m-1-3)
			 (m-1-3) edge node[auto] {$a$} (m-1-4)
			 (m-2-1) edge node[auto] {$\simeq$} (m-2-3)
			 (m-2-3) edge node[auto] {$a$} (m-2-4)
			 (m-1-1) edge node[auto] {$\eta_{\mathcal{F}_{\otimes}(X)}^{\mathcal{F}_{\otimes}}$} (m-2-1)
			 (m-1-3) edge node[auto,swap] {$\eta_{I_{\otimes}}^{\mathcal{F}_{\odot}} \star 1$} (m-2-3)
			 (m-1-4) edge node[auto,swap] {$(\eta_{I_{\otimes}}^{\mathcal{F}_{\odot}} \otimes 1) \star X$} (m-2-4)
			 (m-3-3) edge node[auto] {$\simeq$} (m-3-4)
			 (m-1-4.south east) edge[bend left = 40] node[auto] {$\lambda^{\odot}_{\mathcal{W}} \star X$} (m-3-4.north east)
			 (m-2-4) edge node[auto,swap] {$\nu \star X$} (m-3-4)
			 (m-2-1) edge[dashed] (m-3-3);

\end{tikzpicture}
\end{center}
The top left rectangle commutes by \ref{propertiesofthepullbackaction}. The right bigon commutes by the unit conditions for $\nu$. The top right square commutes by naturality of $a$. The isomorphisms $\simeq$ are again irrelevant, and can be taken to be identities. The unnamed isomorphism is the canonical one, given by the action $\star$. The dashed arrow is determined by the other composites, and is the identity, by the left unit condition for $\star$.

Only the counit can satisfy this equation, and thus $\nu \star X \circ a = \varepsilon_{\mathcal{F}_{\otimes}(X)}^{\mathcal{F}_{\otimes}}$ concluding the proof.
\end{proof}

It should be clear from this argument that distributivity really does tell us that $\nu$ commutes with $\mu$. This is made literal by the pullback action, as we saw above.

In fact we could check that all our identities hold before applying $(-) \star X$ (we have a natural isomorphism $\mathcal{F}_{\odot}^{2}(X) \simeq \mathcal{F}_{\odot}(X) \otimes \mathcal{F}_{\odot}(X)$), but the way we have set up the formalism, this would entail several more pages of direct calculation, and a theorem or two.

\appendix
\section{Proof of the Main Theorem} \label{appendixa}
\paragraph{The Coherence Conditions.} Here we will list, as promised, the coherence conditions arising from the definition of distributivity.
\begin{itemize}
\item Condition I

\begin{tikzpicture}
\matrix (m) [matrix of math nodes, column sep = .5cm, row sep = 1.5cm, text height = 1.5ex, text depth = .25ex]{
(A \otimes X) \odot ((B \otimes X) \odot (C \otimes X)) & ((A \otimes X) \odot (B \otimes X)) \odot (C \otimes X) \\
(A \otimes X) \odot (B \odot C) \otimes X & ((A \odot B) \otimes X) \odot (C \otimes X) \\
(A \odot (B \odot C)) \otimes X & ((A \odot B) \odot C) \otimes X \\
};

\path[->] (m-1-1) edge node[auto] {$\alpha^{\odot}$} (m-1-2)
			 (m-1-1) edge node[auto,swap] {$1 \odot \varphi_{B, C, X}$} (m-2-1)
			 (m-2-1) edge node[auto,swap] {$\varphi_{A, B \odot C, X}$} (m-3-1)
			 (m-1-2) edge node[auto] {$\varphi_{A, B, X} \odot 1$} (m-2-2)
			 (m-2-2) edge node[auto] {$\varphi_{A \odot B, C, X}$} (m-3-2)
			 (m-3-1) edge node[auto] {$\alpha^{\odot} \otimes 1$} (m-3-2);
\end{tikzpicture}

\item Condition II

\begin{tikzpicture}
\matrix (m) [matrix of math nodes,column sep = 1.2cm, row sep = 1.5cm, text height = 1.5ex, text depth = .25ex]{
(A \otimes (X \otimes Y)) \odot (B \otimes (X \otimes Y)) & ((A \otimes X) \otimes Y) \odot ((B \otimes X) \otimes Y) \\
& ((A \otimes X) \odot (B \otimes X)) \otimes Y \\
(A \odot B) \otimes (X \otimes Y) & ((A \odot B) \otimes X) \otimes Y \\
};

\path[->] (m-1-1) edge node[auto] {$\alpha^{\otimes} \odot \alpha^{\otimes}$} (m-1-2)
			 (m-1-1) edge node[auto,swap] {$\varphi_{A, B, X \otimes Y}$} (m-3-1)
			 (m-1-2) edge node[auto] {$\varphi_{A \otimes X, B \otimes X, Y}$} (m-2-2)
			 (m-2-2) edge node[auto] {$\varphi_{A, B, X} \otimes 1$} (m-3-2)
			 (m-3-1) edge node[auto] {$\alpha^{\otimes}$} (m-3-2);
\end{tikzpicture}

\item Condition III

\begin{center}
\begin{tikzpicture}
\matrix (m) [matrix of math nodes,column sep = 1cm, row sep = .5cm, text height = 1.5ex, text depth = .25ex]{
& A \odot B \\
(A \otimes I_{\otimes}) \odot (B \otimes I_{\otimes}) && (A \odot B) \otimes I_{\otimes} \\
};

\path[->] (m-1-2) edge node[auto] {$\rho^{\otimes}$} (m-2-3)
			 (m-1-2) edge node[auto,swap] {$\rho^{\otimes} \odot \rho^{\otimes}$} (m-2-1)
			 (m-2-1) edge node[auto] {$\varphi_{A, B, I_{\otimes}}$} (m-2-3);
\end{tikzpicture}
\end{center}

\item Condition IV, for any morphism $f: A\otimes X \rightarrow Y$

\begin{center}
\begin{tikzpicture}
\matrix (m) [matrix of math nodes,column sep = 1.5cm, row sep = 1cm, text height = 1.5ex, text depth = .25ex]{
(I_{\odot} \otimes X) \odot (A \otimes X) & I_{\odot} \odot Y \\
(I_{\odot} \odot A) \otimes X \\
A \otimes X & Y \\
};

\path[->] (m-1-1) edge node[auto] {$\psi_{X}^{-1} \odot f$} (m-1-2)
			 (m-1-1) edge node[auto] {$\varphi$} (m-2-1)
			 (m-2-1) edge node[auto,swap] {$\lambda^{\odot} \otimes 1$} (m-3-1)
			 (m-3-1) edge node[auto] {$f$} (m-3-2)
			 (m-1-2) edge node[auto] {$\lambda^{\odot}$} (m-3-2);
\end{tikzpicture}
\end{center}

\item Condition V

\begin{center}
\begin{tikzpicture}
\matrix (m) [matrix of math nodes, column sep = 1.2cm, row sep = .5cm, text height = 1.5ex, text depth = .25ex]{
I_{\odot} \otimes (X \otimes Y) & (I_{\odot} \otimes X) \otimes Y & I_{\odot} \otimes Y \\
& I_{\odot} \\
};

\path[->] (m-1-1) edge node[auto] {$\alpha^{\otimes}$} (m-1-2)
			 (m-1-2) edge node[auto] {$\psi_{X}^{-1} \otimes 1$} (m-1-3)
			 (m-1-1) edge node[auto,swap] {$\psi_{X \otimes Y}^{-1}$} (m-2-2)
			 (m-1-3) edge node[auto] {$\psi_{Y}^{-1}$} (m-2-2);
\end{tikzpicture}
\end{center}

\item Condition VI
\begin{displaymath}
\psi_{I_{\otimes}} = \rho_{I_{\odot}}^{\otimes}
\end{displaymath}

\item Condition VII
\begin{center}
\begin{tikzpicture}
\matrix (m) [matrix of math nodes, column sep = 1.5cm, row sep = 1cm, text height = 1.5ex, text depth = .25ex]{
(A \otimes X) \odot I_{\odot} & A \otimes X \\
(A \otimes X) \odot (I_{\odot} \otimes X) & (A \otimes I_{\odot}) \otimes X \\
};

\path[->] (m-1-2) edge node[auto, swap] {$\rho^{\odot}$} (m-1-1)
			 (m-1-1) edge node[auto, swap] {$1 \odot \psi_{X}$} (m-2-1)
			 (m-2-1) edge node[auto] {$\varphi_{A, I_{\odot}, X}$} (m-2-2)
			 (m-1-2) edge node[auto] {$\rho^{\odot} \otimes X$} (m-2-2);

\end{tikzpicture}
\end{center}
\end{itemize}
The conditions are listed in the form in which they are used in the proof. They follow from the definition of distributivity: first, because $\tilde{R}(f)$ is $\odot$-monoidal for all morphisms it follows that $\phi$ and $\psi$ giving the monoidal structure for $\tilde{R}(X)$ are natural in X. Conditions $I$, $IV$ and $VII$ say that $\tilde{R}(X)$ is $\odot$-monoidal. Condition $IV$ is the left unit condition combined with the naturality of $\lambda^{\odot}$. Conditions $III$ and $VI$ are the requirement that $\rho :1_{\mathcal{C}} \rightarrow \tilde{R}(I_{\otimes})$ is an $\odot$-monoidal natural transformation. Conditions $II$ and $V$ say the same for $(\alpha^{\otimes})^{-1} : \tilde{R}(X) \circ \tilde{R}(Y) \rightarrow \tilde{R}(X \otimes Y)$.

\begin{remark}
Condition $VII$ will not be used in the proof. Thus the main theorem is true if we lift $R$ to $End_{\odot}^{L}(\mathcal{C})$, the category of left-unital $\odot$-monoidal functors.
\end{remark}

\paragraph{Determination of $\nu_{1}$ and $\nu_{0}$.} In this paragraph we still suppose that $\nu$ exists. We must proove that $\nu_{1} = (i_{0} \psi_{\mathcal{W}}^{-1}, \lambda_{\mathcal{W}}^{\otimes})$ as claimed above. This follows from the unit conditions. They are
\begin{center}
\begin{tikzpicture}
\matrix (m) [matrix of math nodes, column sep = 1cm, row sep = 1cm, text height = 1.5ex, text depth = .25ex]{
I_{\otimes} \otimes \mathcal{W} & \mathcal{W} \otimes \mathcal{W} & \mathcal{W} \otimes I_{\otimes} \\
&  \mathcal{W} \\
};

\path[->] (m-1-1) edge node[auto] {$\eta \otimes 1$} (m-1-2)
			(m-1-3) edge node[auto,swap] {$1 \otimes \eta$} (m-1-2)
			(m-1-2) edge node[auto] {$\nu$} (m-2-2)
			(m-1-1) edge node[auto,swap] {$\lambda^{\otimes}$} (m-2-2)
			(m-1-3) edge node[auto] {$(\rho^{\otimes})^{-1}$} (m-2-2);
\end{tikzpicture}
\end{center}
We can expand them to the following commutative diagram

\begin{center}
\begin{tikzpicture}
\matrix (m) [matrix of math nodes, column sep = 1.5cm, row sep = 1cm, text height = 1.5ex, text depth = .25ex]{
\mathcal{W}_{1} \otimes \mathcal{W} &&& \mathcal{W}_{1} \otimes \mathcal{W}_{1} \\
I_{\otimes} \otimes \mathcal{W} & \mathcal{W} \otimes \mathcal{W} & \mathcal{W} \otimes I_{\otimes} & \mathcal{W}_{1} \otimes I_{\otimes} \\
& \mathcal{W} \\
};

\path[->] (m-1-4) edge node[auto,swap] {$1 \otimes i_{1}$} (m-1-1)
			 (m-1-1) edge node[auto] {$i_{1} \otimes 1$} (m-2-2)
			 (m-2-1) edge node[auto] {$j \otimes 1$} (m-1-1)
			 (m-2-1) edge node[auto,swap] {$\lambda$} (m-3-2)
			 (m-1-1) edge[densely dotted] node[auto] {$\nu_{1}$} (m-3-2)
			 (m-2-4) edge node[auto] {$1 \otimes j$} (m-1-4)
			 (m-2-4) edge node[auto,swap] {$i_{1}\otimes 1$} (m-2-3)
			 (m-2-3) edge node[auto,swap] {$1 \otimes \eta$} (m-2-2)
			 (m-2-2) edge node[auto] {$\nu$} (m-3-2)
			 (m-2-4) edge node[auto] {$(\rho^{\otimes})^{-1} (i_{1} \otimes 1)$} (m-3-2);
\end{tikzpicture}
\end{center}
where some of the original maps was omitted for readability. Maps labeled $j$ are coprojections of coproducts. They are factorizations of $\eta$ (hence the commutativity). We wish to determine the dotted arrow $\nu_{1}$. We note that
\begin{displaymath}
\mathcal{W}_{1} \otimes \mathcal{W} \simeq (I_{\odot} \sqcup I_{\otimes}) \otimes \mathcal{W} \simeq I_{\odot} \otimes \mathcal{W} \sqcup I_{\otimes} \otimes \mathcal{W}
\end{displaymath}
Thus the map $\nu_{1}$ is determined by what happens on both of these components. The left unit condition immediately implies that the right component is mapped to $\mathcal{W}$ by $\lambda_{\mathcal{W}}^{\otimes}$. To see what happens to $I_{\odot} \otimes \mathcal{W}$ consider the top map composed with $1 \otimes j$ and the inclusion of $I_{\odot} \otimes I_{\otimes}$ into $\mathcal{W}_{1} \otimes I_{\otimes}$. En easy calculation gives that this is $1 \otimes \eta: I_{\odot} \otimes I_{\otimes} \rightarrow I_{\odot} \otimes \mathcal{W}$ followed by the inclusion $I_{\odot} \otimes \mathcal{W} \rightarrow \mathcal{W}_{1} \otimes \mathcal{W}$. We now look at the right unit condition. We obtain the diagram

\begin{center}
\begin{tikzpicture}
\matrix (m) [matrix of math nodes, column sep = 2cm, row sep = .8cm, text height = 1.5ex, text depth = .25ex]{
I_{\odot} & I_{\odot} \\
I_{\odot} \otimes \mathcal{W} & I_{\odot} \otimes I_{\otimes} \\
I_{\odot} \otimes \mathcal{W} \sqcup I_{\otimes} \otimes \mathcal{W} & I_{\odot} \otimes I_{\otimes} \sqcup I_{\otimes} \otimes I_{\otimes} \\
\mathcal{W} \\
};

\path[->] (m-3-2) edge node[auto,swap] {$1 \otimes \eta \sqcup 1 \otimes \eta$} (m-3-1)
			 (m-3-1) edge node[auto,swap] {$(?, \lambda_{\mathcal{W}}^{\otimes})$} (m-4-1)
			 (m-3-2) edge node[auto] {$i_{0} (\rho^{\otimes})^{-1} \sqcup (\rho^{\otimes})^{-1} \eta$} (m-4-1)
			 (m-1-1) edge node[auto,swap] {$\psi_{\mathcal{W}}$} (m-2-1)
			 (m-1-2) edge node[auto,swap] {$1$} (m-1-1)
			 (m-1-2) edge node[auto] {$\psi_{I_{\otimes}}$} (m-2-2)
			 (m-2-1) edge node[auto] {} (m-3-1)
			 (m-2-2) edge node[auto] {} (m-3-2)
			 (m-2-2) edge node[auto,swap] {$1 \otimes \eta$} (m-2-1);
\end{tikzpicture}
\end{center}
We want to determine the map ``$?$''. The unnamed maps are coprojections. On the right components of the coproducts this diagram commutes by naturality of $\lambda^{\otimes}$ and the condition $\lambda_{I_{\otimes}} = \rho_{I_{\otimes}}^{-1}$. The second component of the diagonal map is determined by the naturality of $\rho^{\otimes}$ applied to the inclusion $I_{\odot} = \mathcal{W}_{0} \rightarrow \mathcal{W}_{1} \rightarrow \mathcal{W}$. The top square commutes by naturality of $\psi$. From this we obtain the equation

\begin{displaymath}
? \circ \psi_{\mathcal{W}} \psi_{I_{\otimes}}^{-1} = i_{0} (\rho_{I_{\odot}}^{\otimes})^{-1}
\end{displaymath}
from which follows, using coherence condition VI, that  the map ``$?$'' is

\begin{displaymath}
i_{0} \circ \psi_{\mathcal{W}}^{-1}: I_{\odot} \otimes \mathcal{W} = \mathcal{W}_{0} \otimes \mathcal{W} \rightarrow \mathcal{W}.
\end{displaymath}
Note that these calculations also determine that $\nu_{0} = i_{0, \infty} \psi_{\mathcal{W}}^{-1}$, since this map is $\nu_{1}$ precomposed with the inclusion $I_{\odot} \otimes \mathcal{W} \rightarrow \mathcal{W}_{1} \otimes \mathcal{W}$, and we have just determined exactly this composite.

From now on we use the definition \ref{definitionofnu} for $\nu$ since we have already showed that it is the only possible choice. We still have to check that $\nu$ defines a monoid and makes the main diagram commute. We will intensely use induction - the first component will satisfy an appropriate equality, usually because of the coherence conditions, and then all others will follow by applying the inductive definition \ref{definitionofnun}. The original condition will be recovered by applying the colimit functor.

\paragraph{The unit conditions.} The left unit condition holds as part of our definition of $\nu$, since it factors through $\nu_{1}$, which was defined in part by this condition. This leaves the right unit condition. We will prove it using induction on $n$ starting with $n = 1$, which consists of the calculations above. For the inductive step we need to check that

\begin{displaymath}
\nu_{n} \circ 1 \otimes \eta = i_{n} (\rho^{\otimes})^{-1}.
\end{displaymath}
Consider the following diagram:

\begin{center}
\begin{tikzpicture}
\matrix (m) [matrix of math nodes, column sep = 1cm, row sep = 1cm, text height = 1.5ex, text depth = .25ex]{
(\mathcal{W}_{1} \otimes I_{\otimes}) \odot (\mathcal{W}_{n-1} \otimes I_{\otimes}) & (\mathcal{W}_{1} \otimes \mathcal{W}) \odot (\mathcal{W}_{n-1} \otimes \mathcal{W}) & \mathcal{W} \odot \mathcal{W} \\
& (\mathcal{W}_{1} \odot \mathcal{W}_{n-1}) \otimes \mathcal{W} \\
& \mathcal{W}_{n} \otimes \mathcal{W} & \mathcal{W} \\
(\mathcal{W}_{1} \odot \mathcal{W}_{n-1}) \otimes I_{\otimes} &  \mathcal{W}_{n} \otimes I_{\otimes}\\
};

\path[->] (m-1-1) edge node[auto] (left) {} (m-1-2)
			 (m-1-2) edge node[auto] (right) {} (m-1-3)
			 (m-1-3) edge node[auto] {$\mu$} (m-3-3)
			 (m-4-1) edge node[auto] {$\varphi_{1, n-1}^{-1}$} (m-1-1)
			 (m-4-1) edge node[auto] {$1 \otimes \eta$} (m-2-2)
			 (m-4-2) edge node[auto] {$s_{n} \otimes 1$} (m-4-1)
			 (m-4-2) edge node[auto] {$1 \otimes \eta$} (m-3-2)
			 (m-4-2) edge node[auto,swap] {$i_{n} \rho^{-1}$} (m-3-3)
			 (m-3-2) edge node[auto,swap] {$s_{n} \otimes 1$} (m-2-2)
			 (m-3-2) edge node[auto] {$\nu_{n}$} (m-3-3)
			 (m-2-2) edge node[auto] {$\varphi_{1, n-1}^{-1}$} (m-1-2);

\node[above = 0 of left] {$(1 \otimes \eta) \odot (1 \otimes \eta)$};
\node[above = 0 of right] {$\nu_{1} \odot \nu_{n-1}$};
\end{tikzpicture}
\end{center}
All the regions in it commute except possibly the small triangle below $\nu_{n}$, which we are investigating. This follows from the naturality of $\varphi$ (note the abbreviation we have introduced here) and the definition of $\nu_{n}$. From this and the inductive hypothesis we can calculate that

\begin{displaymath}
\nu_{n} 1 \otimes \eta = i_{n} \mu_{1, n-1} (\rho^{-1} \odot \rho^{-1}) \phi_{1, n-1}^{-1} (s_{n} \otimes 1)
\end{displaymath}
using the explicit definition of $\mu$ given in theorem \ref{freemonoidbjt}. Thus if we can check that

\begin{displaymath}
i_{n} \mu_{1, n-1} (\rho^{-1} \odot \rho^{-1}) \phi_{1, n-1}^{-1} (s_{n} \otimes 1) = i_{n} \rho^{-1},
\end{displaymath}
we would be done. But this comes down to the commutativity of

\begin{center}
\begin{tikzpicture}
\matrix (m) [matrix of math nodes, column sep = 1.5cm, row sep = 1cm, text height = 1.5ex, text depth = .25ex]{
(\mathcal{W}_{1} \otimes I_{\otimes}) \odot (\mathcal{W}_{n-1} \otimes I_{\otimes}) & \mathcal{W}_{1} \odot \mathcal{W}_{n-1} & \mathcal{W} \odot \mathcal{W} & \mathcal{W} \\
(\mathcal{W}_{1} \odot \mathcal{W}_{n-1}) \otimes I_{\otimes} & \mathcal{W}_{1} \odot \mathcal{W}_{n-1} & \mathcal{W}_{n} \\
\mathcal{W}_{n} \otimes I_{\otimes} & \mathcal{W}_{n} \\
};

\path[->] (m-1-1) edge node[auto] {$\rho^{-1} \odot \rho^{-1}$} (m-1-2)
			 (m-1-2) edge node[auto] {$i_{1} \odot i_{n-1}$} (m-1-3)
			 (m-1-3) edge node[auto] {$\mu$} (m-1-4)
			 (m-1-2) edge node[auto] {$\mu_{1, n-1}$} (m-2-3)
			 (m-2-1) edge node[auto] {$\varphi_{1, n-1}^{-1}$} (m-1-1)
			 (m-2-1) edge node[auto] {$\rho^{-1}$} (m-2-2)
			 (m-2-2) edge node[auto] {$1$} (m-1-2)
			 (m-3-1) edge node[auto] {$\rho^{-1}$} (m-3-2)
			 (m-3-1) edge node[auto] {$s_{n} \otimes 1$} (m-2-1)
			 (m-3-2) edge node[auto] {$s_{n}$} (m-2-2)
			 (m-3-2) edge node[auto] {$1$} (m-2-3)
			 (m-2-3) edge node[auto] {$i_{n}$} (m-1-4);
\end{tikzpicture}
\end{center}
which follows from the definition of $\mu$ (top triangle), naturality of $\rho^{\otimes}$ (bottom rectangle), coherence condition III (top rectangle), and proposition \ref{canonicalsection} -- the defining property of $s_{n}$ (bottom triangle).

\paragraph{Commutativity of the main diagram.} We will consider the diagram
\begin{center}
\begin{tikzpicture}
\matrix (m) [matrix of math nodes, column sep = 1.5cm, row sep = 1cm, text height = 1.5ex, text depth = .25ex]{
(\mathcal{W}_{n} \otimes \mathcal{W}) \odot (\mathcal{W}_{m} \otimes \mathcal{W}) & \mathcal{W} \odot \mathcal{W} \\
(\mathcal{W}_{n} \odot \mathcal{W}_{m}) \otimes \mathcal{W} & & \\
\mathcal{W}_{n+m} \otimes \mathcal{W} & \mathcal{W}\\
};

\path[->] (m-1-1) edge node[auto] {$\nu_{n} \odot \nu_{m}$} (m-1-2)
			(m-1-1) edge node[auto, swap] {$\varphi_{n, m}$} (m-2-1)
			(m-2-1) edge node[auto, swap] {$\mu_{n, m} \otimes 1_{\mathcal{W}}$} (m-3-1)
			(m-1-2) edge node[auto] {$\mu$} (m-3-2)
			(m-3-1) edge node[auto] {$\nu$} (m-3-2);
\end{tikzpicture}
\end{center}
and prove its commutativity by induction on $n$ (for arbitrary $m$), starting with $n = 0$. In this case $\mu_{0, m} = \lambda^{\otimes}$ is and isomorphism, and the second part of the bootstrap lemma \ref{bootstraplemma} tells us that we must prove

\begin{displaymath}
\nu_{m} = \mu ((i_{0} \circ \psi_{\mathcal{W}}^{-1}) \odot \nu_{m}) \varphi_{0,m}^{-1} (\lambda^{-1} \otimes 1),
\end{displaymath}
since $\nu_{0} = i_{0} \circ \psi_{\mathcal{W}}^{-1}$. After the applying the unit condition for $\mu$ to the right side of the above equation we find that it is
\begin{displaymath}
\lambda \circ (\psi_{\mathcal{W}}^{-1} \odot \nu_{m}) \varphi_{0,m}^{-1} (\lambda^{-1} \otimes 1).
\end{displaymath}
But by coherence condition IV for $\psi$ this is exactly $\nu_{m}$, and we are done.

The inductive hypothesis is
\begin{displaymath}
\mu(\nu_{n-1} \odot \nu_{m}) = \nu_{n+m-1}(\mu_{n-1,m} \otimes 1) \varphi_{n-1,m}
\end{displaymath}
and we must show that

\begin{displaymath}
\mu(\nu_{n} \odot \nu_{m}) = \nu_{n+m}(\mu_{n,m} \otimes 1) \varphi_{n,m}.
\end{displaymath}
Expanding the left side, we can calculate \footnote{Unfortunately the diagrams involved are simply too big to include here.}

\begin{displaymath}
\begin{array}{lr}
 \mu(\nu_{n} \odot \nu_{m})  = & (\textnormal{definition of } \nu_{n}) \\
\mu([ \mu(\nu_{1} \odot \nu_{n-1}) \varphi_{1, n-1}^{-1} (s_{n} \otimes 1)] \odot \nu_{m}) = &(\textnormal{functoriality of } \odot) \\
\mu ((\mu \odot 1)  \circ [(\nu_{1} \odot \nu_{n-1}) \varphi_{1, n-1}^{-1} (s_{n} \otimes 1)] \odot \nu_{m} = & (\textnormal{associvativity of } \mu) \\ 
\mu ((1 \odot \mu) \circ (\alpha^{\odot})^{-1} \circ [(\nu_{1} \odot \nu_{n-1}) \varphi_{1, n-1}^{-1} (s_{n} \otimes 1)] \odot \nu_{m} = & (\textnormal{functoriality of } \odot) \\
\mu ((1 \odot \mu) \circ (\nu_{1} \odot (\nu_{n-1} \odot \nu_{m})) \circ (\alpha^{\odot})^{-1} \circ [\varphi_{1, n-1}^{-1} (s_{n} \otimes 1)] \odot 1) = & (\textnormal{inductive hypothesis})  \\
\mu(\nu_{1} \odot \nu_{n+m-1}(\mu_{n-1,m} \otimes 1) \varphi_{n-1, m} \circ (\alpha^{\odot})^{-1} \circ [\varphi_{1, n-1}^{-1} (s_{n} \otimes 1)] \odot 1),
\end{array}
\end{displaymath}
similarly for the right side

\begin{displaymath}
\begin{array}{c}
\nu_{n+m} (\mu_{n,m} \otimes 1) \varphi_{n,m}\\
= \\
\mu((\nu_{1} \odot \nu_{n+m-1}) \varphi_{1, n+m-1}^{-1} (s_{n+m} \otimes 1)) (\mu_{n,m} \otimes 1) \varphi_{n, m}.
\end{array}
\end{displaymath}
We will show that

\begin{displaymath}
\begin{array}{c}
\varphi_{1, n+m-1}^{-1} (s_{n+m} \otimes 1)) (\mu_{n,m} \otimes 1) \varphi_{n, m} \\
= \\
(1 \odot (\mu_{n-1,m} \otimes 1)) \varphi_{n-1, m} \circ (\alpha^{\odot})^{-1} \circ [\varphi_{1, n-1}^{-1} (s_{n} \otimes 1)] \odot 1.
\end{array}
\end{displaymath}
This follows from the commutativity of the following diagram (specifically the commutativity of the boundary)

\hspace*{-1.4in}
\begin{tikzpicture}
\matrix (m) [matrix of math nodes, column sep = .5cm, row sep = 1.5cm, text height = 1.5ex, text depth = .25ex]{
(\mathcal{W}_{n} \otimes \mathcal{W}) \odot (\mathcal{W}_{m} \otimes \mathcal{W}) & ((\mathcal{W}_{1} \odot \mathcal{W}_{n-1}) \otimes \mathcal{W}) \odot (\mathcal{W}_{m} \otimes \mathcal{W}) & ((\mathcal{W}_{1} \otimes \mathcal{W}) \odot (\mathcal{W}_{n-1} \otimes \mathcal{W})) \odot (\mathcal{W}_{m} \otimes \mathcal{W}) \\
(\mathcal{W}_{n} \odot \mathcal{W}_{m}) \otimes \mathcal{W} & ((\mathcal{W}_{1} \odot \mathcal{W}_{n-1}) \odot \mathcal{W}_{m}) \otimes \mathcal{W} & (\mathcal{W}_{1} \otimes \mathcal{W}) \odot ((\mathcal{W}_{n-1} \otimes \mathcal{W}) \odot (\mathcal{W}_{m} \otimes \mathcal{W})) \\
\mathcal{W}_{n+m} \otimes \mathcal{W} & (\mathcal{W}_{1} \odot (\mathcal{W}_{n-1} \odot \mathcal{W}_{m})) \otimes \mathcal{W} & (\mathcal{W}_{1} \otimes \mathcal{W}) \odot ((\mathcal{W}_{n-1} \odot \mathcal{W}_{m}) \otimes \mathcal{W}) \\
& (\mathcal{W}_{1} \odot \mathcal{W}_{n+m-1}) \otimes \mathcal{W} & (\mathcal{W}_{1} \otimes \mathcal{W}) \odot (\mathcal{W}_{n+m-1} \otimes \mathcal{W}) \\
};

\path[->] (m-1-1) edge node[auto] (topleft) {} (m-1-2)
			 (m-1-2) edge node[auto] (topright) {} (m-1-3)
			 (m-1-1) edge node[auto] {$\varphi_{n, m}$} (m-2-1)
			 (m-1-2) edge node[auto] {$\varphi_{1 \odot  (n-1), m}$} (m-2-2)
			 (m-2-1) edge node[auto] {$(s_{n} \odot 1) \otimes 1$} (m-2-2)
			 (m-2-1) edge node[auto] {$\mu_{n,m} \otimes 1$} (m-3-1)
			 (m-2-2) edge node[auto] {$(\alpha^{\odot})^{-1} \otimes 1$} (m-3-2)
			 (m-3-2) edge node[auto] {$(1 \odot \mu_{n-1,m}) \otimes 1$} (m-4-2)
			 (m-4-2) edge node[auto] {$\varphi_{1, n+m-1}^{-1}$} (m-4-3)
			 (m-3-2) edge node[auto] {$\varphi_{1, (n-1)\odot m}^{-1}$} (m-3-3)
			 (m-3-3) edge node[auto] {$1 \odot (\mu_{n-1,m} \otimes 1)$} (m-4-3)
			 (m-1-3) edge node[auto] {$(\alpha^{\odot})^{-1}$} (m-2-3)
			 (m-2-3) edge node[auto] {$1 \odot \varphi_{n-1,m}$} (m-3-3)
			 (m-3-1) edge node[auto,swap] {$s_{n+m} \otimes 1$} (m-4-2);

\node[above = 0 of topleft] {$\underbrace{(s_{n} \otimes 1) \odot 1}$};
\node[above = 0 of topright] {$\underbrace{\varphi_{1,n-1}^{-1} \odot 1}$};

\node[below = .5 of topleft, shape = circle, draw] {I};
\node[below = 1 of topright, shape=circle, draw] {II};
\node[below = 2.5 of topleft, shape = circle, draw] {III};
\node[below = .3 of m-3-3.south west, shape = circle, draw] {IV};
\end{tikzpicture}
This diagram commutes, since all the indicated regions commute. I and IV commute by naturality of $\varphi$, II commutes by coherence condition I for $\varphi$, and III commutes by the coherence lemma \ref{coherencelemma}.

\paragraph{Associativity of $\nu$.} We will now check that $\nu$ is associative. Note that we have not used this condition to define $\nu$, so as we have said earlier, it is a consequence of the main diagram and the unit conditions.

\begin{lemma}\label{numnfactorsthroughinm}
The composite $\nu_{n} \circ (1 \otimes i_{m})$ factors through $i_{n \cdot m}$, as in the diagram
\begin{center}
\begin{tikzpicture}
\matrix (m) [matrix of math nodes, column sep = 1cm, row sep = .5cm, text height = 1.5ex, text depth = .25ex]{
\mathcal{W}_{n} \otimes \mathcal{W}_{m} & \mathcal{W}_{n} \otimes \mathcal{W} & \mathcal{W} \\
& \mathcal{W}_{n \cdot m} \\
};

\path[->] (m-1-1) edge node[auto] {$1 \otimes i_{m}$} (m-1-2)
			 (m-1-2) edge node[auto] {$\nu_{n}$} (m-1-3)
			 (m-1-1) edge node[auto,swap] {$\nu_{n,m}$} (m-2-2)
			 (m-2-2) edge node[auto,swap] {$i_{n \cdot m}$} (m-1-3);
\end{tikzpicture}
\end{center}
\end{lemma}
\begin{proof}
By induction on $n$. For $n = 1$ we have $\nu_{1} = (i_{0} \psi_{\mathcal{W}}^{-1}, \lambda_{\mathcal{W}})$, and the claim follows from the naturality of $\lambda^{\otimes}$, $\psi$ and the fact that $i_{0} = i_{m} \circ i_{0}$ (recall our abuse of notation). We obtain $\nu_{1, m} = (i_{0} \psi_{\mathcal{W}_{m}}^{-1}, \lambda_{\mathcal{W}_{m}})$.

The inductive step immediately follows from this commutative diagram:
\hspace*{-1in}
\begin{tikzpicture}
\matrix (m) [matrix of math nodes, column sep = .8cm, row sep = 1cm, text height = 1.5ex, text depth = .25ex]{
(\mathcal{W}_{1} \otimes \mathcal{W}_{m}) \odot (\mathcal{W}_{n-1} \otimes \mathcal{W}_{m}) & (\mathcal{W}_{1} \otimes \mathcal{W}) \odot (\mathcal{W}_{n-1} \otimes \mathcal{W}) & \mathcal{W} \odot \mathcal{W} & \mathcal{W}_{m} \odot \mathcal{W}_{(n-1)m}\\
(\mathcal{W}_{1} \odot \mathcal{W}_{n-1}) \otimes \mathcal{W}_{m} & (\mathcal{W} \odot \mathcal{W}) \otimes \mathcal{W} & & \\
\mathcal{W}_{n} \otimes \mathcal{W}_{m} & \mathcal{W}_{n} \otimes \mathcal{W} & \mathcal{W} & \mathcal{W}_{nm}\\
};

\path[->] (m-1-2) edge node (place) {} node[below = .2 of place] {$\nu_{n} \odot \nu_{n-1}$} (m-1-3)
			 (m-2-2) edge node[auto,swap] {$\varphi^{-1}$} (m-1-2)
			 (m-3-2) edge node[auto,swap] {$s_{n} \otimes 1$} (m-2-2)
			 (m-1-3) edge node[auto] {$\mu$} (m-3-3)
			 (m-3-2) edge node[auto] {$\nu_{n}$} (m-3-3)
			 (m-1-1) edge node {} (m-1-2)
			 (m-2-1) edge node[auto] {$\varphi^{-1}$} (m-1-1)
			 (m-3-1) edge node[auto] {$s_{n} \otimes 1$} (m-2-1)
			 (m-3-1) edge node[auto] {$1\otimes i_{m}$} (m-3-2)
			 (m-2-1) edge node[auto] {$1 \otimes i_{m}$} (m-2-2)
			 (m-1-4) edge node {} (m-1-3)
			 (m-1-4) edge node[auto,swap] {$\mu_{m, (n-1)m}$} (m-3-4)
			 (m-3-4) edge node[auto] {$i_{nm}$} (m-3-3)
			 (m-1-1) edge[bend left = 10] node[auto] {$\nu_{1,m} \odot \nu_{n-1, m}$} (m-1-4);
\end{tikzpicture}
The top arrow implements the inductive hypothesis, and the unnamed arrows are the obvious ones.
\end{proof}

We will show, by induction on $n$, for all $m$, the commutativity of

\begin{center}
\begin{tikzpicture}
\matrix (m) [matrix of math nodes, column sep = 1.2cm, row sep = 1cm, text height = 1.5ex, text depth = .25ex]{
\mathcal{W}_{n} \otimes (\mathcal{W}_{m} \otimes \mathcal{W}) & (\mathcal{W}_{n} \otimes \mathcal{W}_{m}) \otimes \mathcal{W} & \mathcal{W}_{nm} \otimes \mathcal{W} \\
\mathcal{W}_{n} \otimes \mathcal{W} && \mathcal{W} \\
};

\path[->] (m-1-1) edge node[auto] {$\alpha^{\otimes}$} (m-1-2)
			 (m-1-2) edge node[auto] {$\nu_{n, m} \otimes 1$} (m-1-3)
			 (m-1-1) edge node[auto] {$1 \otimes \nu_{m}$} (m-2-1)
			 (m-2-1) edge node[auto] {$\nu_{n}$} (m-2-3)
			 (m-1-3) edge node[auto] {$\nu_{nm}$} (m-2-3);
\end{tikzpicture}
\end{center}
Passing to the limit gives the desired associativity law. For $n = 0$ we need to show that

\begin{displaymath}
i_{0} \psi_{\mathcal{W}_{m}}^{-1} (1 \otimes \nu_{m}) = \nu_{m} ((i_{0} \psi_{\mathcal{W}_{m}}^{-1}) \otimes 1) \alpha^{\otimes},
\end{displaymath}
which is an easy calculation following from the fact that $\psi$ is natural (anything on the right of $\psi^{-1}$ can be canceled), and coherence condition $V$. The inductive hypothesis to be used in passing from $n-1$ to $n$ is

\begin{displaymath}
\nu_{(n-1)m}(\nu_{n-1, m} \otimes 1) \alpha_{n-1, m}^{\otimes} = \nu_{n-1} (1 \otimes \nu_{m})
\end{displaymath}
Again we calculate:

\begin{displaymath}
\begin{array}{lr}
\nu_{n} (1 \otimes \nu_{m}) = & (\textnormal{definition of } \nu_{n}) \\
\mu ((\nu_{1} \odot \nu_{n-1})) \varphi_{1, n-1}^{-1} (s_{n} \otimes 1) (1 \otimes \nu_{m}) = & (\textnormal{naturality of } \varphi) \\
\mu ((\nu_{1}(1 \otimes \nu_{m}) \odot \nu_{n-1}(1 \otimes \nu_{m}))) \varphi_{1, n-1, m \otimes w}^{-1} (s_{n} \otimes 1) = & (\textnormal{inductive hypothesis}) \\
\mu((\nu_{m} (\nu_{1, m} \otimes 1) \alpha_{1, m}^{\otimes}) \odot (\nu_{(n-1)m}(\nu_{n-1, m} \otimes 1) \alpha_{n-1, m}^{\otimes})) \\
\varphi_{1, n-1, m \otimes w}^{-1} (s_{n} \otimes 1) = & (\textnormal{diagram below}) \\
\mu (\nu_{m} \odot \nu_{(n-1)m}) \varphi_{m, (n-1)m}^{-1} [(\nu_{1, m} \odot \nu_{n-1, m}) \otimes 1] \\
(\varphi_{1, n-1, m} \otimes 1) \alpha_{1 \odot (n-1), m \otimes w}^{\otimes} (s_{n} \otimes 1)
\end{array}
\end{displaymath}
The relevant diagram is

\hspace*{-1in}
\begin{tikzpicture}
\matrix (m) [matrix of math nodes, column sep = 1.2cm, row sep = 1.5cm, text height = 1.5ex, text depth = .25ex]{
(\mathcal{W}_{1} \odot \mathcal{W}_{n-1}) \otimes (\mathcal{W}_{m} \otimes \mathcal{W}) & ((\mathcal{W}_{1} \odot \mathcal{W}_{n-1}) \otimes \mathcal{W}_{m}) \otimes \mathcal{W} \\
(\mathcal{W}_{1} \otimes (\mathcal{W}_{m} \otimes \mathcal{W})) \odot (\mathcal{W}_{n-1} \otimes (\mathcal{W}_{m} \otimes \mathcal{W})) \\
((\mathcal{W}_{1} \otimes \mathcal{W}_{m}) \otimes \mathcal{W}) \odot ((\mathcal{W}_{n-1} \otimes \mathcal{W}_{m}) \otimes \mathcal{W}) & ((\mathcal{W}_{1} \otimes \mathcal{W}_{m}) \odot (\mathcal{W}_{n-1} \otimes \mathcal{W}_{m})) \otimes \mathcal{W} \\
(\mathcal{W}_{m} \otimes \mathcal{W}) \odot (\mathcal{W}_{(n-1)m} \otimes \mathcal{W}) & (\mathcal{W}_{m} \odot \mathcal{W}_{(n-1)m}) \otimes \mathcal{W} \\
};

\path[->] (m-1-1) edge node[auto] {$\alpha_{1 \odot (n-1),m}^{\otimes}$} (m-1-2)
			 (m-1-1) edge node[auto,swap] {$\varphi_{1, n-1, m \otimes w}^{-1}$} (m-2-1)
			 (m-1-2) edge node[auto] {$\varphi_{1, n-1, m}^{-1} \otimes 1$} (m-3-2)
			 (m-2-1) edge node[auto,swap] {$\alpha_{1, m}^{\otimes} \odot \alpha_{n-1, m}^{\otimes}$} (m-3-1)
			 (m-3-1) edge node[auto,swap] {$(\nu_{1,m} \otimes 1) \odot (\nu_{n-1, m} \otimes 1)$} (m-4-1)
			 (m-3-2) edge node[auto] {$(\nu_{1, m} \odot \nu_{n-1, m}) \otimes 1$} (m-4-2)
			 (m-3-2) edge node (place) {} node[below = .2 of place] {$\varphi_{1 \otimes m, (n-1) \otimes m}^{-1}$} (m-3-1)
			 (m-4-2) edge node[auto] {$\varphi_{m, (n-1)m}^{-1}$} (m-4-1);
\end{tikzpicture}
The top rectangle commutes by coherence condition II for $\varphi$. The bottom one commutes by naturality of $\varphi$.

We will now calculate the left side of the associativity condition. To do this we first calculate $\nu_{n,m}$, using the diagram from the proof of lemma \ref{numnfactorsthroughinm}:

\begin{displaymath}
\nu_{n, m} = \mu_{m, (n-1)m} (\nu_{1,m} \odot \nu_{n-1, m}) \varphi_{1, n-1, m}^{-1} (s_{n} \otimes 1).
\end{displaymath}
Putting this into

\begin{displaymath}
\nu_{nm}(\nu_{n, m} \otimes 1) \alpha_{n, m}^{\otimes},
\end{displaymath}
we obtain

\begin{displaymath}
\nu_{nm}((\mu_{m, (n-1)m} (\nu_{1,m} \odot \nu_{n-1, m}) \varphi_{1, n-1, m}^{-1} (s_{n} \otimes 1)) \otimes 1) \alpha_{n, m}^{\otimes}.
\end{displaymath}
Now consider the diagram:

\begin{center}
\begin{tikzpicture}
\matrix (m) [matrix of math nodes,column sep = 1cm, row sep = 1cm, text height = 1.5ex, text depth = .25ex]{
 &(\mathcal{W}_{m} \otimes \mathcal{W}) \odot (\mathcal{W}_{(n-1)m} \otimes \mathcal{W}) & \mathcal{W} \odot \mathcal{W} \\
 &(\mathcal{W}_{m} \odot \mathcal{W}_{(n-1)m}) \otimes \mathcal{W} & & \\
(\mathcal{W}_{n} \otimes \mathcal{W}_{m}) \otimes \mathcal{W} & \mathcal{W}_{nm} \otimes \mathcal{W} & \mathcal{W} \\
};

\path[->] (m-1-2) edge node[auto] {$\nu_{m} \odot \nu_{(n-1)m}$} (m-1-3)
			 (m-2-2) edge node[auto, swap] {$\varphi_{m, (n-1)m}^{-1}$} (m-1-2)
			 (m-2-2) edge node[auto] {$\mu_{m, (n-1)m} \otimes 1$} (m-3-2)
			 (m-1-3) edge node[auto] {$\mu$} (m-3-3)
			 (m-3-2) edge node[auto] {$\nu_{nm}$} (m-3-3)
			 (m-3-1) edge node[auto] {$\nu_{n,m} \otimes 1$} (m-3-2)
			 (m-3-1) edge node[auto] {} (m-2-2);
\end{tikzpicture}
\end{center}
which is commutative by the commutativity of the main diagram. The unnamed arrow is

\begin{displaymath}
[(\nu_{1,m} \odot \nu_{n-1, m}) \varphi_{1, n-1, m}^{-1} (s_{n} \otimes 1)] \otimes 1.
\end{displaymath}
From this, we obtain that the left side of the associativity condition is

\begin{displaymath}
\mu (\nu_{m} \odot \nu_{(n-1)m}) \varphi_{m, (n-1)m}^{-1} [(\nu_{1,m} \odot \nu_{n-1, m}) \varphi_{1, n-1, m}^{-1} (s_{n} \otimes 1)] \otimes 1 \circ \alpha_{n, m}^{\otimes},
\end{displaymath}
which is the same as the the right side of the associativity condition, finishing this part of the proof.

\paragraph{The fibered version.} The only thing left to check in this version is that the $\mathcal{W}(u)$ are homomorphisms with respect to $\nu$. To see this, note that by theorem \ref{freemonoidbjt} the map $\eta$ is fibered and natural. This means that the units of $\mathcal{W}$ are preserved across fibers. That theorem also tells us that the unit of $\mu$ and $\mu$ itself are preserved, since $\mathcal{F}_{\odot}$ is fibered, and taks values in the category of $\odot$-monoids. The structure maps $\lambda$ and $\psi$ are preserved by assumption.

From this it follows that $\nu_{0}$ is preserved, by naturality of $\psi$ and the fact that $i_{0} : I_{\odot} \rightarrow \mathcal{W}$ is preserved:

\begin{center}
\begin{tikzpicture}
\matrix (m) [matrix of math nodes, column sep = 2cm, row sep = 1cm, text height = 1.5ex, text depth = .25ex]{
I_{\odot}(Q) & \mathcal{W}_{0}(Q) \otimes_{Q} \mathcal{W}(Q) & \mathcal{W}(Q) \\
I_{\odot}(O) & \mathcal{W}_{0}(O) \otimes_{O} \mathcal{W}(O) & \mathcal{W}(O)\\
};

\path[->] (m-1-1) edge node[auto] {$\psi_{\mathcal{W}(Q)}$} (m-1-2)
			 (m-2-1) edge node[auto,swap] {$\psi_{\mathcal{W}(O)}$} (m-2-2)
			 (m-2-1) edge node[auto] {$I_{\odot}(u)$} (m-1-1)
			 (m-1-2) edge node[auto] {$\nu_{0}(Q)$} (m-1-3)
			 (m-2-2) edge node[auto,swap] {$\nu_{0}(O)$} (m-2-3)
			 (m-2-2) edge node[auto] {$\mathcal{W}_{0}(u) \otimes_{u} \mathcal{W}(u)$} (m-1-2)
			 (m-2-3) edge node[auto,swap] {$\mathcal{W}(u)$} (m-1-3)
			 (m-1-1) edge[out = 25, in = 155] node[auto] {$i_{0}(Q)$} (m-1-3)
			 (m-2-1) edge[out = 335, in = 205] node[auto] {$i_{0}(O)$} (m-2-3);
\end{tikzpicture}
\end{center}
The bigons and the left square commute, as does the boundary of the diagram (the outermost arrows). Since $\psi$ are isomorphisms, the right square also commutes. Similarly $\nu_{1}$ is preserved, since it is defined using $\lambda^{\otimes}, \psi$ and $i_{0}$, all of which are preserved by assumption, and coproducts, for which the appropriate equalities are easy to check (using corollary \ref{fiberedcolimits}). The inductive step is taken care of by the following lemma.

\begin{lemma}[Functoriality lemma]\label{functorialitylemma}
The following diagram commutes, for $n > 0$
\begin{center}
\begin{tikzpicture}
\matrix (m) [matrix of math nodes, column sep = 3cm, row sep = 1cm, text height = 1.5ex, text depth = .25ex]{
\mathcal{W}_{1}(O) \odot_{O} \mathcal{W}_{n-1}(O) & \mathcal{W}_{1}(Q) \odot_{Q} \mathcal{W}_{n-1}(Q) \\
\mathcal{W}_{n}(O) & \mathcal{W}_{n-1}(Q) \\
};

\path[->] (m-1-1) edge node[auto] {$\mathcal{W}_{1}(u) \odot_{u} \mathcal{W}_{n-1}(u)$} (m-1-2)
			 (m-2-1) edge node[auto] {$s_{n}(O)$} (m-1-1)
			 (m-2-2) edge node[auto,swap] {$s_{n}(Q)$} (m-1-2)
			 (m-2-1) edge node[auto] {$\mathcal{W}_{n}(u)$} (m-2-2);
\end{tikzpicture}
\end{center}
\end{lemma}
\begin{proof}
Expanding the definitions we have
\begin{center}
\begin{tikzpicture}
\matrix (m) [matrix of math nodes, column sep = 2.5cm, row sep = 1cm, text height = 1.5ex, text depth = .25ex]{
\mathcal{W}_{n-1}(O) \sqcup I_{\otimes,O} \odot \mathcal{W}_{n-1}(O) & \mathcal{W}_{n-1}(Q) \sqcup I_{\otimes,Q} \odot \mathcal{W}_{n-1}(Q)\\
I_{\odot,O} \sqcup I_{\otimes,O} \odot \mathcal{W}_{n-1}(O) & I_{\odot,Q} \sqcup I_{\otimes,Q} \odot \mathcal{W}_{n-1}(Q) \\
};

\path[->] (m-1-1) edge node[auto] (abovethis) {} (m-1-2)
			 (m-2-1) edge node[auto] {$i_{0, n-1}(O) \sqcup 1$} (m-1-1)
			 (m-2-2) edge node[auto,swap] {$i_{0, n-1}(Q) \sqcup 1$} (m-1-2)
			 (m-2-1) edge node[auto] {$I_{\odot,u} \sqcup I_{\otimes,u} \odot \mathcal{W}_{n-1}(u)$} (m-2-2);
\node[above = 0 of abovethis] {$\underbrace{\mathcal{W}_{n-1}(u) \sqcup I_{\otimes,u} \odot_{u} \mathcal{W}_{n-1}(u)}$};
\end{tikzpicture}
\end{center}
which commutes, since $i_{0, n-1}$ is preserved, and the other vertical \footnote{In the pictorial \emph{and} fibered sense!} components are identities.
\end{proof}

The preservation of $\nu_{n}$ now follows by induction from the following diagram

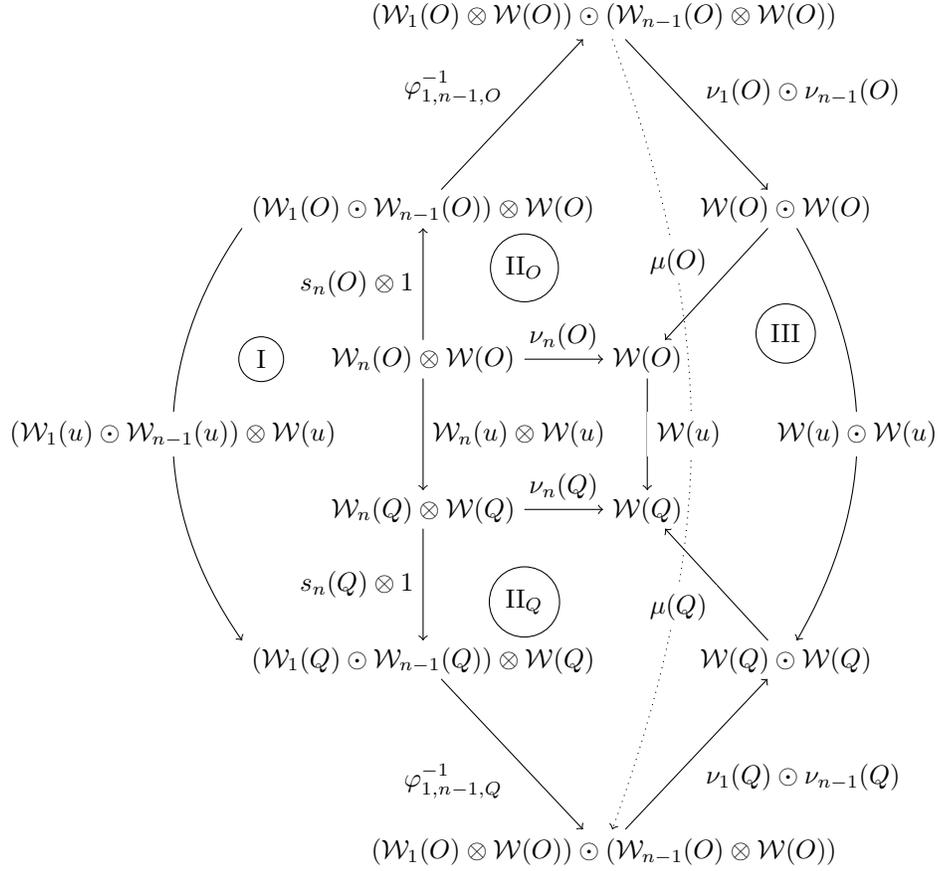
\begin{center}
\begin{tikzpicture}
\matrix (m) [matrix of math nodes, column sep = 0cm, row sep = 1.5cm, text height = 1.5ex, text depth = .25ex]{
(\mathcal{W}_{1}(O) \odot \mathcal{W}_{n-1}(O)) \otimes \mathcal{W}(O) && \mathcal{W}(O) \odot \mathcal{W}(O) \\
 \mathcal{W}_{n}(O) \otimes \mathcal{W}(O) & \mathcal{W}(O) \\
 \mathcal{W}_{n}(Q) \otimes \mathcal{W}(Q) & \mathcal{W}(Q) \\
(\mathcal{W}_{1}(Q) \odot \mathcal{W}_{n-1}(Q)) \otimes \mathcal{W}(Q) && \mathcal{W}(Q) \odot \mathcal{W}(Q) \\
};

\node[above = 2 of m-1-1.north east] (top) {$(\mathcal{W}_{1}(O) \otimes \mathcal{W}(O)) \odot (\mathcal{W}_{n-1}(O) \otimes \mathcal{W}(O))$};
\node[below = 2 of m-4-1.south east] (bottom) {$(\mathcal{W}_{1}(O) \otimes \mathcal{W}(O)) \odot (\mathcal{W}_{n-1}(O) \otimes \mathcal{W}(O))$};
\node[left = .5 of m-2-1, shape = circle, draw] {I};
\node[above = .5 of m-2-1.north east, shape = circle, draw] {II$_{O}$};
\node[below = .5 of m-3-1.south east, shape = circle, draw] {II$_{Q}$};
\node[below = 1 of m-1-3, shape = circle, draw] {III};

\path[->] (top) edge[dotted, bend left = 20] (bottom)
			 (m-1-1) edge node[auto] {$\varphi_{1, n-1, O}^{-1}$} (top)
			 (top) edge node[auto] {$\nu_{1}(O) \odot \nu_{n-1}(O)$} (m-1-3)
			 (m-4-1) edge node[auto,swap] {$\varphi_{1, n-1, Q}^{-1}$} (bottom)
			 (bottom) edge node[auto,swap]  {$\nu_{1}(Q) \odot \nu_{n-1}(Q)$} (m-4-3)
			 (m-1-1.south west) edge [bend right = 35] node[fill = white] {$(\mathcal{W}_{1}(u) \odot \mathcal{W}_{n-1}(u)) \otimes \mathcal{W}(u)$} (m-4-1.north west)
			 (m-2-1) edge node[auto] {$s_{n}(O) \otimes 1$} (m-1-1)
			 (m-2-1) edge node[auto] {$\mathcal{W}_{n}(u) \otimes \mathcal{W}(u)$} (m-3-1)
			 (m-2-1) edge node[auto] {$\nu_{n}(O)$} (m-2-2)
			 (m-3-1) edge node[auto,swap] {$s_{n}(Q) \otimes 1$} (m-4-1)
			 (m-3-1) edge node[auto] {$\nu_{n}(Q)$} (m-3-2)
			 (m-2-2) edge node[auto, fill = white] {$\mathcal{W}(u)$} (m-3-2)
			 (m-1-3) edge node[auto,swap, fill = white] {$\mu(O)$} (m-2-2)
			 (m-1-3) edge[bend left = 30] node[fill = white] {$\mathcal{W}(u) \odot \mathcal{W}(u)$} (m-4-3)
			 (m-4-3) edge node[auto, fill = white] {$\mu(Q)$} (m-3-2);
\end{tikzpicture}
\end{center}
The dotted arrow is $(\mathcal{W}_{1}(u) \otimes \mathcal{W}(u)) \odot (\mathcal{W}_{n-1}(u) \otimes \mathcal{W}(u))$. The reader should imagine two of the main diagrams defining $\nu(O)$ and $\nu(Q)$, side by side, connected by the various $\mathcal{W}_{k}(u)$. The picture above is a flattening of that situation. We need to check the commutativity of the central square. To do this we check that every other region commutes. I commutes by the functoriality lemma \ref{functorialitylemma}, applied to the left variable. II$_{O}$ and II$_{Q}$ commute by the definition of $\nu_{n}$, and III commutes by the fact that $\mathcal{W}(u)$ are $\odot$-monoid homomorphisms, which they are by definition. We now turn to unnamed regions, which should be obvious if the reader followed our advice. The region defined by the dotted arrow, the leftmost arrow in region $I$, and $\varphi$ commutes by naturality of $\varphi$. The region defined by the dotted arrow and the solid boundary of the diagram to the right of it commutes by the inductive hypothesis. Hence the central square commutes.

The preservation of $\nu$ now follows from corollary \ref{fiberedcolimits}. This completes the proof of theorems \ref{threetensorstheorem} and \ref{fiberedttt}.

\section{Coherence Calculations for Distributivity in $\mathbf{Sig}_{ma}$}\label{coherencefordistributivity}
We must check whether $\varphi_{A, B, X}$ and $\psi_{X}$ define a distributivity structure for $\mathbf{Sig}_{a}$. To do this we will show that they are natural and prove that they satisfy all the coherence diagrams listed in appendix \ref{appendixa}. By the remarks following the definition of distributivity this is what the definition comes down to.

Before we dive into the calculations two remarks are in order. First, conditions $III - VII$ are very easy. To make our point we will leave the last diagram (which is unnecessary anyway) as a rather trivial exercise. The only real problem is the complexity of the terms in conditions $I$ and $II$. Our calculations for these conditions will look somewhat like a physicist's version of tensor calculus -- crawling with indices. To ease our problems we will use the separation principle -- all the functors involved in these conditions are agreeable and separated (as we will see).

\begin{proposition}
$\varphi_{A, B, X}$ and $\psi_{X}$ are isomorphisms, fibered natural in all variables.
\end{proposition}
\begin{proof}
The statement for $\psi$ is obvious, as is the isomorphism part. We consider naturality for $\varphi$:

\begin{center}
\begin{tikzpicture}
\matrix (m) [matrix of math nodes, row sep = 1cm, column sep = 3.5cm, text height = 1.5ex, text depth = .25ex] {
(A \otimes X) \odot (B \otimes X) & (A' \otimes X') \odot (B' \otimes X') \\
(A \odot B) \otimes X &  (A' \odot B') \otimes X' \\
M & N \\
};

\path[->] (m-1-1) edge node[auto] {$(f \otimes h) \odot (g \otimes h)$} (m-1-2)
			 (m-2-1) edge node[auto] {$(f \odot g) \otimes h$} (m-2-2)
			 (m-1-1) edge node[auto, swap] {$\varphi_{A, B, X}$} (m-2-1)
			 (m-1-2) edge node[auto] {$\varphi_{A', B', X'}$} (m-2-2)
			 (m-3-1) edge node[auto] {$u$} (m-3-2);

\end{tikzpicture}
\end{center}
where $f, g$ and $h$ are over the homomorphism $u$. The amalgamation permutations of these morphisms are denoted $\sigma, \tau, \delta$ and $\theta$ respectively. We consider a term

\begin{displaymath}
\dot{\langle} \langle a, x_{0, k} \rangle, \langle b_{i}, x_{i, k} \rangle \dot{\rangle}
\end{displaymath}
in $(A \otimes X) \odot (B \otimes X)$ and apply $(f \odot g) \otimes h \circ \varphi_{A, B, X}$ to it obtaining

\begin{equation}\label{someeqn}
\langle \dot{\langle} f(a), g(b_{\theta_{\check{a}}^{-1}(i)}) \dot{\rangle}, h(x_{\sigma \odot \tau_{\dot{\langle} a, b_{i} \dot{\rangle}}^{-1}(i, k)}) \rangle
\end{equation}
On the other hand we can apply $(f \otimes h) \odot (g \otimes h)$ and obtain

\begin{displaymath}
\dot{\langle} \langle f(a), h(x_{0, \sigma_{a}^{-1}(k)} \rangle, \langle g(b_{\theta_{\check{a}}^{-1}(i)}), h(x_{\theta_{\check{a}}^{-1}(i), \tau_{b_{\theta_{\check{a}}^{-1}(i)}}^{-1}(k)}) \rangle \dot{\rangle},
\end{displaymath}
which $\varphi_{A', B', X'}$ maps to

\begin{eqnarray*}
\langle \dot{\langle} f(a), g(b_{\theta_{\check{a}}^{-1}(1)}), \dots, g(b_{\theta_{\check{a}}^{-1}(k)}) \dot{\rangle}, \phantom{xxxxxxxxxxxxxxxxxxxxxxxxxxxx} \\
h(x_{0, \sigma_{a}^{-1}(1)}), \dots, h(x_{0, \sigma_{a}^{-1}(l_{0})}), h(x_{\theta_{\check{a}}^{-1}(1), \tau_{b_{\theta_{\check{a}}^{-1}(1)}}^{-1}(1)}), \dots h(x_{\theta_{\check{a}}^{-1}(k), \tau_{b_{\theta_{\check{a}}^{-1}(k)}}^{-1}(l_{k})}) \rangle,
\end{eqnarray*}
which is equal to term \ref{someeqn} by our definition of $\tau \odot \delta$ (and the formula for inverses in the operad of symmetries).

We must still prove that the amalgamation permutations are equal. This means that
\begin{displaymath}
(\sigma \otimes \delta) \odot (\tau \otimes \delta)_{\dot{\langle} \langle a, x_{0, k} \rangle, \langle b_{i}, x_{i, k} \rangle \dot{\rangle}} = (\sigma \odot \tau) \otimes \delta_{\langle \dot{\langle} a, b_{i} \dot{\rangle}, x_{i, k} \rangle}
\end{displaymath}
Fortunately we can write out both sides in this case, using just our definitions. The left side is

\begin{displaymath}
(1, \theta_{\check{a}}) \ast (\sigma_{a} \ast (\delta_{x_{0, k}}), \tau_{b_{i}} \ast (\delta_{x_{i,k}}) )
\end{displaymath}
and the right side is

\begin{displaymath}
[ (1, \theta_{\check{a}}) \ast (\sigma_{a}, \tau_{b_{i}}) ] \ast (\delta_{x_{i,k}})
\end{displaymath}
They are equal by the associativity of the operad of symmetries.
\end{proof}

Let us see why all our functors are jointly agreeable. This is a simple consequence of our formulas. Prone morphisms in $\mathbf{Sig}_{ma}$ are defined using $\mathbf{Set}$-pullback (just like in $\mathbf{Sig}_{a}$), and are therefore strict, and the projection $\pi: M_{\mathbb{N}} \rightarrow M$ is strict. Therefore all possible combinations of $\otimes$ and $\odot$ on these morphisms will have standard amalgamation -- the formulas for amalgamation give identities if they are supplied only with identities. Nonstandard amalgamation does not appear out of thin air, so to speak.

Separation can be seen, in some sense, in the same way. All our functors are combinations of $\otimes$ and $\odot$, and the typing of their values is defined from the typings of their arguments. Thus the arguments contain the simplest building blocks of the terms we will consider. If our term is, for example $\langle f, g_{1}, \dots, g_{n} \rangle \in A \otimes B$, then the simplest building blocks are $f$ and the $g_{i}$. We can attach consecutive natural numbers to the inputs of $g_{1}$, bigger numbers to the inputs of $g_{2}$, and so on up to $g_{n}$. We can attach numbers to outputs of $g_{i}$ to maintain composability with $f$ -- in this case they can be arbitrary. This will define a term in $\pi^{\ast} A \otimes \pi^{\ast} B$ with injective typing mapping to the original one when we forget the added numbers. This proves separability of the functor $\otimes$ for $\mathbf{Sig}_{ma}$ (and $\mathbf{Sig}_{a}$ also). This procedure works for terms of arbitrary complexity, in particular for those which are elements of our diagrams\footnote{Formally we should use induction on complexity of the terms, but this only obscures the idea. An argument essentially equivalent to the separability of $\mathcal{F}_{\odot} \odot \mathcal{F}_{\odot}$ can be found in \cite[part II, lemma 7]{HMP}. It was not appreciated by the second author.}.

We start with the easy diagrams (the last one is an exercise).

\paragraph{Condition III.} An element of $A \odot B$ is of the form

\begin{displaymath}
\dot{\langle}a, b_{i} \dot{\rangle},
\end{displaymath}
where $i$ ranges over the horizontal inputs of $a$. The map $\rho^{\otimes} \odot \rho^{\otimes}$ maps this to

\begin{displaymath}
\dot{\langle} \langle a, 1_{\partial_{a}^{A}(j)} \rangle, \langle b_{i}, 1_{\partial_{b_{i}}^{B}(j')} \rangle \dot{\rangle},
\end{displaymath}
where $j$ and $j'$ range over the vertical inputs. The map $\varphi$ maps this to

\begin{displaymath}
\langle \dot{\langle} a, b_{i} \dot{\rangle}, 1_{\partial_{a}^{A}(j)}, 1_{\partial_{b_{i}}^{B}(j')} \rangle,
\end{displaymath}
which is exactly what $\rho^{\otimes}$ does to the original term. By separability we are finished.

\paragraph{Condition IV.} We start with

\begin{displaymath}
\dot{\langle} \langle 1_{\partial^{M}_{\check{a}}(0)}, - \rangle, \langle a, x_{i} \rangle \dot{\rangle},
\end{displaymath}
where $(-)$ represents the empty list. There is only one $a$ since $1_{\partial^{M}_{\check{a}}(0)}$ is unary. $\varphi$ maps this to

\begin{displaymath}
\langle \dot{\langle} 1_{\partial^{M}_{\check{a}}(0)}, a \dot{\rangle} , x_{i} \rangle,
\end{displaymath}
which $\lambda^{\odot} \otimes 1$ maps to

\begin{displaymath}
\langle a, x_{i} \rangle.
\end{displaymath}
The final result of this way is thus $f$ applied to the above term. The other way around the diagram goes like this. Starting with the original term we obtain in the first step

\begin{displaymath}
\dot{\langle} 1_{\partial^{M}_{\check{a}}(0)}, f(\langle a, x_{i} \rangle) \dot{\rangle},
\end{displaymath}
and then, applying $\lambda^{\odot}$,

\begin{displaymath}
f(\langle a, x_{i} \rangle)
\end{displaymath}
in the second step. Thus both ways agree.

\paragraph{Condition V.} The condition says very little in our case. We start with

\begin{displaymath}
\langle 1_{o}, - \rangle,
\end{displaymath}
which is mapped by $\psi^{-1}$ to $1_{o}$.

\noindent
Alternately it is mapped by $\alpha^{\otimes}$ to (surprise!)

\begin{displaymath}
\langle \langle 1_{o}, - \rangle, - \rangle,
\end{displaymath}
and then by $\psi^{-1} \otimes 1$ to

\begin{displaymath}
\langle 1_{o}, - \rangle,
\end{displaymath}
which the final $\psi^{-1}$ maps to $1_{o}$, thus agreeing with the first way.

\paragraph{Condition VI.} This is entirely trivial -- there is only one way to add an empty list to a unary term. We start with $1_{o} \in I_{\odot}$ and both $\psi$ and $\rho^{\otimes}$ map it to

\begin{displaymath}
\langle 1_{o}, - \rangle,
\end{displaymath}
by definition for $\psi$, and for $\rho^{\otimes}$ because $I_{\odot}$ has no vertical inputs.

Now we turn to the more complicated cases

\paragraph{Condition II.} We start with a rather unwieldy

\begin{displaymath}
\dot{\langle} \langle a, \langle x_{i}, y_{i,j} \rangle \rangle , \langle b_{k},  \langle x'_{i'}, y'_{i', j'} \rangle \rangle  \dot{\rangle},
\end{displaymath}
which $\varphi$ maps to

\begin{displaymath}
\langle \dot{\langle} a, b_{k} \dot{\rangle}, \langle x_{i}, y_{i,j} \rangle, \langle x'_{i'}, y'_{i', j'} \rangle \rangle,
\end{displaymath}
which after $\alpha^{\otimes}$ becomes

\begin{displaymath}
\langle \langle \dot{\langle} a, b_{k} \dot{\rangle}, x_{i}, x'_{i'} \rangle, y_{i,j}, y'_{i',j'} \rangle.
\end{displaymath}

The other way maps the original term by $\alpha^{\otimes} \odot \alpha^{\otimes}$ to

\begin{displaymath}
\dot{\langle} \langle \langle a, x_{i} \rangle , y_{i,j} \rangle \rangle, \langle \langle b_{k}, x'_{i'} \rangle, y'_{i', j'} \rangle \rangle  \dot{\rangle},
\end{displaymath}
and then by $\varphi$ to

\begin{displaymath}
\langle \dot{\langle} \langle a, x_{i} \rangle, \langle b_{k}, x'_{i'} \rangle \dot{\rangle}, y_{i,j}, y'_{i', j'} \rangle.
\end{displaymath}
Applying the final $\varphi \otimes 1$ yields

\begin{displaymath}
\langle \langle \dot{\langle} a, b_{k} \dot{\rangle}, x_{i}, x'_{i'} \rangle, y_{i,j}, y'_{i',j'} \rangle,
\end{displaymath}
in agreement with our previous calculation.

\paragraph{Condition I.} Up to now we did not have to deal with any permutations acting on terms (the other ones were taken care of by the separation principle). The first condition is the hardest one because this is not true in its case. Fortunately the permutations are manageable. We begin with

\begin{displaymath}
\dot{\langle} \langle a, x_{i} \rangle, \dot{\langle} \langle b_{k}, x_{k, j} \rangle, \langle c_{k,l}, x_{k, l, m} \rangle \dot{\rangle}  \dot{\rangle}.
\end{displaymath}
The last index of each instance of $x$ ranges over the vertical inputs of $a, b_{k}$ or $c_{k,l}$. The indices $k$ and $(k, l)$ range over horizontal inputs of $a$ and $b_{k}$.

After applying $1 \odot \varphi$ to this term we obtain

\begin{displaymath}
\dot{\langle} \langle a, x_{i} \rangle, \langle \dot{\langle} b_{k}, c_{k,l} \dot{\rangle}, x_{k,j}, x_{k, l, m} \rangle \dot{\rangle},
\end{displaymath}
which $\varphi$ maps to

\begin{displaymath}
\langle \dot{\langle} a, \dot{\langle} b_{k}, c_{k, l} \dot{\rangle} \dot{\rangle}, x_{i}, x_{1, j}, x_{1, l, m}, \dots, x_{2, j}, x_{2, l, m}, \dots \rangle.
\end{displaymath}
We must now determine what $\alpha^{\odot} \otimes 1$ does to this term. This means looking at the definition of the amalgamation permutations of $\alpha^{\odot}$, which we have denoted by $\pi$:

\begin{displaymath}
\pi_{\dot{\langle} a, \dot{\langle} b_{i}, c_{i, j} \dot{\rangle} \dot{\rangle}} = \kappa_{\dot{\langle} a, \dot{\langle} b_{i}, c_{i, j} \dot{\rangle} \dot{\rangle}} \ast (1_{(|a|]}, \dots 1_{(|c_{k, l_{k}}|]}),
\end{displaymath}
where $\kappa$ is the permutation that implements the movements of the function symbols between $\dot{\langle} a, \dot{\langle} b_{i}, c_{i, j} \dot{\rangle} \dot{\rangle}$ and $\dot{\langle} \dot{\langle} a, b_{i} \dot{\rangle}, c_{\gamma_{\langle \check{a}, \check{b}_{i} \rangle}^{-1}(i, j)} \dot{\rangle}$. Therefore $\alpha^{\odot} \otimes 1$ acts on our term as follows

\begin{displaymath}
\langle \dot{\langle} \dot{\langle} a, b_{k} \dot{\rangle}, c_{\gamma_{\langle \check{a}, \check{b}_{k} \rangle}^{-1}(k, l)} \dot{\rangle}, x_{i}, x_{k,j}, x_{\gamma_{\langle \check{a}, \check{b}_{k} \rangle}^{-1}(k,l), m} \rangle,
\end{displaymath}
where $\gamma$ are the permutation amalgamations of multiplication in $M$.

We must determine what happens when we take the other way around the diagram. By our typing conventions $\alpha^{\odot}$ maps our original term to

\begin{displaymath}
\dot{\langle} \dot{\langle} \langle a, x_{i} \rangle, \langle b_{k}, x_{k, j} \rangle \dot{\rangle}, \langle c_{\gamma_{\langle \check{a}, \check{b}_{i} \rangle}^{-1}(k,l)}, x_{\gamma_{\langle \check{a}, \check{b}_{k} \rangle}^{-1}(k, l), m} \rangle \dot{\rangle},
\end{displaymath}
which $\varphi \odot 1$ makes into

\begin{displaymath}
\dot{\langle} \langle \dot{\langle} a, b_{k} \dot{\rangle}, x_{i}, x_{k,j} \rangle, \langle c_{\gamma_{\langle \check{a}, \check{b}_{i} \rangle}^{-1}(k,l)}, x_{\gamma_{\langle \check{a}, \check{b}_{k} \rangle}^{-1}(k, l), m} \rangle \dot{\rangle}.
\end{displaymath}
Applying the final $\varphi$ we obtain

\begin{displaymath}
\langle \dot{\langle} \dot{\langle} a, b_{k} \dot{\rangle}, c_{\gamma_{\langle \check{a}, \check{b}_{i} \rangle}^{-1}(k, l)} \dot{\rangle}, x_{i}, x_{k,j}, x_{\gamma_{\langle \check{a}, \check{b}_{k} \rangle}^{-1}(k,l), m} \rangle,
\end{displaymath}
as we should. This concludes the proof of theorem \ref{distributivityforsigma}.

\section{Nonstandard Amalgamation is Necessary}\label{amalgamationisnecessaryexample}
The following remarkably simple example shows that the web monoid $\mathcal{W}(M)$ need not be isomorphic to any monoid with standard amalgamation, even if $M$ is standard. Such an isomorphism amounts to being able to retype the elements of the web monoid in such a way as to get standard amalgamation for multiplication. The example consists purely of pictures, and hence applies to most other constructions in the literature (for example the multicategory of function replacement).

Consider a set $O$ of three distinct types $\{$ circle, square, triangle $\}$ and a signature $M$ consisting of the following function symbols: $\{b, c, s, t, 1_{c}, 1_{s}, 1_{t} \}$. The symbol $b$ (like binary) is binary, and its typing is arbitrary, but injective. We have fixed one such typing in the pictures below. The symbols $c, s, t$ are unary of input and output type circle, square and triangle, respectively. The unary symbols $1_{x}$ are to be thought of as identities on their respective types (we will want to consider $M$ as a monoid). We will draw the nonidentity symbols like this

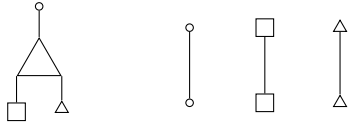
\begin{center}
\begin{tikzpicture}
\node[isosceles triangle, isosceles triangle apex angle = 60, shape border rotate = 90, draw, minimum height = .5cm] (a) at (0,0) {};
\node[circle, draw, inner sep = 1pt, above = .35 of a.apex] (aout) {};
\node[rectangle, draw, below = .35 of a.left corner] (ain1) {};
\node[regular polygon, regular polygon sides = 3, draw, inner sep = 1pt, below = .32 of a.right corner] (ain2) {};
\draw[-] (a.apex) -- (aout)
			(a.left corner) -- (ain1)
			(a.right corner) -- (ain2.corner 1);

\node[circle, draw, inner sep = 1pt] (A0) at (2,-.5) {};
\node[circle, draw, inner sep = 1pt] (A1) at (2,.5) {};

\node[rectangle, draw] (B0) at (3,-.5) {};
\node[rectangle, draw] (B1) at (3,.5) {};

\node[regular polygon, regular polygon sides = 3, inner sep = 1pt, draw] (C0) at (4,-.5) {};
\node[regular polygon, regular polygon sides = 3, inner sep = 1pt, draw] (C1) at (4,.5) {};

\draw (A0) -- (A1) (B0) -- (B1) (C0.corner 1) -- (C1);

\end{tikzpicture}
\end{center}
The shapes indicate input/output types of the symbols. We will never draw the identity symbols.

Note that $M$ has, because of our typing choices, a unique structure of a monoid in $\mathbf{Sig}_{a}$ (up to a choice of identities, one of which we have indicated), and this monoid has standard amalgamation. Below we draw part of the multiplication table for $M$. It shows the result of computing $\mu(\langle c, b \rangle)$, $\mu(\langle b, 1_{t} , s \rangle)$ and $\mu(\langle b, t, 1_{s} \rangle)$, where $\mu$ is the multiplication map.

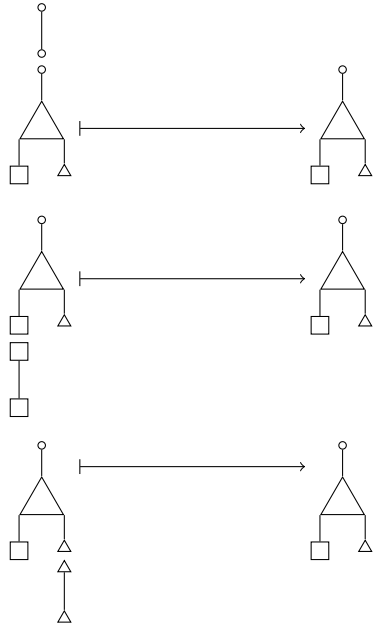
\begin{center}
\begin{tikzpicture}
\node[isosceles triangle, isosceles triangle apex angle = 60, shape border rotate = 90, draw, minimum height = .5cm] (a) at (0,0) {};
\node[circle, draw, inner sep = 1pt, above = .35 of a.apex] (aout) {};
\node[rectangle, draw, below = .35 of a.left corner] (ain1) {};
\node[regular polygon, regular polygon sides = 3, draw, inner sep = 1pt, below = .32 of a.right corner] (ain2) {};
\draw[-] (a.apex) -- (aout)
			(a.left corner) -- (ain1)
			(a.right corner) -- (ain2.corner 1);

\node[isosceles triangle, isosceles triangle apex angle = 60, shape border rotate = 90, draw, minimum height = .5cm] (b) at (0,-2) {};
\node[circle, draw, inner sep = 1pt, above = .35 of b.apex] (bout) {};
\node[rectangle, draw, below = .35 of b.left corner] (bin1) {};
\node[regular polygon, regular polygon sides = 3, draw, inner sep = 1pt, below = .32 of b.right corner] (bin2) {};
\draw[-] (b.apex) -- (bout)
			(b.left corner) -- (bin1)
			(b.right corner) -- (bin2.corner 1);

\node[isosceles triangle, isosceles triangle apex angle = 60, shape border rotate = 90, draw, minimum height = .5cm] (c) at (0,-5) {};
\node[circle, draw, inner sep = 1pt, above = .35 of c.apex] (cout) {};
\node[rectangle, draw, below = .35 of c.left corner] (cin1) {};
\node[regular polygon, regular polygon sides = 3, draw, inner sep = 1pt, below = .32 of c.right corner] (cin2) {};
\draw[-] (c.apex) -- (cout)
			(c.left corner) -- (cin1)
			(c.right corner) -- (cin2.corner 1);

\node[isosceles triangle, isosceles triangle apex angle = 60, shape border rotate = 90, draw, minimum height = .5cm] (d) at (4,0) {};
\node[circle, draw, inner sep = 1pt, above = .35 of d.apex] (dout) {};
\node[rectangle, draw, below = .35 of d.left corner] (din1) {};
\node[regular polygon, regular polygon sides = 3, draw, inner sep = 1pt, below = .32 of d.right corner] (din2) {};
\draw[-] (d.apex) -- (dout)
			(d.left corner) -- (din1)
			(d.right corner) -- (din2.corner 1);

\node[isosceles triangle, isosceles triangle apex angle = 60, shape border rotate = 90, draw, minimum height = .5cm] (e) at (4,-2) {};
\node[circle, draw, inner sep = 1pt, above = .35 of e.apex] (eout) {};
\node[rectangle, draw, below = .35 of e.left corner] (ein1) {};
\node[regular polygon, regular polygon sides = 3, draw, inner sep = 1pt, below = .32 of e.right corner] (ein2) {};
\draw[-] (e.apex) -- (eout)
			(e.left corner) -- (ein1)
			(e.right corner) -- (ein2.corner 1);

\node[isosceles triangle, isosceles triangle apex angle = 60, shape border rotate = 90, draw, minimum height = .5cm] (f) at (4,-5) {};
\node[circle, draw, inner sep = 1pt, above = .35 of f.apex] (fout) {};
\node[rectangle, draw, below = .35 of f.left corner] (fin1) {};
\node[regular polygon, regular polygon sides = 3, draw, inner sep = 1pt, below = .32 of f.right corner] (fin2) {};
\draw[-] (f.apex) -- (fout)
			(f.left corner) -- (fin1)
			(f.right corner) -- (fin2.corner 1);

\path[|->] (.5,0) edge (3.5,0) (.5,-2) edge (3.5, -2) (.5, -4.5) edge (3.5, -4.5);

\node[circle, draw, inner sep = 1pt, above = .1 of aout] (cmp1in) {};
\node[circle, draw, inner sep = 1pt, above = .5 of cmp1in] (cmp1out) {};
\node[rectangle, draw, below = .1 of bin1] (cmp2out) {};
\node[rectangle, draw, below = .5 of cmp2out] (cmp2in) {};
\node[regular polygon, regular polygon sides = 3, draw, inner sep = 1pt, below = .1 of cin2] (cmp3out) {};
\node[regular polygon, regular polygon sides = 3, draw, inner sep = 1pt, below = .5 of cmp3out] (cmp3in) {};

\draw (cmp1in) -- (cmp1out) (cmp2in) -- (cmp2out) (cmp3in) -- (cmp3out);

\end{tikzpicture}
\end{center}

The web monoid $\mathcal{W}(M)$ consists of formal composites of these symbols (its universe is $\mathcal{F}_{\odot}(I_{\otimes}(M))$). The list of input types of a formal composite is the list of the symbols used to build it (in ``tree order'', but this is irrelevant since we will consider all the orderings), and its output type is its composite in $M$ (image under the counit).

Looking at the multiplication table for $M$, we see that in $\mathcal{W}(M) \otimes \mathcal{W}(M)$ the following elements are well-defined. Each formal composite on the right is input to the central binary symbol in the formal composite on the left (again, other inputs get identities):

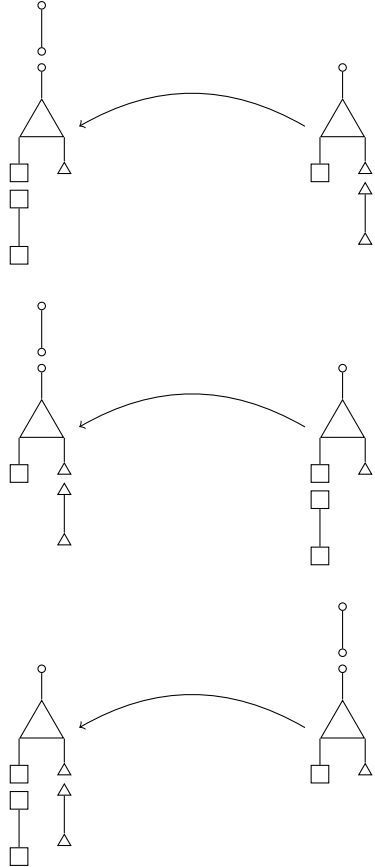
\begin{center}
\begin{tikzpicture}
\node[isosceles triangle, isosceles triangle apex angle = 60, shape border rotate = 90, draw, minimum height = .5cm] (a) at (0,0) {};
\node[circle, draw, inner sep = 1pt, above = .35 of a.apex] (aout) {};
\node[rectangle, draw, below = .35 of a.left corner] (ain1) {};
\node[regular polygon, regular polygon sides = 3, draw, inner sep = 1pt, below = .32 of a.right corner] (ain2) {};
\draw[-] (a.apex) -- (aout)
			(a.left corner) -- (ain1)
			(a.right corner) -- (ain2.corner 1);

\node[isosceles triangle, isosceles triangle apex angle = 60, shape border rotate = 90, draw, minimum height = .5cm] (b) at (0,-4) {};
\node[circle, draw, inner sep = 1pt, above = .35 of b.apex] (bout) {};
\node[rectangle, draw, below = .35 of b.left corner] (bin1) {};
\node[regular polygon, regular polygon sides = 3, draw, inner sep = 1pt, below = .32 of b.right corner] (bin2) {};
\draw[-] (b.apex) -- (bout)
			(b.left corner) -- (bin1)
			(b.right corner) -- (bin2.corner 1);

\node[isosceles triangle, isosceles triangle apex angle = 60, shape border rotate = 90, draw, minimum height = .5cm] (c) at (0,-8) {};
\node[circle, draw, inner sep = 1pt, above = .35 of c.apex] (cout) {};
\node[rectangle, draw, below = .35 of c.left corner] (cin1) {};
\node[regular polygon, regular polygon sides = 3, draw, inner sep = 1pt, below = .32 of c.right corner] (cin2) {};
\draw[-] (c.apex) -- (cout)
			(c.left corner) -- (cin1)
			(c.right corner) -- (cin2.corner 1);

\node[isosceles triangle, isosceles triangle apex angle = 60, shape border rotate = 90, draw, minimum height = .5cm] (d) at (4,0) {};
\node[circle, draw, inner sep = 1pt, above = .35 of d.apex] (dout) {};
\node[rectangle, draw, below = .35 of d.left corner] (din1) {};
\node[regular polygon, regular polygon sides = 3, draw, inner sep = 1pt, below = .32 of d.right corner] (din2) {};
\draw[-] (d.apex) -- (dout)
			(d.left corner) -- (din1)
			(d.right corner) -- (din2.corner 1);

\node[circle, draw, inner sep = 1pt, above = .1 of aout] (cmp1in) {};
\node[circle, draw, inner sep = 1pt, above = .5 of cmp1in] (cmp1out) {};
\node[rectangle, draw, below = .1 of ain1] (cmp2out) {};
\node[rectangle, draw, below = .5 of cmp2out] (cmp2in) {};
\node[regular polygon, regular polygon sides = 3, draw, inner sep = 1pt, below = .1 of din2] (cmp3out) {};
\node[regular polygon, regular polygon sides = 3, draw, inner sep = 1pt, below = .5 of cmp3out] (cmp3in) {};

\draw (cmp1in) -- (cmp1out) (cmp2in) -- (cmp2out) (cmp3in) -- (cmp3out);

\node[isosceles triangle, isosceles triangle apex angle = 60, shape border rotate = 90, draw, minimum height = .5cm] (e) at (4,-4) {};
\node[circle, draw, inner sep = 1pt, above = .35 of e.apex] (eout) {};
\node[rectangle, draw, below = .35 of e.left corner] (ein1) {};
\node[regular polygon, regular polygon sides = 3, draw, inner sep = 1pt, below = .32 of e.right corner] (ein2) {};
\draw[-] (e.apex) -- (eout)
			(e.left corner) -- (ein1)
			(e.right corner) -- (ein2.corner 1);

\node[circle, draw, inner sep = 1pt, above = .1 of bout] (cmp1in2) {};
\node[circle, draw, inner sep = 1pt, above = .5 of cmp1in2] (cmp1out2) {};
\node[rectangle, draw, below = .1 of ein1] (cmp2out2) {};
\node[rectangle, draw, below = .5 of cmp2out2] (cmp2in2) {};
\node[regular polygon, regular polygon sides = 3, draw, inner sep = 1pt, below = .1 of bin2] (cmp3out2) {};
\node[regular polygon, regular polygon sides = 3, draw, inner sep = 1pt, below = .5 of cmp3out2] (cmp3in2) {};

\draw (cmp1in2) -- (cmp1out2) (cmp2in2) -- (cmp2out2) (cmp3in2) -- (cmp3out2);

\node[isosceles triangle, isosceles triangle apex angle = 60, shape border rotate = 90, draw, minimum height = .5cm] (f) at (4,-8) {};
\node[circle, draw, inner sep = 1pt, above = .35 of f.apex] (fout) {};
\node[rectangle, draw, below = .35 of f.left corner] (fin1) {};
\node[regular polygon, regular polygon sides = 3, draw, inner sep = 1pt, below = .32 of f.right corner] (fin2) {};
\draw[-] (f.apex) -- (fout)
			(f.left corner) -- (fin1)
			(f.right corner) -- (fin2.corner 1);

\node[circle, draw, inner sep = 1pt, above = .1 of fout] (cmp1in3) {};
\node[circle, draw, inner sep = 1pt, above = .5 of cmp1in3] (cmp1out3) {};
\node[rectangle, draw, below = .1 of cin1] (cmp2out3) {};
\node[rectangle, draw, below = .5 of cmp2out3] (cmp2in3) {};
\node[regular polygon, regular polygon sides = 3, draw, inner sep = 1pt, below = .1 of cin2] (cmp3out3) {};
\node[regular polygon, regular polygon sides = 3, draw, inner sep = 1pt, below = .5 of cmp3out3] (cmp3in3) {};

\draw (cmp1in3) -- (cmp1out3) (cmp2in3) -- (cmp2out3) (cmp3in3) -- (cmp3out3);

\path[->] (3.5,0) edge[bend right = 30] (.5,0) (3.5,-4) edge[bend right = 30] (.5,-4) (3.5,-8) edge[bend right = 30] (.5,-8);

\end{tikzpicture}
\end{center}
and all of them compose to

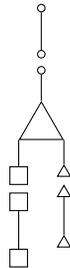
\begin{center}
\begin{tikzpicture}
\node[isosceles triangle, isosceles triangle apex angle = 60, shape border rotate = 90, draw, minimum height = .5cm] (a) at (0,0) {};
\node[circle, draw, inner sep = 1pt, above = .35 of a.apex] (aout) {};
\node[rectangle, draw, below = .35 of a.left corner] (ain1) {};
\node[regular polygon, regular polygon sides = 3, draw, inner sep = 1pt, below = .32 of a.right corner] (ain2) {};
\draw[-] (a.apex) -- (aout)
			(a.left corner) -- (ain1)
			(a.right corner) -- (ain2.corner 1);

\node[circle, draw, inner sep = 1pt, above = .1 of aout] (cmp1in) {};
\node[circle, draw, inner sep = 1pt, above = .5 of cmp1in] (cmp1out) {};
\node[rectangle, draw, below = .1 of ain1] (cmp2out) {};
\node[rectangle, draw, below = .5 of cmp2out] (cmp2in) {};
\node[regular polygon, regular polygon sides = 3, draw, inner sep = 1pt, below = .1 of ain2] (cmp3out) {};
\node[regular polygon, regular polygon sides = 3, draw, inner sep = 1pt, below = .5 of cmp3out] (cmp3in) {};

\draw (cmp1in) -- (cmp1out) (cmp2in) -- (cmp2out) (cmp3in) -- (cmp3out);
\end{tikzpicture}
\end{center}
(this ultimately follows from the definition of $\nu_{n}$ given in equation \ref{definitionofnun}).

Consider the amalgamation permutations of multiplication in $\mathcal{W}(M)$. The above composite has four distinct types -- the list of function symbols used to build it -- as do the above elements of $\mathcal{W}(M) \otimes \mathcal{W}(M)$ which compose to it. Thus the permutation amalgamations are determined uniquely once we determine the order in which these types are listed for all the four elements we are considering. We must list them in such a way that all the amalgamations arising from the above compositions can be taken to be the identity. But in the above elements of $\mathcal{W}(M) \otimes \mathcal{W}(M)$ the binary symbol was listed next to each nonidentity unary symbol (because of our convention for typing tensor products). The identity permutation preserves the ``was listed next to'' relation. Thus if we want standard amalgamation the binary function symbol in the above formal composite must have as neighbors all three unary symbols. Three neighbors in a list is one neighbor too many -- contradiction.

It is easy to see that the situation above actually arises in our opetopic sets. Therefore some pasting diagram monoids must have nonstandard amalgamation.

\end{document}